\documentclass{amsart}
\usepackage{etex}
\reserveinserts{28}
\usepackage{pict2e}
\usepackage{xypic}
\usepackage{tikz}
\usepackage{MnSymbol}
\usepackage{latexsym}
\usepackage{stmaryrd}
\usepackage[capitalize,nameinlink,noabbrev,nosort]{cleveref}
\def\source{\operatorname{source}}
\def\target{\operatorname{target}}
\def\Rect{\operatorname{Rect}}
\usepackage{pictex}
\usepackage{amsmath,amscd}
\usepackage{afterpage}
\usepackage{youngtab}
\usepackage{psfrag}
\usepackage[boxsize=.3 em]{ytableau}

\usepackage{verbatim}

\usepackage{graphics}
\usepackage{graphicx}
\usepackage{tabularx}

\font\co=lcircle10

\def\jr{\ \smash{\raise2pt\hbox{\co \rlap{\rlap{\char'005} \char'007}}
               \raise6pt\hbox{\rlap{\vrule height6.5pt}}
                \raise2pt\hbox{\rlap{\hskip4pt \vrule
          height0.4pt depth0pt
                width7.7pt}}}}
\def\textcross{\ \smash{\lower4pt\hbox{\rlap{\hskip4.15pt\vrule height14pt}}
                \raise2.8pt\hbox{\rlap{\hskip-3pt \vrule height.4pt depth0pt
                                width14.7pt}}}\hskip12.7pt}

\def\textelbow{\ \hskip.1pt\smash{\raise2.75pt%
                \hbox{\co \hskip 4.15pt\rlap{\rlap{\char'004} \char'006}
                \lower6.8pt\rlap{\vrule height3.5pt}
                \raise3.6pt\rlap{\vrule height3.5pt}}
                \raise2.8pt\hbox{%
                  \rlap{\hskip-7.15pt \vrule height.4pt depth0pt
width3.5pt}%
                  \rlap{\hskip4.05pt \vrule height.4pt depth0pt
width3.5pt}}}
                \hskip8.7pt}

\newtheorem{thm}{Theorem}[section]
\newtheorem{theorem}[thm]{Theorem}
\newtheorem{cor}[thm]{Corollary}
\newtheorem{corollary}[thm]{Corollary}
\newtheorem{conjecture}[thm]{Conjecture}
\newtheorem{lem}[thm]{Lemma}
\newtheorem{lemma}[thm]{Lemma}
\newtheorem{prop}[thm]{Proposition}
\newtheorem{proposition}[thm]{Proposition}
\theoremstyle{definition}
\newtheorem{defn}[thm]{Definition}
\newtheorem{example}[thm]{Example}
\newtheorem{definition}[thm]{Definition}
\newtheorem{rem}[thm]{Remark}  
\newtheorem{remark}[thm]{Remark}  

\numberwithin{equation}{section}

\newcommand{\abs}[1]{\left\vert#1\right\vert}
\newcommand{\set}[1]{\left\{#1\right\}}

\newcommand{\CC}{\mathbb C}

\newcommand{\C}{\mathbb{C}}
\newcommand{\Z}{\mathbb{Z}}

\newcommand{\x}{\times}

\newcommand{\xx}{\widetilde{\mathbf{x}}}
\newcommand{\A}{\mathcal{A}}

\newcommand{\Gl}{\operatorname{GL}}
\newcommand{\Fl}{\operatorname{Fl}}

\def\Fcal{\mathcal{F}}

\def\TT{\mathbb{T}}
\def\Xcal{\mathcal{X}}
\def\Acal{\mathcal{A}}

\newcommand{\overunder}[2]{
\!\begin{array}{c}
\scriptstyle{#1}\\[-.1in]
-\!\!\!-\!\!\!-\\[-.1in]
\scriptstyle{#2}
\end{array}
\!
}

\newcommand{\Le}{\rotatebox[origin=c]{180}{$\Gamma$}} 
\newcommand{\pathne}[1]{L^\nearrow_{#1}} 
\newcommand{\pathsw}[1]{L^\swarrow_{#1}}
\newcommand{\vertne}[1]{V^\nearrow(#1)} 
\newcommand{\vertsw}[1]{V^\swarrow(#1)}
\newcommand{\partne}[1]{\lambda^\nearrow(#1)} 
\newcommand{\partsw}[1]{\lambda^\swarrow(#1)}
\newcommand{\permne}[1]{u^\nearrow_{#1}} 

\newcommand{\ppermsw}[1]{\pi^\swarrow_{#1}}

\begin{document}
\title
{Cluster structures in Schubert varieties in the Grassmannian}
\date{\today}
\author{K. Serhiyenko}%
\address{Department of Mathematics,
            University of California at Berkeley,
            Berkeley, CA
            USA
}%
\email{khrystyna.serhiyenko@berkeley.edu}%
\author{M. Sherman-Bennett}
\address{Department of Mathematics,
            University of California at Berkeley,
            Berkeley, CA
            USA
}%
\email{m\_shermanbennett@berkeley.edu}%
\author{L. Williams}%
\address{Department of Mathematics,
            Harvard University,
            Cambridge, MA
            USA
}%
\email{williams@math.harvard.edu}%


\thanks{}
\date{\today}

\begin{abstract}
In this article we explain how the 
	coordinate ring of each (open) Schubert variety in the Grassmannian can be identified with a 
	cluster algebra, whose combinatorial structure is encoded using 
(target labelings of)
Postnikov's \emph{plabic graphs}.  This result generalizes a theorem of 
 Scott \cite{Scott} for the Grassmannian,
and proves a folklore conjecture for Schubert varieties that has been 
believed by experts since \cite{Scott}, 
though the statement was not formally written down
until M\"uller-Speyer explicitly conjectured it  \cite{MullerSpeyer}.
To prove this conjecture we use a result of Leclerc \cite{Leclerc}, who used the module category of the preprojective algebra  
to prove that coordinate rings of 
many Richardson varieties in the complete flag variety 
can be identified with cluster algebras. 
	Our proof also uses a construction of  Karpman \cite{Karpman} 
to 
build plabic graphs associated to reduced expressions.
We additionally generalize our result to the setting of 
 skew Schubert varieties; 
the latter result 
 uses \emph{generalized} plabic graphs, i.e. plabic graphs
whose boundary vertices need not be labeled in cyclic order.
\end{abstract}

\maketitle
\setcounter{tocdepth}{1}
\tableofcontents

\section{Introduction}\label{sec:intro}

The main goal of this paper is to show that the coordinate ring of 
(the affine cone over)
any (open) Schubert variety
of the Grassmannian (embedded into projective space via the Pl\"ucker embedding)
 can be identified with a cluster algebra, 
 whose combinatorial structure is described explicitly in terms of 
plabic graphs.  
\emph{Cluster algebras} are a class of commutative rings
which were  introduced by Fomin and Zelevinsky \cite{ca1}; 
they are connected to 
many fields of mathematics including Teichm\"uller theory and quiver representations,
and their generators satisfy many nice properties, including the \emph{Laurent phenomenon}
\cite{ca1} 
and \emph{positivity theorem} \cite{LeeSchiffler, GHKK}. 
\emph{Plabic graphs} are certain planar bicolored
graphs which were introduced by Postnikov \cite{Postnikov}; plabic graphs (or rather an 
equivalent object, namely alternating strand diagrams) 
were subsequently used by 
Scott \cite{Scott} to show that the coordinate ring of the affine cone over the 
Grassmannian in its Pl\"ucker embedding \emph{admits a cluster algebra structure}, i.e. it can be identified with a cluster algebra.

There is a natural plabic graph 
generalization of Scott's construction which experts have believed
for some time should give a cluster structure for (open) Schubert varieties 
(and more generally, open
positroid varieties).  
This construction was stated explicitly as a conjecture in 
a recent paper of M\"uller-Speyer \cite{MullerSpeyer0}, who 
 provided some 
evidence in \cite{MullerSpeyer}.  The conjecture 
can be stated roughly as follows.
\begin{conjecture}\label{conj:vague}
	Let $G$ be a reduced plabic graph corresponding to an (open) Schubert (or more generally an open positroid) variety.
	Then the target labeling of the faces of $G$ (which we identify
	with a collection of Pl\"ucker coordinates)
together with the dual quiver of $G$ gives rise to a seed for a cluster
	structure on the coordinate ring of the open Schubert (or positroid) 
	variety.
\end{conjecture}
Meanwhile, Leclerc \cite{Leclerc} constructed a subcategory $\mathcal{C}_{v,w}$ of the module category of the preprojective algebra, that has a cluster structure,
to show that the coordinate ring of 
each open Richardson variety $\mathcal{R}_{v,w}$ of the 
complete flag variety contains a subalgebra which is a cluster algebra; when 
$w$ has a factorization of the form $w=xv$ with 
$\ell(w) = \ell(x)+\ell(v)$,  he showed that 
this subalgebra coincides with the 
coordinate ring.  
Because open Schubert varieties are isomorphic to open Richardson varieties with 
the above property,
Leclerc's result implies that their coordinate rings admit a cluster structure.
However, Leclerc's description of the cluster structure is 
very different from the 
plabic graph description and is far from explicit:
e.g. his cluster quiver is defined in terms of 
morphisms of modules of the preprojective algebra. 

In this paper we prove 
\cref{conj:vague} for Schubert varieties by relating 
Leclerc's cluster structure to the conjectural one coming from plabic graphs.
We also generalize our result to construct cluster structures in skew Schubert varieties; 
 interestingly, these cluster structures for skew Schubert varieties
 depart from the one in 
\cref{conj:vague}, since they 
use \emph{generalized} plabic graphs (with 
boundary vertices which are not longer cyclically labeled).

Once we have proved that the coordinate rings of (open) Schubert and skew Schubert
varieties have cluster structures, we obtain a number of results ``for free''
 from the cluster theory, including the Laurent phenomenon
and the Positivity theorem for cluster variables.
As a consequence of our results we also obtain many combinatorially 
explicit cluster seeds for each (open) Schubert and skew Schubert
variety.  
Note that (open) Schubert varieties provide  examples
of cluster structures of all the finite type simply-laced cluster types ($ADE$), 
see 
\cref{sec:conclusion}.  
Combining our main results with 
\cite[Theorem 3.3]{MullerSpeyer0} and \cite{Muller}, we find that the coordinate rings of 
(open) Schubert and skew Schubert varieties (viewed as cluster algebras) are \emph{locally acyclic},
which implies that each one is finitely generated, normal, locally a complete
intersection, and equal to its own upper cluster algebra.
Combining our result with \cite[Theorem 1.2]{FordSer}, we find that the quivers giving 
rise to the cluster structures for Schubert and skew Schubert varieties admit
\emph{green-to-red sequences}, which by \cite{GHKK} implies that the cluster algebras have
\emph{Enough Global Monomials} and hence each coordinate ring has a canonical basis of 
theta functions, parameterized by the lattice of $g$-vectors.
Finally we obtain applications to the structure of indecomposable summands of cluster-tilting modules in $\mathcal{C}_{v,w}$ 
and the morphisms between them.

\subsection{Notation for the flag variety}

Let $\Gl_n$ denote the general linear group, $B$ the Borel subgroup of 
lower triangular matrices, $B^+$ the opposite Borel subgroup of 
upper triangular matrices, and 
$W=S_n$ the Weyl group (which is in this case the symmetric group on $n$ letters).
$W$ is generated by the simple reflections $s_i$ for $1 \leq i \leq n-1$, where 
$s_i$ is the transposition exchanging $i$ and $i+1$, and it contains 
a longest element, which we denote by $w_0$, with $\ell(w_0) = {n \choose 2}$.
The \emph{complete flag variety} $\Fl_n$ is the homogeneous
space $B\setminus \Gl_n$. 
Concretely, each element $g$ of $\Gl_n$ gives rise to a flag of 
subspaces $\{V_1 \subset V_2 \subset \dots \subset V_n\}$, 
where $V_i$ denotes the span of the top $i$ rows of $g$.
The action of $B$ on the left preserves the flag, so we can identify 
$\Fl_n$ with 
the set of \emph{flags} $\{V_1 \subset V_2 \subset \dots \subset V_n\}$
where $\dim V_i = i$.

Let $\pi:\Gl_n \to \Fl_n$  
denote the natural projection 
$\pi(g):= Bg$.  The Bruhat decomposition 
$$\Gl_n = \bigsqcup_{w\in W} BwB$$
projects to the Schubert decomposition 
$$\Fl_n = \bigsqcup_{w\in W} C_w$$
where $C_w = \pi(BwB)$ is the \emph{Schubert cell} associated to $w$, 
 isomorphic to $\C^{\ell(w)}$.
We also have the Birkhoff decomposition 
$$\Gl_n = \bigsqcup_{w\in W} BwB^+,$$ which projects to the 
{opposite} Schubert decomposition 
$$\Fl_n = \bigsqcup_{w\in W} C^w$$
where $C^w = \pi(B w B^+)$ is the \emph{opposite Schubert cell} 
associated to $w$, isomorphic to $\C^{\ell(w_0) - \ell(w)}$.

The intersection $$\mathcal{R}_{v,w}:=C^v \cap C_w$$
has been considered by Kazhdan and Lusztig \cite{KL} in relation to 
Kazhdan-Lusztig polynomials.  $\mathcal{R}_{v,w}$ is 
nonempty only if $v \leq w$ in the Bruhat order of $W$, 
and it is then a smooth irreducible locally closed subset of $C_w$ 
of dimension $\ell(w)-\ell(v)$.  Sometimes $\mathcal{R}_{v,w}$
is called an \emph{open Richardson variety} \cite{KLS} because its closure 
is a \emph{Richardson variety} \cite{Rich}.
We have a stratification of the complete flag variety
$$\Fl_n = \bigsqcup_{v \leq w} \mathcal{R}_{v,w}.$$



\subsection{Notation for 
 the Grassmannian}

Fix $1 < k < n$.
The parabolic subgroup 
$W_K = \langle s_1,\dots,s_{k-1}\rangle \times \langle s_{k+1}, s_{k+2},\dots,
s_{n-1} \rangle < W$ 
gives rise to a parabolic subgroup $P_K$ in $\Gl_n$, namely
$P_K = \bigsqcup_{w \in W_K} B \dot{w} B$, where $\dot{w}$ is a matrix representative for $w$
in $\Gl_n$.
$W_K$ contains a longest element, which we denote by  $w_K$,
with $\ell(w_K) = {k \choose 2} + {n-k \choose 2}$. 

The \emph{Grassmannian} $Gr_{k,n}$ 
is the homogeneous space 
$P_K\setminus \Gl_n$.
We can think of the 
{Grassmannian} $Gr_{k,n} = 
P_K\setminus \Gl_n$ more concretely as 
the set of all $k$-planes in an $n$-dimensional
vector space $\CC^n$. 
An element of
$Gr_{k,n}$ can be viewed as a full rank 
$k\times n$ matrix of rank $k$, modulo left
multiplication by invertible $k\times k$ matrices.  
That is, two $k\times n$ matrices of rank $k$ represent the same point in $Gr_{k,n}$ if and only if they
can be obtained from each other by invertible row operations.

For integers $a, b$, let $[a, b]$ denote $\{a, a+1, \dots, b-1, b\}$ if $a\leq b$ and the emoty set otherwise. We use the shorthand $[n]:=[1, n]$. Let $\binom{[n]}{k}$ the set of all $k$-element 
subsets of $[n]$. 

Given $V\in Gr_{k,n}$ represented by a $k\times n$ matrix $A$, for $I\in \binom{[n]}{k}$ we let $\Delta_I(V)$ be the maximal minor of $A$ located in the column set $I$. The $\Delta_I(V)$ do not depend on our choice of matrix $A$ (up to simultaneous rescaling by a nonzero constant), and are called the {\itshape Pl\"{u}cker coordinates} of $V$.
The Pl\"ucker coordinates give an embedding of $Gr_{k,n}$ into projective space of dimension
${n \choose k} - 1$.

We have the usual projection $\pi_k$ from the complete flag variety $\Fl_n$ to the 
Grassmannian $Gr_{k,n}$.
Let $W^K = W^K_{\min}$ and 
$W^K_{\max}$ denote the set of minimal- and
maximal-length coset representatives for $W_K \setminus W$; 
we also let
$^{K}W$ (or $^{K}_{\min}W$)
denote the set of 
minimal-length coset representatives for $W/W_K$. Such a permutation $\sigma \in S_n$ is called 
a \emph{Grassmannian permutation} of type $(k,n)$; it has the property that it has at most one descent, and when present, that descent must be in position $k$, i.e.
$\sigma(k)>\sigma(k+1)$.  

Rietsch studied the projections 
of the open Richardson varieties in the complete flag variety to 
partial flag varieties \cite{RietschThesis}.  In particular,
when  $v\in W^K_{\max}$
 (or when $w\in W^K_{\min}$), 
the projection $\pi_k$ is an isomorphism from $\mathcal{R}_{v,w}$ to 
 $\pi_k(\mathcal{R}_{v,w})$.  
We obtain a stratification 
$$Gr_{k,n}=\bigsqcup \pi_k(\mathcal{R}_{v,w})$$
where $(v,w)$ range over all $v\in W^K_{\max}$, $w\in W$, such that 
$v\leq w$.
 Following work of Postnikov \cite{Postnikov, KLS},
 the strata $\pi_k(\mathcal{R}_{v,w})$ are sometimes called 
 \emph{open positroid varieties}, while their closures
 are called \emph{positroid varieties}.

It follows from the definitions (see e.g. \cite[Section 6]{KLS}) that 
positroid varieties include Schubert and opposite Schubert varieties
in the Grassmannians, which we now define.

\begin{definition}
Let $I$ denote a $k$-element subset of $[n]$.
The \emph{Schubert cell} $\Omega_{I}$  is defined to be 
$$\Omega_{I} = \{A \in Gr_{k,n}\ \vert \ \text{the lexicographically minimal nonvanishing
Plucker coordinate of }A\text{ is }\Delta_{I}(A) \}.$$
	The \emph{Schubert variety} $X_{I}$ is defined to be the closure
	$\overline{\Omega_I}$ of $\Omega_I$.

	Meanwhile the \emph{opposite Schubert cell} $\Omega^{I}$  is defined to be 
$$\Omega^{I} = \{A \in Gr_{k,n}\ \vert \ \text{the lexicographically maximal 
nonvanishing
Plucker coordinate of }A\text{ is }\Delta_{I}(A) \}.$$
	The \emph{opposite Schubert variety} $X^{I}$ is defined to be the 
	closure $\overline{\Omega^I}$ of $\Omega^I$.
\end{definition}

It's easy to see 
that elements $v$ of $W^K_{\max}$ and elements $w$ of $W^K_{\min}$
are also in bijection with $k$-element subsets of $[n]$, which we
denote by $I(v)$ and $I(w)$, respectively.
When $w\in W^K_{\min}$, 
$\overline{\pi_k(\mathcal{R}_{e,w})}$  
is isomorphic to the  
\emph{opposite Schubert variety} $X^{I(w)}$ in the Grassmannian, which has dimension 
$\ell(w)$.  We therefore refer to 
${\pi_k(\mathcal{R}_{e,w})}$ as an \emph{open opposite Schubert variety}.
Similarly, when $v\in W^K_{\max}$,
$\overline{\pi_k(\mathcal{R}_{v,w_0})}$ is isomorphic to the 
\emph{Schubert variety} $X_{I(v)}$ in the Grassmannian, 
which has dimension 
$\ell(w_0)-\ell(v)$. 
We refer to ${\pi_k(\mathcal{R}_{v,w_0})}$ as an \emph{open Schubert variety}.
More generally, if $v\in W^K_{max}$ and $w\in W$ 
has a factorization of the form $w=xv$ which is \emph{length-additive}, i.e. where 
$\ell(w) = \ell(x)+\ell(v)$, then we refer to 
$\overline{\pi_k(\mathcal{R}_{v,w})}$ 
(respectively, ${\pi_k(\mathcal{R}_{v,w})}$)
as a \emph{skew Schubert variety} (respectively, open skew Schubert variety).
See \cref{sec:whichpositroids} for more discussion of skew Schubert varieties, 
including some justification for the terminology.


Let $\lambda$
 denote a Young diagram contained in a $k \times (n-k)$ rectangle.  
We can identify $\lambda$ with the lattice path $\pathsw{\lambda}$ in the rectangle taking steps west 
and south from the northeast corner of the rectangle to the southwest corner (where the $\swarrow$ indicates that the path``goes southwest").
If we label the steps of the lattice path from $1$ to $n$, then the labels of the 
south steps give a $k$-element subset of $[n]$ that we denote by $\vertsw{\lambda}$ (the ``vertical steps" of $\lambda$).
Conversely, each $k$-element subset $I$ of $[n]$ can be identified with  a 
Young diagram, which we denote by 
$\partsw{I}$. 
Since this gives a bijection between 
Young diagrams contained in a $k \times (n-k)$ rectangle 
and $k$-element subsets  of $[n]$, we also index Schubert and opposite 
Schubert cells and varieties by Young diagrams, denoting them 
$\Omega_{\lambda}$, $\Omega^{\lambda}$, 
$X_{\lambda}$, and $X^{\lambda}$, 
respectively.  The open Schubert and opposite Schubert varieties are denoted by 
$X_{\lambda}^\circ$, and $(X^{\lambda})^\circ$.
The dimension of $\Omega_{\lambda}$, $X_{\lambda}$, and $X_{\lambda}^{\circ}$ is $|\lambda|$, the number of boxes of 
$\lambda$, while the codimension of $\Omega^{\lambda}$, $X^{\lambda}$, and $(X^{\lambda})^\circ$ is $|\lambda|$.

\begin{remark}
Throughout this paper we will be primarily working with open Schubert (and skew-Schubert) varieties.  
The reader should be cautioned that 
we will mostly drop the adjective ``open'' from now on but will try to consistently use the notation 
$X_{\lambda}^\circ$ for clarity.
\end{remark}

We also associate with a Young diagram $\lambda$  
the \emph{Grassmannian permutation}
$\ppermsw{\lambda}$ of type $(n-k,n)$: 
in list notation, this permutation is obtained by first reading 
the labels of the horizontal steps of $\pathsw{\lambda}$, and then reading the labels of the 
vertical steps of $\pathsw{\lambda}$.
(Moreover any fixed points in positions $1,2,\dots, n-k$ are ``black"
and any fixed points in positions $n-k+1,\dots, n$ are ``white.")
Note that $\ell(\ppermsw{\lambda}) = |\lambda|$.

\begin{remark} ``Going northeast" along the lattice path $\pathsw{\lambda}$ gives rise to analogous bijections between Young diagrams in a $k \times (n-k)$ rectangle, $k$-element subsets of $n$, and Grassmannian permutations of type $(k, n)$. So we can define all the notations that we did before, 
	switching each $\swarrow$ to a $\nearrow$.  
	So a Young diagram $\lambda$ is identified with the lattice path $\pathne{\lambda}$ in the rectangle taking steps east and north from the southwest corner of the rectangle to the northeast corner. If we label the path with $1$ to $n$, the labels of the north steps give the $k$-element subset $\vertne{\lambda}.$ 
Similarly we define $\partne{I}$. 
\end{remark}


\subsection{The main result}

We now state the main result.  Note that 
the definitions of plabic graph and trip permutation can be found in 
\cref{sec:background}.


\begin{theorem}\label{thm:main}
Consider the Schubert variety $X_{\lambda}^\circ$ of $Gr_{k,n}$. 
Let $G$ be a reduced plabic graph (with boundary vertices
labeled clockwise from $1$ to $n$) with trip permutation 
$\ppermsw{\lambda}$.
Construct the dual quiver of $G$ and label its vertices by the Pl\"ucker coordinates
given by the target labeling of $G$, so as
to obtain a labeled seed 
$\Sigma^{\target}_G$ (see \cref{def:faces}, \cref{fig:plabic}, \cref{def:graphseed}).
Then the coordinate ring 
$\CC[\hat{X_{\lambda}^\circ}]$ 
of the (affine cone over) $X_{\lambda}^\circ$ coincides with the 
cluster algebra
$\mathcal{A}
(\Sigma^{\target}_G)$.
\end{theorem}


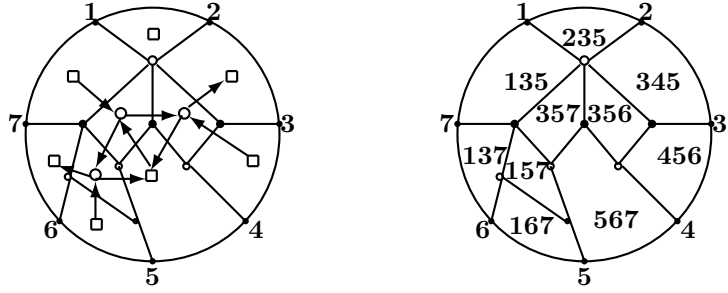
\begin{figure}[h]
\begin{center}
\setlength{\unitlength}{.8pt}
\begin{picture}(100,120)(-70,-70)
\thicklines
\put(15,30){\circle{4}}
\put(-18,0){\circle*{4}}
\put(15,0){\circle*{4}}
\put(47,0){\circle*{4}}
\put(13,30){\line(-31,23){25}}
\put(17,30){\line(31,23){25}}
\put(13,28){\line(-25,-23){30}}
\put(17,28){\line(25,-23){30}}
\put(15,28){\line(0,-1){29}}
\put(-18,0){\line(-1,0){28}}
\put(-18,0){\line(1,-1.1){17}}
\put(47,0){\line(-1,-1.2){16}}
\put(47,0){\line(1,0){28}}
\put(15,-5){\circle{120}}
\put(15,-65){\circle*{3}}
\put(-1,-20){\circle{3}}
\put(-12,48){\circle*{3}}
\put(31,-20){\circle{3}}
\put(42,48){\circle*{3}}
\put(-45,0){\circle*{3}}
\put(75,0){\circle*{3}}
\put(-25,-25){\circle{3}}
	\put(7,-46){\circle*{3}}
\put(-29,-46){\circle*{3}}
\put(59,-46){\circle*{3}}
\put(-29,-46){\line(5,20){5}}
	\put(-24,-25){\line(5,20){6.5}}
	\put(-24,-25){\line(10,-7){30}}
	\put(-1,-20){\line(5,-14){16}}
	\put(15,0){\line(-10,-12){16}}
	\put(15,0){\line(10,-12){16}}
\put(59,-46){\line(-20,20){26}}         
\put(-14,53){\makebox(0,0){$\mathbf{1}$}}
\put(44,53){\makebox(0,0){$\mathbf{2}$}}
\put(79,0){\makebox(0,0){$\mathbf{3}$}}
\put(64,-50){\makebox(0,0){$\mathbf{4}$}}
\put(15,-72){\makebox(0,0){$\mathbf{5}$}}
\put(-33,-50){\makebox(0,0){$\mathbf{6}$}}
\put(-50,0){\makebox(0,0){$\mathbf{7}$}}

\put(13,40){\line(0,1){5}}
\put(13,40){\line(1,0){5}}
\put(13,45){\line(1,0){5}}
\put(18,40){\line(0,1){5}}

\put(50,20){\line(0,1){5}}
\put(50,20){\line(1,0){5}}
\put(50,25){\line(1,0){5}}
\put(55,20){\line(0,1){5}}

\put(60,-20){\line(0,1){5}}
\put(60,-20){\line(1,0){5}}
\put(60,-15){\line(1,0){5}}
\put(65,-20){\line(0,1){5}}

\put(-34,-20){\line(0,1){5}}
\put(-34,-20){\line(1,0){5}}
\put(-34,-15){\line(1,0){5}}
\put(-29,-20){\line(0,1){5}}

\put(-14,-50){\line(0,1){5}}
\put(-14,-50){\line(1,0){5}}
\put(-14,-45){\line(1,0){5}}
\put(-9,-50){\line(0,1){5}}

\put(12,-27){\line(0,1){5}}
\put(12,-27){\line(1,0){5}}
\put(12,-22){\line(1,0){5}}
\put(17,-27){\line(0,1){5}}

\put(-25,20){\line(0,1){5}}
\put(-25,20){\line(1,0){5}}
\put(-25,25){\line(1,0){5}}
\put(-20,20){\line(0,1){5}}

\put(30,5){\circle{5}}
\put(0,5){\circle{5}}
\put(-12,-24){\circle{5}}
\put(32,8){{\vector(4,3){17}}}
\put(29,4){{\vector(-4,-7){15}}}
\put(2,4){{\vector(1,0){25}}}
\put(0,3){{\vector(-5,-10){12}}}
\put(-12,-26){{\vector(1,0){24}}}
\put(-14,-25){{\vector(-3,1){16}}}
\put(-12,-45){{\vector(0,1){18}}}
\put(14,-20){{\vector(-2,3){15}}}
\put(-20,20){{\vector(15,-13){17}}}
\put(60,-15){{\vector(-3,2){28}}}
\end{picture}
\qquad \qquad
\qquad \qquad
\begin{picture}(100,120)(-70,-70)
\thicklines
\put(15,30){\circle{4}}
\put(-18,0){\circle*{4}}
\put(15,0){\circle*{4}}
\put(47,0){\circle*{4}}
\put(13,30){\line(-31,23){25}}
\put(17,30){\line(31,23){25}}
\put(13,28){\line(-25,-23){30}}
\put(17,28){\line(25,-23){30}}
\put(15,28){\line(0,-1){29}}
\put(-18,0){\line(-1,0){28}}
\put(-18,0){\line(1,-1.1){17}}
\put(47,0){\line(-1,-1.2){16}}
\put(47,0){\line(1,0){28}}
\put(15,-5){\circle{120}}
\put(15,-65){\circle*{3}}
\put(-1,-20){\circle{3}}
\put(-12,48){\circle*{3}}
\put(31,-20){\circle{3}}
\put(42,48){\circle*{3}}
\put(-45,0){\circle*{3}}
\put(75,0){\circle*{3}}
\put(-25,-25){\circle{3}}
	\put(7,-46){\circle*{3}}
\put(-29,-46){\circle*{3}}
\put(59,-46){\circle*{3}}
\put(-29,-46){\line(5,20){5}}
	\put(-24,-25){\line(5,20){6.5}}
	\put(-24,-25){\line(10,-7){30}}
	\put(-1,-20){\line(5,-14){16}}
	\put(15,0){\line(-10,-12){16}}
	\put(15,0){\line(10,-12){16}}
\put(59,-46){\line(-20,20){26}}         
\put(-14,53){\makebox(0,0){$\mathbf{1}$}}
\put(44,53){\makebox(0,0){$\mathbf{2}$}}
\put(79,0){\makebox(0,0){$\mathbf{3}$}}
\put(64,-50){\makebox(0,0){$\mathbf{4}$}}
\put(15,-72){\makebox(0,0){$\mathbf{5}$}}
\put(-33,-50){\makebox(0,0){$\mathbf{6}$}}
\put(-50,0){\makebox(0,0){$\mathbf{7}$}}

\put(15,40){\makebox(0,2){$\mathbf{235}$}}
\put(15,15){\makebox(24,-20){$\mathbf{356}$}}
\put(15,15){\makebox(-25,-20){$\mathbf{357}$}}
\put(0,30){\makebox(-25,-20){$\mathbf{135}$}}
\put(50,20){\makebox(0,0){$\mathbf{345}$}}
\put(60,-15){\makebox(0,0){$\mathbf{456}$}}
\put(30,-45){\makebox(0,0){$\mathbf{567}$}}
\put(-10,-48){\makebox(0,0){$\mathbf{167}$}}
\put(-32,-15){\makebox(0,0){$\mathbf{137}$}}
\put(-12,-22){\makebox(0,0){$\mathbf{157}$}}
\end{picture}
\end{center}
	\caption{A plabic graph $G$ for $Gr_{3,7}$ with trip permutation 
	$\ppermsw{\lambda}=(2,4,6,7,1,3,5)$, for $\lambda=(4, 3, 2)$, together with the dual quiver of $G$ and the face labeling given by  target labels.
	The associated Le-diagram is a Young diagram of shape $\lambda$ 
	which is filled with $+$'s.}
\label{fig:plabic}
\end{figure}

\cref{thm:main} can be rephrased as follows:
\begin{itemize}
	\item Each of the (in general, infinitely many) cluster variables in 
$\mathcal{A}
(\Sigma^{\target}_G)$
		is a regular function on 
		$\hat{X_{\lambda}^\circ}$. 
	\item The cluster variables in 
$\mathcal{A}
(\Sigma^{\target}_G)$
		generate the ring 
		$\CC[\hat{X_{\lambda}^\circ}]$ 
		of regular functions on 
		$\hat{X_{\lambda}^\circ}$. 
\end{itemize}

We actually prove something a bit more general than \cref{thm:main};
we prove the following.

\begin{theorem}\label{thm:main2}
Consider the skew Schubert variety 
	${\pi_k(\mathcal{R}_{v,w})}$,
where $v\in W^K_{max}$ and 
$w$ has a length-additive factorization $w=xv$.
Let $G$ be a reduced plabic graph
(with boundary vertices labeled clockwise from $1$ to $n$) with 
trip permutation $vw^{-1} = x^{-1}$, and such that boundary lollipops
are white if and only if they are in $[k]$. 
	Apply $v^{-1}$ to the boundary vertices of $G$, obtaining
	the relabeled graph $v^{-1}(G)$, and apply the target labeling to 
	obtain the labeled seed
	$\Sigma^{\target}_{v^{-1}(G)}$. 
Then the coordinate
ring 	
	$\CC[\widehat{\pi_k(\mathcal{R}_{v,w})}]$
	of the (affine cone over) the 
skew Schubert variety 
	${\pi_k(\mathcal{R}_{v,w})}$
coincides with the cluster algebra 
$\mathcal{A}
	(\Sigma^{\target}_{v^{-1}(G)})$.
\end{theorem}

In the case of Schubert varieties, \cref{thm:main} resolves
\cref{conj:vague}, which has been believed to be true by experts for some time,
though it wasn't written down explicitly as a conjecture
until recently, see \cite[Conjecture 3.4]{MullerSpeyer0}.  Note that 
there is another version of the conjecture which uses the \emph{source
labeling} of $G$ instead of the 
target labeling \cite[Remark 3.5]{MullerSpeyer0}.
Both conjectures make sense more generally for positroid varieties
and arbitrary reduced plabic graphs (whose trip permutations can be
arbitrary decorated permutations).
However, the cluster structure that we give in \cref{thm:main2}
is different from either of the cluster structures proposed in 
\cite{MullerSpeyer0}.

Our strategy of proof is to find, for each skew Schubert variety, 
one distinguished seed coming from 
Leclerc's cluster structure, which we can describe completely explicitly.
 We then show that this seed agrees with 
a corresponding seed coming from the combinatorial construction  of
\cref{thm:main2}, and justify that mutations in both cluster structures agree.
We use a (modification) of a construction of 
Karpman \cite{Karpman} as a key tool in the proof.





\begin{remark}
In his thesis \cite{Chevalier}, Chevalier describes a cluster-tilting
object associated to Richardson varieties 
$\mathcal{R}_{v,w}$ where $v=w_K$ and $w \geq v$ in Bruhat order.
These Richardson varieties correspond to positroid
varieties in $Gr_{k,n}$ whose $\Le$-diagrams have shape
	$k \times (n-k)$.  (This case is somewhat complementary to the cases
that we consider in this paper, in the sense that Chevalier 
treats $\Le$-diagrams of shape $k \times (n-k)$ with 
arbitrary fillings, while on the other hand Schubert 
varieties correspond to $\Le$-diagrams of arbitrary shape
	whose boxes are all filled with $+$'s.)
	In the case of 
	the big open Schubert 
	variety in the Grassmannian (i.e. the positroid whose
	$\Le$-diagram is a $k \times (n-k)$ rectangle filled with 
	all $+$'s)  we get the same quiver as 
Chevalier does, but 
 different modules (and hence
different Pl\"ucker coordinates).  And in other cases 
	of overlap (i.e. skew-Schubert varieties with $v = w_K$)
	even our quivers are different from Chevalier's.
\end{remark}

\subsection{Outline of the paper}

Our paper is structured as follows.
In \cref{sec:background}, we give background on cluster structures, plabic graphs, 
and reduced expressions.
While each skew Schubert variety 
${\pi_k(\mathcal{R}_{v,w})}$
(where 
$v=w_K v'\in W^K_{max}$ 
and $w\in W$ has a length-additive factorization 
$\mathbf{w}=\mathbf{xv}=\mathbf{x w_K v'}$ into reduced expressions for $x$, $w_K$, and $v'$)
corresponds to an equivalence class of plabic graphs 
(more generally to a collection of cluster seeds), there is 
one among them which is particularly nice, which we call the \emph{rectangles seed}. 
In \cref{sec:seed}, we give an explicit description of the rectangles seed for a skew Schubert variety
${\pi_k(\mathcal{R}_{v,w})}$ as above,
 together with its dual cluster quiver.
In \cref{sec:Karpman} we describe a construction of Karpman \cite{Karpman} which produces
a bridge-decomposable plabic graph associated to a pair $(y,\mathbf{z})$, where 
$y^{-1} \in W^K_{\max}$, $\mathbf{z}$ is a reduced decomposition for $z$, and $y\leq z$.
And we show that if we perform her construction for the pair 
$(w_K, \mathbf{x w_K})$ and then relabel boundary vertices of the resulting plabic graph $G$ by $v^{-1}$, 
the target labeling of the dual quiver of $G$ gives rise to 
 the rectangles seed 
for ${\pi_k(\mathcal{R}_{v,w})}$. 
 In \cref{sec:Leclerc} we describe a construction of Leclerc \cite{Leclerc}, which produces a cluster seed
 associated to each pair $(v,\mathbf{w})$, where $v\in W^K_{max}$ and $v \leq w$.
 We also prove that for the choice $(v, \mathbf{w} = \mathbf{x w_K v'})$,  
Leclerc's construction gives rise to the rectangles seed.
In \cref{sec:proofs}, we build on the results of the previous sections to 
prove \cref{thm:main2} and then deduce \cref{thm:main} from it.  
 \cref{sec:conclusion}
gives applications of \cref{thm:main} and \cref{thm:main2}, 
and characterizes for which Schubert varieties the 
cluster structure is of finite type.
In \cref{sec:whichpositroids}, 
we give a concrete description of skew Schubert varieties.
And in \cref{sec:nonrealizable}, we give an example showing that outside of the skew-Schubert case, 
the cluster subalgebra 
of the coordinate
ring of ${\pi_k(\mathcal{R}_{v,w})}$ coming from Leclerc's construction is in general impossible to realize from a plabic graph.

\vskip .2cm

\noindent{\bf Acknowledgements:~}
We are grateful to Bernard Leclerc for numerous helpful discussions, and 
for bringing the work of Chevalier \cite{Chevalier} to our attention.
K.S. acknowledges support from the National Science Foundation Postdoctoral Fellowship MSPRF-1502881.
M.S.B acknowledges support by an NSF Graduate Research Fellowship
No. DGE-1752814. 
L. W. was partially supported by 
the NSF grant DMS-1600447.
Any opinions, findings
and conclusions or recommendations expressed in this material are those of 
the authors and do not necessarily reflect the views of the National
Science Foundation.

\section{Background on cluster structures and plabic graphs}\label{sec:background}




\subsection{Background on cluster structures}
Cluster algebras are a class of rings with a particular 
combinatorial structure; they were introduced by Fomin and Zelevinsky in \cite{ca1}.

\begin{definition}[Quiver]\label{quiver}
A \emph{quiver} $Q$ is a directed graph; we will assume that $Q$ has no 
loops or $2$-cycles.
Each vertex is designated either  \emph{mutable} or \emph{frozen}.
\end{definition}

\begin{definition}[Quiver Mutation]\label{def:mutation}
Let $q$ be a mutable vertex of quiver $Q$.  The quiver mutation 
$\mu_q$ transforms $Q$ into a new quiver $Q' = \mu_q(Q)$ via a sequence of three steps:
\begin{enumerate}
\item For each oriented two path $r \to q \to s$, add a new arrow $r \to s$
(unless $r$ and $s$ are both frozen, in which case do nothing).
\item Reverse the direction of all arrows incident to the vertex $q$.
\item Repeatedly remove oriented $2$-cycles until unable to do so.
\end{enumerate}
\end{definition}

We say that two quivers $Q$ and $Q'$ are \emph{mutation equivalent} if $Q$
can be transformed into a quiver isomorphic to $Q'$ by a sequence of mutations.

\begin{definition}
[\emph{Labeled seeds}]
\label{def:seed0}
Choose $M\geq N$ positive integers.
Let $\Fcal$ be an \emph{ambient field}
of rational functions
in $N$ independent
variables
over
$\CC(x_{N+1},\dots,x_M)$.
A \emph{labeled seed} in~$\Fcal$ is
a pair $(\xx, Q)$, where
\begin{itemize}
\item
$\xx = (x_1, \dots, x_M)$ forms a free generating
set for
$\Fcal$,
and
\item
$Q$ is a quiver on vertices
$1, 2, \dots,N, N+1, \dots, M$,
whose vertices $1,2, \dots, N$ are
\emph{mutable}, and whose vertices $N+1,\dots, M$ are \emph{frozen}.
\end{itemize}
We refer to~$\xx$ as the (labeled)
\emph{extended cluster} of a labeled seed $(\xx, Q)$.
The variables $\{x_1,\dots,x_N\}$ are called \emph{cluster
variables}, and the variables $c=\{x_{N+1},\dots,x_M\}$ are called
\emph{frozen} or \emph{coefficient variables}.
We often view the labeled seed as a quiver $Q$ where each vertex $i$ is labeled by 
the corresponding variable $x_i$.
\end{definition}

\begin{definition}
[\emph{Seed mutations}]
\label{def:seed-mutation0}
Let $(\xx, Q)$ be a labeled seed in $\Fcal$,
and let $q \in \{1,\dots,N\}$.
The \emph{seed mutation} $\mu_q$ in direction~$q$ transforms
$(\xx, Q)$ into the labeled seed
$\mu_q(\xx,  Q)=(\xx', \mu_q(Q))$, where the cluster
$\xx'=(x'_1,\dots,x'_M)$ is defined as follows:
$x_j'=x_j$ for $j\neq q$,
whereas $x'_q \in \Fcal$ is determined
by the \emph{exchange relation}
\begin{equation}
\label{exchange relation0}
x'_q\ x_q = 
\prod_{q \to r} x_r + \prod_{s \to q} x_s,
\end{equation}
where the first product is over all arrows $q \to r$ in $Q$
which start at $q$, and the second product is over all arrows $s\to q$ which 
end at $q$.
\end{definition}

\begin{remark}
It is not hard to check that seed mutation is an involution.
\end{remark}

\begin{remark}
Note that arrows between two frozen vertices of a quiver do not
affect seed mutation (they do not affect the mutated quiver
or the exchange relation).  For that reason, one may omit
arrows between two frozen vertices. 
\end{remark}

\begin{definition}
[\emph{Patterns}]
\label{def:patterns0}
Consider the \emph{$N$-regular tree}~$\TT_N$
whose edges are labeled by the numbers $1, \dots, N$,
so that the $N$ edges emanating from each vertex receive
different labels.
A \emph{cluster pattern}  is an assignment
of a labeled seed $\Sigma_t=(\xx_t, Q_t)$
to every vertex $t \in \TT_N$, such that the seeds assigned to the
endpoints of any edge $t \overunder{q}{} t'$ are obtained from each
other by the seed mutation in direction~$q$.
The components of
$\xx_t$ are written as $\xx_t = (x_{1;t}\,,\dots,x_{N;t}).$
\end{definition}

Clearly, a cluster pattern  is uniquely determined
by an arbitrary  seed.

\begin{definition}
[\emph{Cluster algebra}]
\label{def:cluster-algebra0}
Given a cluster pattern, we denote
\begin{equation}
\label{eq:cluster-variables0}
\Xcal
= \bigcup_{t \in \TT_N} \xx_t
= \{ x_{i,t}\,:\, t \in \TT_N\,,\ 1\leq i\leq N \} \ ,
\end{equation}
the union of clusters of all the seeds in the pattern.
The elements $x_{i,t}\in \Xcal$ are called \emph{cluster variables}.
The
\emph{cluster algebra} $\Acal$ associated with a
given pattern is the $\CC[x_{N+1}^{\pm1}, \dots, x_{M}^{\pm 1}]$-subalgebra of the ambient field $\Fcal$
generated by all cluster variables: $\Acal = \CC[c^{\pm 1}] [\Xcal]$.
We denote $\Acal = \Acal(\xx,  Q)$, where
$(\xx,Q)$
is any seed in the underlying cluster pattern.
In this generality,
$\Acal$ is called a \emph{cluster algebra from a quiver}, or a
\emph{skew-symmetric cluster algebra of geometric type.}
We say that $\Acal$ has \emph{rank $N$} because each cluster contains
$N$ cluster variables.
\end{definition}

\begin{remark}
One of the earliest definitions of cluster algebra defined 
it as  $\Acal = \CC[c] [\Xcal]$ instead of 
 $\Acal = \CC[c^{\pm 1}] [\Xcal]$.
 This is the definition 
Scott worked with in proving that the coordinate ring of the Grassmannian is a cluster algebra \cite{Scott}.  If one uses
\cref{def:cluster-algebra0} instead, then the statement is that the coordinate ring of the open Schubert variety
in the Grassmannian is a cluster algebra.  In fact the latter statement was 
verified in \cite[Section 3.3]{GSV03}, who exhibited an initial 
quiver which is the one from the \emph{rectangles seed} we discuss
in \cref{sec:seed}.
\end{remark}

\subsection{Background on plabic graphs}

In this section 
we review Postnikov's notion of \emph{plabic graphs} \cite{Postnikov}, which 
we will then use to define cluster structures in Schubert varieties.

\begin{definition}
A {\it plabic (or planar bicolored) graph\/}
is an undirected graph $G$ drawn inside a disk
(considered modulo homotopy)
with $n$ {\it boundary vertices\/} on the boundary of the disk,
labeled $1,\dots,n$ in clockwise order, as well as some
colored {\it internal vertices\/}.
These internal vertices
are strictly inside the disk and are
colored in black and white. 
An internal vertex of degree one adjacent to a boundary vertex is a \emph{lollipop}.
We will always assume that no vertices of the same color are adjacent, and that 
each boundary vertex $i$ is adjacent to a single internal vertex.

\end{definition}

See Figure \ref{G25} for an example of a plabic graph.
\begin{figure}[h]
\centering
\includegraphics[height=1in]{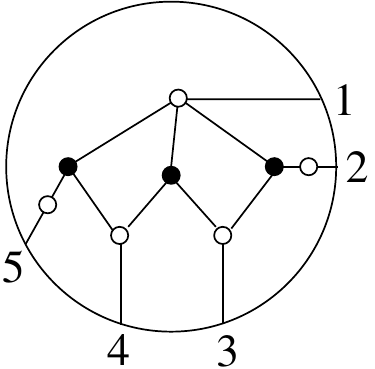}
\caption{A plabic graph}
\label{G25}
\end{figure}

\begin{defn} A \emph{generalized plabic graph} is a plabic graph with boundary vertices are labeled by $1, \dots, n$ in some order, not necessarily clockwise.
\end{defn}

Generalized plabic graphs naturally arise in the course of our arguments. Though we state all of the following definitions for plabic graphs for clarity, they can equally be made for generalized plabic graphs.

There is a natural set of local transformations (moves) of plabic graphs, which we now describe.
Note that we will always assume that a plabic graph $G$ has no isolated 
components (i.e. every connected component contains at least
one boundary vertex).  We will also assume that $G$ is \emph{leafless}, 
i.e.\ if $G$ has an 
internal vertex of degree $1$, then that vertex must be adjacent to a boundary
vertex.

(M1) SQUARE MOVE (Urban renewal).  If a plabic graph has a square formed by
four trivalent vertices whose colors alternate,
then we can switch the
colors of these four vertices (and add some degree $2$ vertices to preserve
the bipartiteness of the graph).

(M2) CONTRACTING/EXPANDING A VERTEX.
Any degree $2$ internal vertex not adjacent to the boundary can be deleted,
and the two adjacent vertices merged.
This operation can also be reversed.  Note that this operation can always be used
to change an arbitrary
 square face of $G$ into a square face whose four vertices are all trivalent.

(M3) MIDDLE VERTEX INSERTION/REMOVAL.
We can always remove or add degree $2$ vertices at will, subject to the 
  condition that the graph remains bipartite.

See \cref{M1} for depictions of these three moves.

\begin{figure}[h]
\centering
\includegraphics[height=.5in]{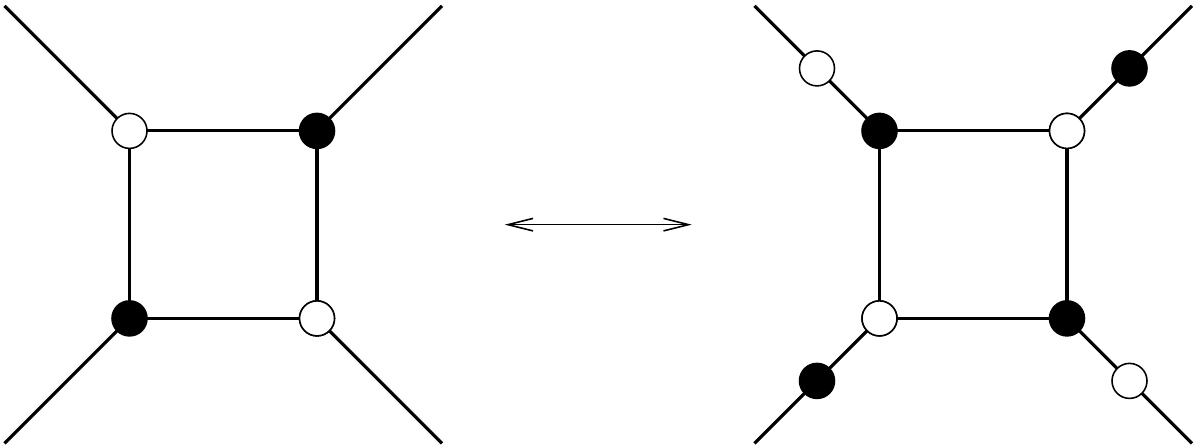}
\hspace{.5in}
\raisebox{6pt}{\includegraphics[height=.4in]{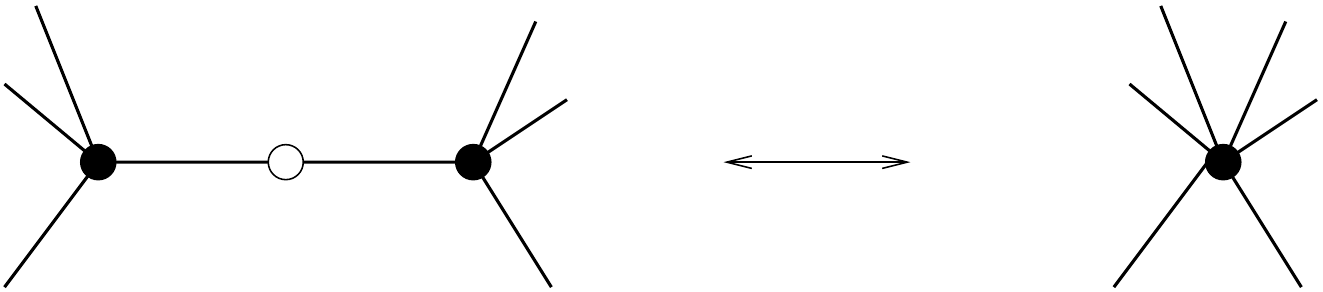}}
\hspace{.5in}
\raisebox{16pt}{\includegraphics[height=.07in]{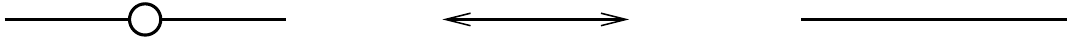}}
\caption{%
A square move, an edge 
contraction/expansion, and a vertex insertion/removal.}
\label{M1}
\end{figure}

(R1) PARALLEL EDGE REDUCTION.  If a plabic graph contains
two trivalent vertices of different colors connected
by a pair of parallel edges, then we can remove these
vertices and edges, and glue the remaining pair of edges together.

\begin{figure}[h]
\centering
\includegraphics[height=.25in]{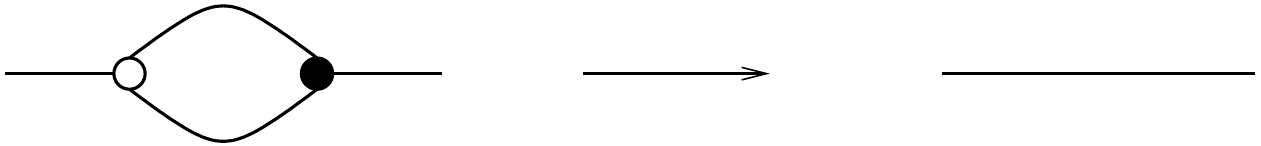}
\caption{Parallel edge reduction}
\label{R1}
\end{figure}

\begin{definition}
Two plabic graphs are called \emph{move-equivalent} if they can be obtained
from each other by moves (M1)-(M3).  The \emph{move-equivalence class}
of a given plabic graph $G$ is the set of all plabic graphs which are move-equivalent
to $G$.
A leafless plabic graph without isolated components
is called \emph{reduced} if there is no graph in its move-equivalence
class to which we can apply (R1).
\end{definition}

\begin{definition} A \emph{decorated permutation} $\pi^:$ is a permutation $\pi \in S_n$ together with a coloring $\set{i \ \vert \ \pi(i)=i} \to \{\text{black, white}\}$. 
\end{definition}

\begin{definition}\label{def:rules}
Given a reduced plabic graph $G$,
a  \emph{trip} $T$ is a directed path which starts at
some boundary vertex
$i$, and follows the ``rules of the road": it turns (maximally) right at a
black vertex,  and (maximally) left at a white vertex.
 Note that $T$ will also
end at a boundary vertex $j$; we then refer to this trip as
$T_{i \to j}$. Setting $\pi(i)=j$ for each such trip,
 we associate a (decorated) \emph{trip permutation}
$\pi_G=(\pi(1),\dots,\pi(n))$ to each reduced plabic graph $G$, where a fixed point $\pi(i)=i$ is colored white (black) if there is a white (black) lollipop at boundary vertex $i$.
We say that $G$ has \emph{type $\pi_G$}.
\end{definition}

As an example, the trip permutation associated to the
reduced plabic graph in Figure \ref{G25} is $(3,4,5,1,2)$.

\begin{remark}\label{rem:moves}
Note that the trip permutation of a plabic graph is preserved 
by the local moves (M1)-(M3), but not by (R1). For reduced plabic graphs the converse holds, namely 
it follows from \cite[Theorem 13.4]{Postnikov} 
that any two reduced plabic graphs with the same trip permutation are 
move-equivalent.
\end{remark}



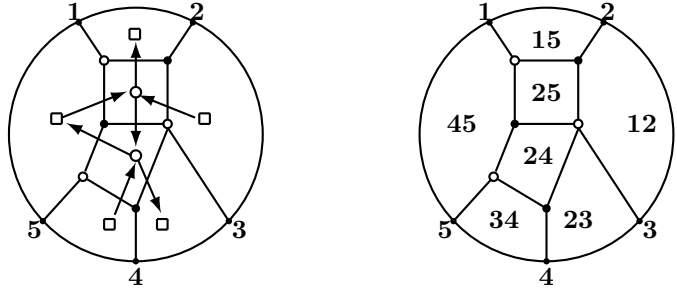
\begin{figure}[h]
\begin{center}
\setlength{\unitlength}{.8pt}
\begin{picture}(100,120)(-30,-70)
\thicklines
\multiput(1,0)(1,30){2}{\line(1,0){27.5}}
\multiput(0,1)(30,1){2}{\line(0,1){27.5}}

\put(15,15){\circle{5}}
\put(15,-15){\circle{5}}

\put(12,40){\line(0,1){5}}
\put(12,40){\line(1,0){5}}
\put(12,45){\line(1,0){5}}
\put(17,40){\line(0,1){5}}

\put(45,0){\line(0,1){5}}
\put(45,0){\line(1,0){5}}
\put(45,5){\line(1,0){5}}
\put(50,0){\line(0,1){5}}

\put(-25,0){\line(0,1){5}}
\put(-25,0){\line(1,0){5}}
\put(-25,5){\line(1,0){5}}
\put(-20,0){\line(0,1){5}}

\put(25,-50){\line(0,1){5}}
\put(25,-50){\line(1,0){5}}
\put(25,-45){\line(1,0){5}}
\put(30,-50){\line(0,1){5}}

\put(0,-50){\line(0,1){5}}
\put(0,-50){\line(1,0){5}}
\put(0,-45){\line(1,0){5}}
\put(5,-50){\line(0,1){5}}

\put(15,17){{\vector(0,1){22}}}
\put(15,13){{\vector(0,-1){24}}}
\put(-20,3){{\vector(2,.8){31}}}
\put(43,3){{\vector(-2,.8){27}}}
\put(13,-15){{\vector(-4,2){31}}}
\put(16,-17){{\vector(1,-2.3){11}}}
\put(5,-43){{\vector(1,2.5){10}}}

\put(0,30){\circle{4}}
\put(30,30){\circle*{4}}
\put(30,0){\circle{4}}
\put(0,0){\circle*{4}}
\put(-10,-25){\circle{4}}
\put(15,-40){\circle*{4}}
\put(0,0){\line(-10,-25){9}}
\put(15,-40){\line(-25,15){23}}
\put(15,-40){\line(20,50){15}}        
\put(15,-5){\circle{120}}
\put(15,-65){\circle*{3}}
\put(15,-40){\line(0,-1){24}}
\put(-12,48){\circle*{3}}
\put(42,48){\circle*{3}}
\put(-12,48){\line(12,-18){11}}
\put(42,48){\line(-12,-18){11}}
\put(-29,-46){\circle*{3}}
\put(59,-46){\circle*{3}}
\put(-29,-46){\line(19,21){18}}
\put(59,-46){\line(-19,29){29}}         
\put(-14,53){\makebox(0,0){$\mathbf{1}$}}
\put(44,53){\makebox(0,0){$\mathbf{2}$}}
\put(64,-50){\makebox(0,0){$\mathbf{3}$}}
\put(15,-72){\makebox(0,0){$\mathbf{4}$}}
\put(-33,-50){\makebox(0,0){$\mathbf{5}$}}
\end{picture}
\qquad \qquad
\begin{picture}(100,120)(-70,-70)
\thicklines
\multiput(1,0)(1,30){2}{\line(1,0){27.5}}
\multiput(0,1)(30,1){2}{\line(0,1){27.5}}
\put(0,30){\circle{4}}
\put(30,30){\circle*{4}}
\put(30,0){\circle{4}}
\put(0,0){\circle*{4}}
\put(-10,-25){\circle{4}}
\put(15,-40){\circle*{4}}
\put(0,0){\line(-10,-25){9}}
\put(15,-40){\line(-25,15){23}}
\put(15,-40){\line(20,50){15}}        
\put(15,-5){\circle{120}}
\put(15,-65){\circle*{3}}
\put(15,-40){\line(0,-1){24}}
\put(-12,48){\circle*{3}}
\put(42,48){\circle*{3}}
\put(-12,48){\line(12,-18){11}}
\put(42,48){\line(-12,-18){11}}
\put(-29,-46){\circle*{3}}
\put(59,-46){\circle*{3}}
\put(-29,-46){\line(19,21){18}}
\put(59,-46){\line(-19,29){29}}         
\put(-14,53){\makebox(0,0){$\mathbf{1}$}}
\put(44,53){\makebox(0,0){$\mathbf{2}$}}
\put(64,-50){\makebox(0,0){$\mathbf{3}$}}
\put(15,-72){\makebox(0,0){$\mathbf{4}$}}
\put(-33,-50){\makebox(0,0){$\mathbf{5}$}}

\put(15,40){\makebox(0,0){$\mathbf{15}$}}
\put(15,15){\makebox(0,0){$\mathbf{25}$}}
\put(11,-15){\makebox(0,0){$\mathbf{24}$}}
\put(60,0){\makebox(0,0){$\mathbf{12}$}}
\put(30,-45){\makebox(0,0){$\mathbf{23}$}}
\put(-5,-45){\makebox(0,0){$\mathbf{34}$}}
\put(-25,0){\makebox(0,0){$\mathbf{45}$}}
\end{picture}
\end{center}
\caption{A plabic graph $G$ together with
 $Q(G)$ and  its face labeling $\mathcal{F}_{\source}(G)$.
Here $\pi_G = (3,4,5,1,2)$.}
\label{fig:plabic2}
\end{figure}
Now we use the notion of trips to label each face of $G$
by a Pl\"ucker coordinate.
Towards this end, note that every trip
will partition the faces of a plabic graph into
two parts: those on the left of the trip, and those on the right
of a trip.

\begin{definition}\label{def:faces}
Let $G$ be a reduced plabic graph with $b$ boundary vertices.
For each one-way trip $T_{i\to j}$ with $i \neq j$, we place the label $i$
(respectively, $j$)
 in every face which is to the left of $T_{i\to j}$. If $i=j$ (that is, $i$ is adjacent to a lollipop), we place the label $i$ 
in all faces if the lollipop is white and in no faces if the lollipop is black.
We then obtain a labeling $\mathcal{F}_{\source}(G)$
(respectively, $\mathcal{F}_{\target}(G)$)
of faces of $G$ by subsets of $[b]$ which
we call the \emph{source} (respectively, \emph{target})
\emph{labeling} of $G$.  We identify each $a$-element subset of $[b]$
with the corresponding Pl\"ucker coordinate.
\end{definition}

The right-hand side of
\cref{fig:plabic2}
shows
a plabic graph with the face labeling
 $\mathcal{F}_{\source}(G)$.


We next associate a quiver to each plabic graph, and 
relate
quiver mutation to moves on plabic graphs.

\begin{definition}
Let $G$ be a reduced plabic graph.  We associate a quiver $Q(G)$ as follows.  The vertices of 
$Q(G)$ are labeled by the faces of $G$.  We say that a vertex of $Q(G)$ is \emph{frozen}
if the 
corresponding face is incident to the boundary of the disk, and is \emph{mutable} otherwise.
For each edge $e$ in $G$ which separates two faces, at least one of which is mutable, 
we introduce an arrow connecting the faces;
 this arrow is oriented so that it ``sees the white endpoint of $e$ to the left and the 
black endpoint to the right'' as it crosses over $e$.  We then remove oriented $2$-cycles
from the resulting quiver, one by one, to get $Q(G)$. See 
\cref{fig:plabic2}.
\end{definition}

\begin{definition}\label{def:graphseed}
Given a reduced plabic graph $G$, we let 
$\Sigma^{\target}_G$ (respectively, $\Sigma^{\source}_G$)  be the labeled seed consisting of the quiver $Q(G)$ 
with vertices labeled by the Pl\"ucker coordinates
 $\mathcal{F}_{\target}(G)$
(respectively,  $\mathcal{F}_{\source}(G)$).
Given a plabic graph $G$ on $n$ vertices and a permutation $v\in S_n$,
	we will sometimes use relabeled plabic graphs $v^{-1}(G)$ 
	(where the boundary vertices have been modified by applying 
	$v^{-1}$ to them).  We will refer to the corresponding seeds
	with the induced target labelings by e.g. 
	$\Sigma^{\target}_{v^{-1}(G)}$.
\end{definition}

The following lemma is straightforward, and is implicit in \cite{Scott}.

\begin{lemma}\label{lem:mutG}
If $G$ and $G'$ are related via a square move at a face,
	then $\Sigma^{\target}_G$ and $\Sigma^{\target}_{G'}$ are related via mutation at the corresponding vertex.
	Similarly for $\Sigma^{\source}_G$ and $\Sigma^{\source}_{G'}$. 
\end{lemma}

Because of \cref{lem:mutG}, we will subsequently refer to 
``mutating" at a nonboundary face of $G$, meaning that we mutate
at the corresponding vertex of quiver $Q(G)$.
Note that in general the quiver $Q(G)$ admits mutations at vertices
which do not correspond to moves of plabic graphs.  For example, $G$ might have a 
hexagonal face, all of whose vertices are trivalent; 
in that case, $Q(G)$ admits a mutation at the corresponding
vertex, but there is no move of plabic graphs which corresponds 
to this mutation.

\begin{remark}
The open positroid varieties $\pi_k(\mathcal{R}_{v,w}) \subseteq Gr_{k, n}$ are in bijection 
with a variety of combinatorial objects introduced by Postnikov in \cite{Postnikov}, including the 
	decorated permutations on $n$ letters with $k$ \emph{antiexcedances}.  Here we say that 
	$i\in [n]$ is an \emph{antiexcedance} if 
	 $\pi_{v,w}^{-1}(i)>\pi_{v, w}(i)$ or $i$ is a white lollipop.

As pointed out in \cite{shelling}, the trip permutation of $\pi_k(\mathcal{R}_{v,w})$ is $\pi_{v, w}:=v^{-1}w$ with all white fixed points lying in $v^{-1}([k])$ 
(see \cite[Section 2.4, Equation 2.27]{Karpman} for phrasing that is closer to ours). 
 The set of antiexcedances is exactly $v^{-1}([k])$.
	Clearly one can recover the pair $(v,w)$ from $\pi_{v,w}$ since $v \in W^K_{\max}$.
\end{remark}

\subsection{A fact about permutations}
 
We will need the following lemma on reduced expressions for permutations in $^K W$ and $W^K$. 
It is illustrated in \cref{fig:redexpression}.

\begin{lemma} \cite{Stem} \label{lem:redexpression}
Let $x \in {^KW}$ and let $\lambda:=\partne{x([k])}$. 
Choose a ``reading order" for the boxes of $\lambda$ such that each box is read before the box immediately below it and the box immediately to its right (that is, choose a standard Young tableaux of shape $\lambda$). 
Fill each box with a simple transposition; the box in row $r$ and column $c$ is filled with $s_{k-c+r}$.
Then reading the fillings of the boxes according to the reading order gives a reduced expression for $x$ (written  from right to left).
\end{lemma}

\begin{figure}
\setlength{\unitlength}{0.7mm}
\begin{center}
 \begin{picture}(50,35)
  \put(5,32){\line(1,0){36}}
  \put(5,23){\line(1,0){36}}
  \put(5,14){\line(1,0){36}}
  \put(5,5){\line(1,0){18}}
  \put(5,-4){\line(1,0){9}}
  \put(5,-4){\line(0,1){36}}
  \put(14,-4){\line(0,1){36}}
  \put(23,5){\line(0,1){27}}
  \put(32,14){\line(0,1){18}}
  \put(41,14){\line(0,1){18}}
  \put(8,26){$1$}
  \put(17,26){$5$}
  \put(26,26){$8$}
  \put(35,26){$10$}
  \put(8,17){$2$}
  \put(17,17){$6$}
  \put(26,17){$9$}
  \put(35,17){$11$}
  \put(8,8){$3$}
  \put(17,8){$7$}
  \put(8,-1){$4$}
 \end{picture}
    \qquad
	\begin{picture}(50,35)
  \put(5,32){\line(1,0){36}}
  \put(5,23){\line(1,0){36}}
  \put(5,14){\line(1,0){36}}
  \put(5,5){\line(1,0){18}}
  \put(5,-4){\line(1,0){9}}
  \put(5,-4){\line(0,1){36}}
  \put(14,-4){\line(0,1){36}}
  \put(23,5){\line(0,1){27}}
  \put(32,14){\line(0,1){18}}
  \put(41,14){\line(0,1){18}}
  \put(7,26){$s_4$}
  \put(16,26){$s_5$}
  \put(25,26){$s_6$}
  \put(34,26){$s_7$}
  \put(7,17){$s_3$}
  \put(16,17){$s_4$}
  \put(25,17){$s_5$}
  \put(34,17){$s_6$}
  \put(7,8){$s_2$}
  \put(16,8){$s_3$}
		\put(7,-1){$s_1$}
  \end{picture}
\end{center}
\caption{
	\label{fig:redexpression} 
Let $x=(2, 4, 7, 8, 1, 3, 5, 6) \in {^KW}$,
 and $\partne{x([k])}=(4, 4, 2, 1)$. On the left, the 
 \emph{columnar} reading order for the boxes of $\partne{x([k])}$; on the right, the filling of $\partne{x([k])}$ with simple transpositions. 
 This reading order produces the reduced expression $\textbf{x}=s_6 s_7 s_5 s_6 s_3 s_4 s_5 s_1 s_2 s_3 s_4$ for $x\in {^KW}$, and the reduced expression
 $s_4 s_3 s_2 s_1 s_5 s_4 s_3 s_6 s_5 s_7 s_6$ for $x^{-1}\in {W^K}$.}
\end{figure}

Since the elements of $W^K$ are just the inverses
of the elements of ${^KW}$,
 one can also obtain reduced expressions for $y \in W^K$ by the process described in \cref{lem:redexpression}, using the partition $\partne{y^{-1}([k])}$. 
 The only difference is the resulting reduced expression for $y$ is written down from left to right. 

\begin{remark}\label{rem:columnar} For simplicity, we will always use the \emph{columnar} reading order, which reads the columns of $\lambda$ from top to bottom, moving left to right (see \cref{fig:redexpression}). We will call the resulting reduced expressions \emph{columnar expressions}.
\end{remark}

We will be particularly concerned with pairs 
$(v,w)$ where 
$v \in W^K_{\max}$ and $w$ has a \emph{length-additive factorization}
$w = xv$, i.e. $\ell(w)=\ell(x)+\ell(v)$.  We will often use reduced expressions for such permutations $w$ that reflect their length-additive factorizations.
 
 \begin{defn} \label{def:stdredexpression} Let $v \leq w$, with $v \in W^K_{\max}$ and $w=xv$ length-additive. Let $v=w_K v'$ be length-additive, where $v'$ is necessarily in $W^K_{\min}$. Then a \emph{standard reduced expression} for $w$ is a reduced expression $\textbf{w}=\textbf{x}\textbf{w}_K \textbf{v'}$, where $\textbf{x}$ and $\textbf{v'}$ are the columnar expressions for $x$ and $v'$, respectively, and $\textbf{w}_K$ is an arbitrary reduced expression for $w_K$.
 \end{defn}

\section{The rectangles seed associated
to a skew Schubert variety} 
\label{sec:seed}

\cref{def:labeledquiver} explains how to associate to a pair of permutations  a quiver whose vertices are labeled by Pl\"ucker coordinates.
The construction is illustrated in \cref{fig:combconstruct}.


\begin{definition}[\emph{The rectangles seed $\Sigma_{v, w}$}]
\label{def:labeledquiver}
Let $v \leq w$, where $v \in W_{max}^K$ and $w=xv$ is a length-additive factorization. Let $\lambda:=\partne{x([k])}$. If $b$ is a box of $\lambda$, let $\Rect(b)$ be the largest rectangle contained in $\lambda$ whose lower right corner is $b$.

We obtain a quiver $Q_{v, w}$ as follows: place one vertex in each box of $\lambda$. A vertex is mutable if it lies in a box $b$ of the Young diagram
	and the box immediately southeast of $b$ is also in $\lambda$. 
We add arrows between vertices in adjacent boxes, with all arrows pointing either up or to the left. Finally, in every $2 \times 2$ rectangle in $\lambda$, we add an arrow from the upper left box to the lower right box. Equivalently, we add an arrow from the vertex in box $a$ to the vertex in box $b$ if  

\begin{itemize}
\item $\Rect(b)$ is obtained from $\Rect(a)$ by removing a row or column.
\item $\Rect(b)$ is obtained from $\Rect(a)$ by adding a hook shape.
\end{itemize}

We then remove all arrows between two frozen vertices.

To obtain the \emph{rectangles seed} $\Sigma_{v, w}$, we label each vertex of $Q_{v,w}$ with a 
Pl\"ucker coordinate. For $b$ a box of $\lambda$, let $J(b):=\vertne{\Rect(b)}$. The label of the vertex in $b$ is $\Delta_{v^{-1}(J(b))}$. This labeled quiver $\Sigma_{v,w}$ gives a seed as in \cref{def:seed0}, where the Pl\"ucker coordinates labeling the vertices  give the extended cluster.
\end{definition}  


\begin{definition}\label{def:hook-maximal}
	Let $\lambda$ be a partition and let $b$ be a box of $\lambda$. We say that $\Rect(b)$ is \emph{frozen for $\lambda$} or \emph{$\lambda$-frozen}
	 if $b$  touches the south or east boundary of $\lambda$ (either along an edge or at the southeast corner). 
\end{definition}
Note that the $\lambda$-frozen rectangles correspond to the frozen vertices of 
$\Sigma_{v,w}$.
\begin{figure}[h]
\setlength{\unitlength}{1.3mm}
\begin{center}
 \begin{picture}(50,35)
  \put(5,32){\line(1,0){36}}
  \put(5,23){\line(1,0){36}}
  \put(5,14){\line(1,0){27}}
  \put(5,5){\line(1,0){18}}
  \put(5,5){\line(0,1){27}}
  \put(14,5){\line(0,1){27}}
  \put(23,5){\line(0,1){27}}
  \put(32,14){\line(0,1){18}}
  \put(41,23){\line(0,1){9}}
	 \put(8,26){$\ydiagram{1}$}
	 \put(17,26){$\ydiagram{2}$}
	 \put(26,26){$\ydiagram{3}$}
	 \put(35,26){$\ydiagram{4}$}
	 \put(8,18){$\ydiagram{1,1}$}
	 \put(17,18){$\ydiagram{2,2}$}
	 \put(26,18){$\ydiagram{3,3}$}
	 \put(8,10){$\ydiagram{1,1,1}$}
	 \put(17,10){$\ydiagram{2,2,2}$}
	 \put(16,26.5){{\vector(-1,0){6.5}}}
	 \put(25, 26.5){{\vector(-1,0){6}}}
	 \put(16,18){{\vector(-1,0){6.5}}}
	 \put(8.5,11.5){{\vector(0,1){4.5}}}
	 \put(8.5,20){{\vector(0,1){5}}}
	 \put(18,20){{\vector(0,1){5}}}
	 \put(10,25){{\vector(1,-1){6}}}
	 \put(19,25){{\vector(1,-1){6}}}
	 \put(10,16){{\vector(1,-1){6}}}
 \end{picture}
    \qquad
        \begin{picture}(50,35)
  \put(5,32){\line(1,0){36}}
  \put(5,23){\line(1,0){36}}
  \put(5,14){\line(1,0){27}}
  \put(5,5){\line(1,0){18}}
  \put(5,5){\line(0,1){27}}
  \put(14,5){\line(0,1){27}}
  \put(23,5){\line(0,1){27}}
  \put(32,14){\line(0,1){18}}
  \put(41,23){\line(0,1){9}}
		\put(9,26){\circle*{1}}
		\put(8,28){$237$}
		\put(18,26){\circle*{1}}
		\put(17,28){$236$}
	 \put(27,26){$\ydiagram{1}$}
		\put(26,28){$235$}	
	 \put(36,26){$\ydiagram{1}$}
		\put(35,28){$234$}
		\put(9,18){\circle*{1}}
		\put(9,19){$137$}
	 \put(17.5,18){$\ydiagram{1}$}
		\put(17,15){$367$}
	 \put(26.5,18){$\ydiagram{1}$}
		\put(25,15){$156$}
	 \put(8,10){$\ydiagram{1}$}
		\put(8,7){$127$}
	 \put(17,10){$\ydiagram{1}$}
		\put(17,7){$167$}
	 \put(16.5,26.5){{\vector(-1,0){6.5}}}
	 \put(25.5, 26.5){{\vector(-1,0){6}}}
	 \put(16.5,18){{\vector(-1,0){6.5}}}
	 \put(8.5,11.5){{\vector(0,1){4.5}}}
	 \put(8.5,20){{\vector(0,1){5}}}
	 \put(18,20){{\vector(0,1){5}}}
	 \put(10,25){{\vector(1,-1){6}}}
	 \put(19,25){{\vector(1,-1){6}}}
	 \put(10,16){{\vector(1,-1){6}}}

	\end{picture}
\end{center}
\caption{	\label{fig:combconstruct}  
	An example of $\Sigma_{v, w}$ for $k=3$, $n=7$, $v=w_K$ and $x=wv^{-1}=(3,5,7,1,2,4,6)$. 
	At the left, the $\lambda$-frozen rectangles are 
	$\ydiagram{4}$, $\ydiagram{3}$, $\ydiagram{3,3}$, $\ydiagram{2,2}$,
	$\ydiagram{2,2,2}$, $\ydiagram{1,1,1}$.  On the right, the same quiver is shown but rectangles have been replaced by the corresponding 
	$3$-element subsets of $[7]$, which should be interpreted as  Pl\"ucker coordinates.} 
\end{figure}
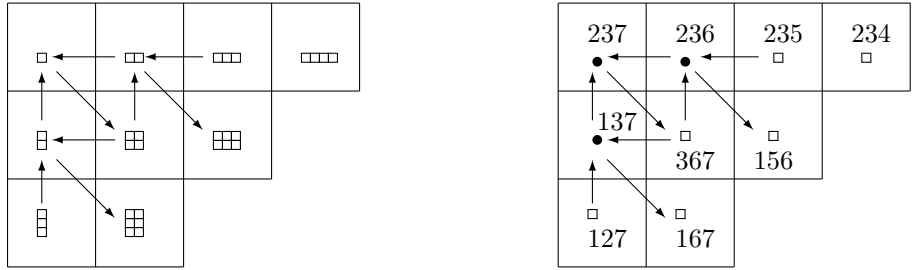

\begin{prop} \label{prop:rectanglesseed} Let ${\pi_k(\mathcal{R}_{v,w})}$ be a skew Schubert variety. Then the 
	rectangles seed 
	$\Sigma_{v, w}$ 
	is a seed for a cluster structure on the coordinate ring of (the affine cone over) 
	${\pi_k(\mathcal{R}_{v,w})}$, i.e. 
	$\CC[\widehat{\pi_k(\mathcal{R}_{v,w})}] = 
	\mathcal{A}(\Sigma_{v, w})$. 
\end{prop}

This result follows as an immediate corollary from \cref{thm:seedscoincide}, whose proof is the focus of \cref{sec:Leclerc}. 

In the following section, we discuss the generalized plabic graph whose dual quiver (with the target labeling) coincides
with  $\Sigma_{v, w}$, as well as the connections of this theorem to \cref{conj:vague}.

Recall that if $v\in W^K_{max}$ and $\lambda=\partsw{v^{-1}([k])}$, then ${\pi_k(\mathcal{R}_{v,w_0})}$ is the open Schubert variety $X^\circ_\lambda$. So as an immediate corollary to this result, we obtain the following.

\begin{cor} Let $v \in {W^K_{max}}$ and let $\lambda :=\partsw{v^{-1}([k])}$. Then the rectangles seed 
	$\Sigma_{v, w_0}$ 
	is a seed for the cluster structure on 
	$X^\circ_\lambda$, 
	i.e. 
	$\CC[\hat{X^\circ_\lambda}]= 
	\mathcal{A}(\Sigma_{v, w_0}).$ 
\end{cor}

\section{Obtaining the rectangles seed from a bridge graph}\label{sec:Karpman}


Here we give a construction of a special kind of plabic graph -- a \emph{bridge graph} -- from a pair of permutations \cite{Karpman}, and explain
how to use this construction to produce the rectangles seed.

\subsection{Bridge graphs} 

One can obtain a plabic graph with arbitrary trip permutation by successively adding ``bridges" 
(see \cref{fig:bridge}) 
to a graph consisting entirely of lollipops. The plabic graphs created this way are \emph{bridge graphs}.
We will define them below, after reviewing 
the notion of  \emph{(bounded) affine permutations}. 

\begin{figure}
\centering
\includegraphics[width=0.75\textwidth]{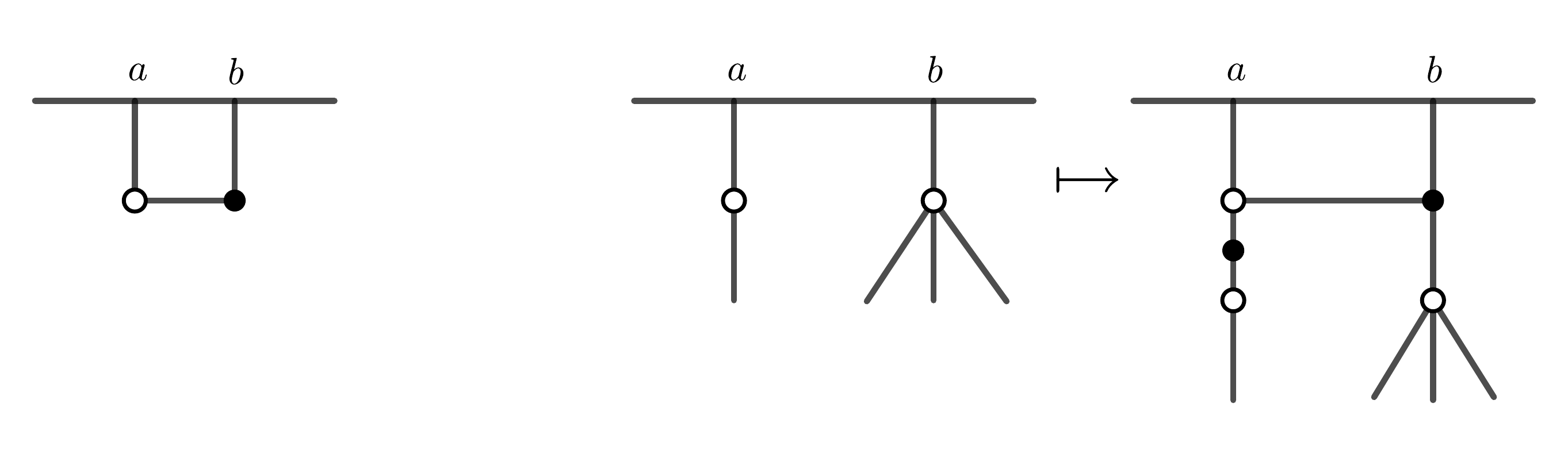}
\caption{\label{fig:bridge} On the left, an $(a~ b)$-bridge. On the right, an example of adding an $(a~ b)$-bridge to a plabic graph.}
\end{figure}

	An \emph{affine permutation} of order $n$ is a bijection $f: \Z \to \Z$ such that $f(i+n)=f(i)+n$ for all $i \in \Z$.

\begin{defn}[The bounded affine permutation associated to a decorated permutation]
If $\sigma$ is a decorated permutation of $[n]$, we define the bounded affine permutation $\tilde{\sigma}$ on $[n]$ as

\[\tilde{\sigma}(i):=\left\{ 
\begin{array}{l l} \sigma(i) & \text{if } \sigma(i)> i \text{ or } i \text{ is a black fixed point}\\
\sigma(i)+n & \text{if } \sigma(i)<i \text{ or } i \text{ is a white fixed point}
\end{array}
\right.
\]

and extend periodically to $\Z$.
In other words, to obtain a bounded affine permutation, add $n$ to all antiexcedances of $\sigma$ and then extend periodically to $\Z$.
\end{defn}

An \emph{$(a ~ b)$-bridge} is a collection of two vertices and three edges inserted at boundary vertices
$a$ and $b$ as in the left of 
\cref{fig:bridge}.
Let $G$ be a plabic graph with (bounded affine) trip permutation $\tilde{\sigma}_G$. For a pair of boundary vertices $a< b$, we say that the $(a ~ b)$-bridge is \emph{valid} if $\tilde{\sigma}_G(a)>\tilde{\sigma}_G(b)$, all boundary vertices $c$ between $a$ and $b$ are lollipops, and if $a$ (resp. $b$) is a lollipop it is white (resp. black). 

To add a bridge to $G$, choose boundary vertices $a, b$ such that the $(a ~ b)$-bridge is valid. Place a white (resp. black) vertex in the middle of the edge adjacent to $a$ (resp. $b$) and put an edge between these two vertices; if $a$ (resp. $b$) is a lollipop, we use the boundary leaf as the white (resp. black) vertex of the bridge. We then add degree two vertices as necessary to make the resulting graph bipartite.  A plabic graph obtained by successively adding valid bridges to a lollipop graph is called a \emph{bridge graph}.

Adding a bridge changes the trip permutation in a predictable way.

\begin{lem}[\protect{\cite[Proposition 2.5]{Karpman}}] \label{lem:bridgeperm} Suppose $G$ is a reduced plabic graph with (bounded affine) trip permutation $\tilde{\sigma}_G$. Let $1\leq a < b \leq n$ be vertices such that the $(a ~b)$-bridge is valid and let $G'$ be the plabic graph obtained by adding an $(a ~ b)$-bridge to $G$. Then $G'$ is reduced and has trip permutation $\tilde{\sigma}_G \circ (a ~ b)$.
\end{lem}

\begin{rem} Let $G_0$ be a lollipop graph, $(a_1 ~ b_1), \dots, (a_r ~ b_r)$ a sequence of bridges, and $G_i$ the graph obtained from adding bridge $(a_i ~ b_i)$ to $G_{i-1}$. We also assume that $(a_i ~ b_i)$ is a valid bridge for $G_{i-1}$. In the construction given above, new bridges are always added at the boundary and ``push" the existing faces towards the center of the disk (see \cref{ex:bridge}). One can also obtain $G_r$ from an empty graph by adding bridges in the opposite order, placing new bridges ``below" existing bridges, and adding lollipops at the end if necessary.  We will always use the former algorithm, but the latter can be useful as well.
\end{rem}

If $G'$ is obtained from a plabic graph $G$ by adding a valid bridge, all faces of $G'$ correspond to faces in $G$, except for the face bounded by the bridge.

\begin{lem} Suppose $G$ is a reduced plabic graph, $1\leq a < b \leq n$ vertices such that the $(a ~b)$-bridge is valid, and $G'$ the plabic graph resulting from adding an $(a ~ b)$-bridge to $G$. Then, using the target labeling, the labels of faces in $G$ coincide with the labels of corresponding faces in $G'$.
\end{lem}

It is not hard to find the (target) label of the remaining face of $G'$.

\begin{definition} \label{def:grassmannnecklace} Let $\sigma$ be a decorated permutation of $[n]$. The \emph{Grassmann necklace} of $\sigma$ is a sequence $\mathcal{J}=(J_1, J_2,\dots, J_n)$ of subsets of $[n]$ where $J_1:=\set{i \in [n] \ \vert \ \sigma^{-1}(i) > i \text{ or }i \text{ is a white fixed point}}$ and 
 \[J_{i+1}:= \begin{cases}
(J_i \setminus \set{i}) \cup \set{\sigma(i)} & \text{ if } i \in J_i\\
J_i & \text{ else}.
\end{cases}
\]
\end{definition}

If $\sigma_{G'}$ is the trip permutation of $G'$, the boundary faces of $G'$ are labeled with the Grassmann necklace of $\sigma_{G'}$ \cite{OPS}. 
So the label of the face bounded by the $(a ~b)$-bridge is just the $(a+1)^{st}$ entry of the Grassmann necklace of $\sigma_{G'}$.

\subsection{Bridge graphs from pairs of permutations}

In \cite{Karpman}, Karpman gives an algorithm for producing a bridge graph with trip permutation $vw^{-1}$ from a pair $(v, \textbf{w})$, where $v^{-1}\in W^K_{\max}$ and $\textbf{w}$ is a reduced expression for some permutation $w \geq v$. We use a special case of her construction in the following definition.

\begin{defn} \label{bridgegraph} 
	Let $w\in W$ with a  length-additive factorization $w=xw_K$, where
	$x\in {^KW}$.
	Let $\textbf{x}=s_{i_r} \dots s_{i_1}$ be the columnar expression for $x$ (see \cref{rem:columnar}) 
	and let $\textbf{w}$ be a standard reduced expression for $w$ 
	(\cref{def:stdredexpression}).
We define $B_{w_K, \textbf{w}}$ to be the bridge graph obtained from the lollipop graph with white lollipops $[k]$ and black lollipops $[k+1, n]$ with bridge sequence $s_{i_1}, s_{i_2}, \dots, s_{i_r}$. 
\end{defn}

By \cite{Karpman}, $B_{w_K, \textbf{w}}$ is a reduced plabic graph. By \cref{lem:bridgeperm}, $B_{w_K, \textbf{w}}$ has (decorated) trip permutation $x^{-1}$ with fixed points in $[k]$ colored white.

\begin{example} \label{ex:bridge} Let $k=2$, $n=5$, $x=(3, 5, 1, 2, 4)$ and $w=xw_K$. The partition $\partne{x([2])}$ corresponding to $x([2])=\{x(1), x(2)\}$ is $(3, 2)$, and the columnar expression for $x$ is $\textbf{x}=s_4 s_2 s_3 s_1 s_2$. So the bridge sequence for $B_{w_K, \textbf{w}}$ is $(2 ~ 3), (1 ~2), (3~ 4), (2~3), (4 ~5)$. To build $B_{w_K, \textbf{w}}$, we start with the lollipop graph

\begin{center}
\includegraphics[width=0.3\textwidth]{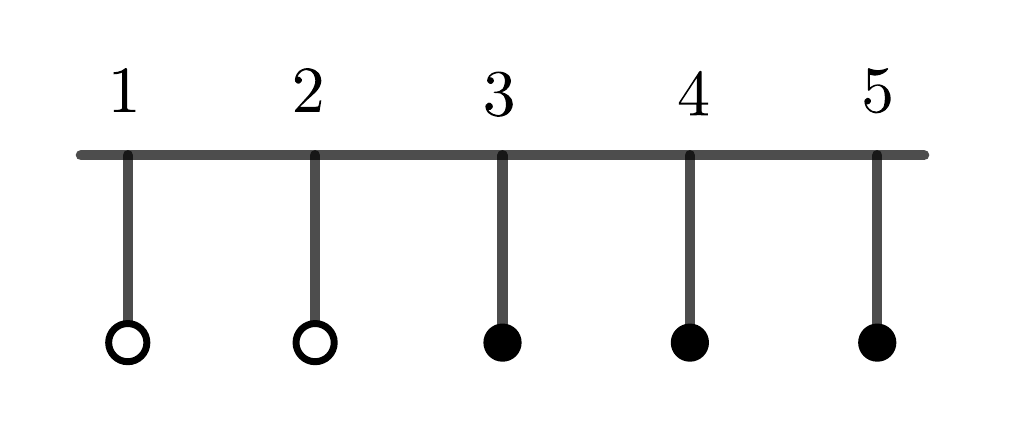}
\end{center}

then add the bridge $(2~3)$,

\begin{center}
\includegraphics[width=0.3\textwidth]{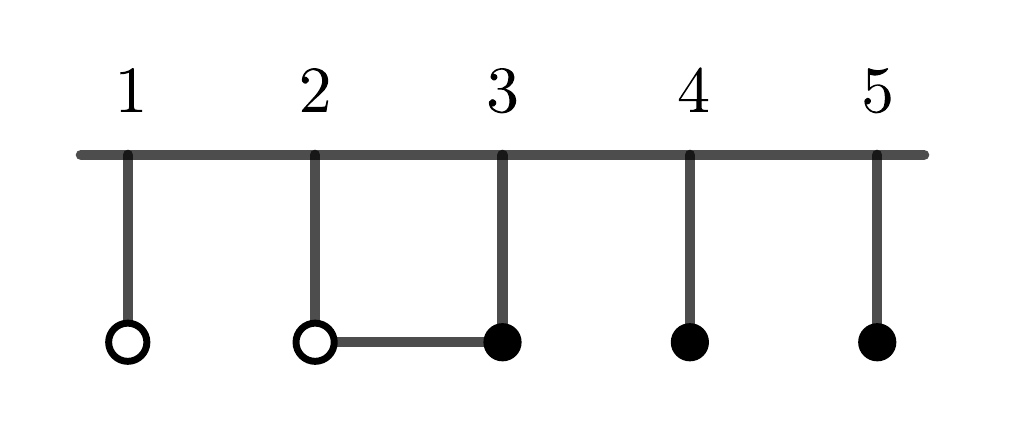}
\end{center}

the bridge $(1 ~ 2)$, 

\begin{center}
\includegraphics[width=0.3\textwidth]{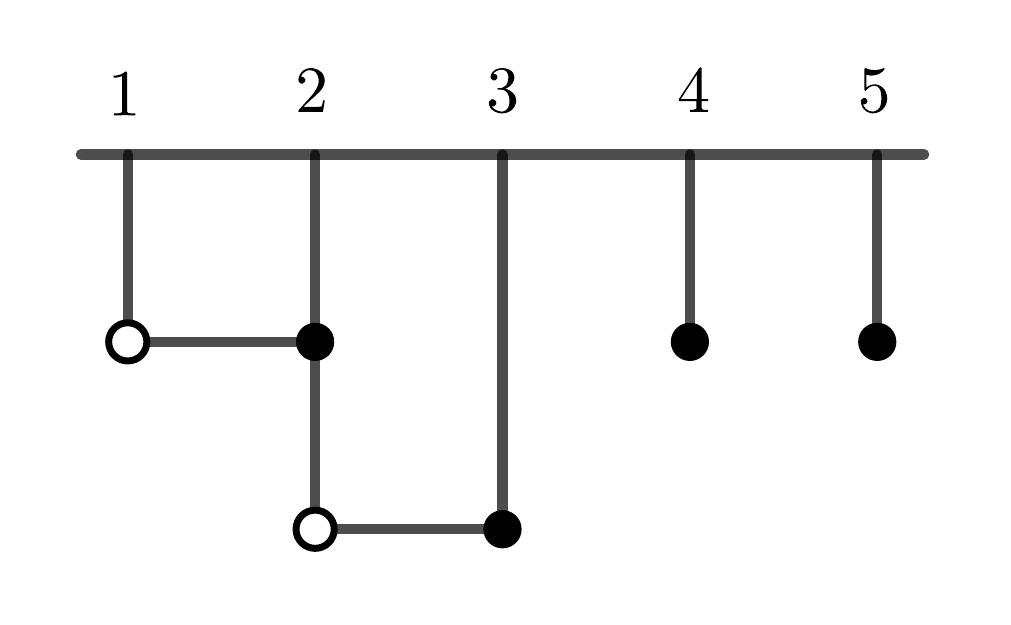}
\end{center}

and the bridges $ (3~ 4), (2~3), (4 ~5)$ to obtain the following graph.

\begin{center}
\includegraphics[height=3in]{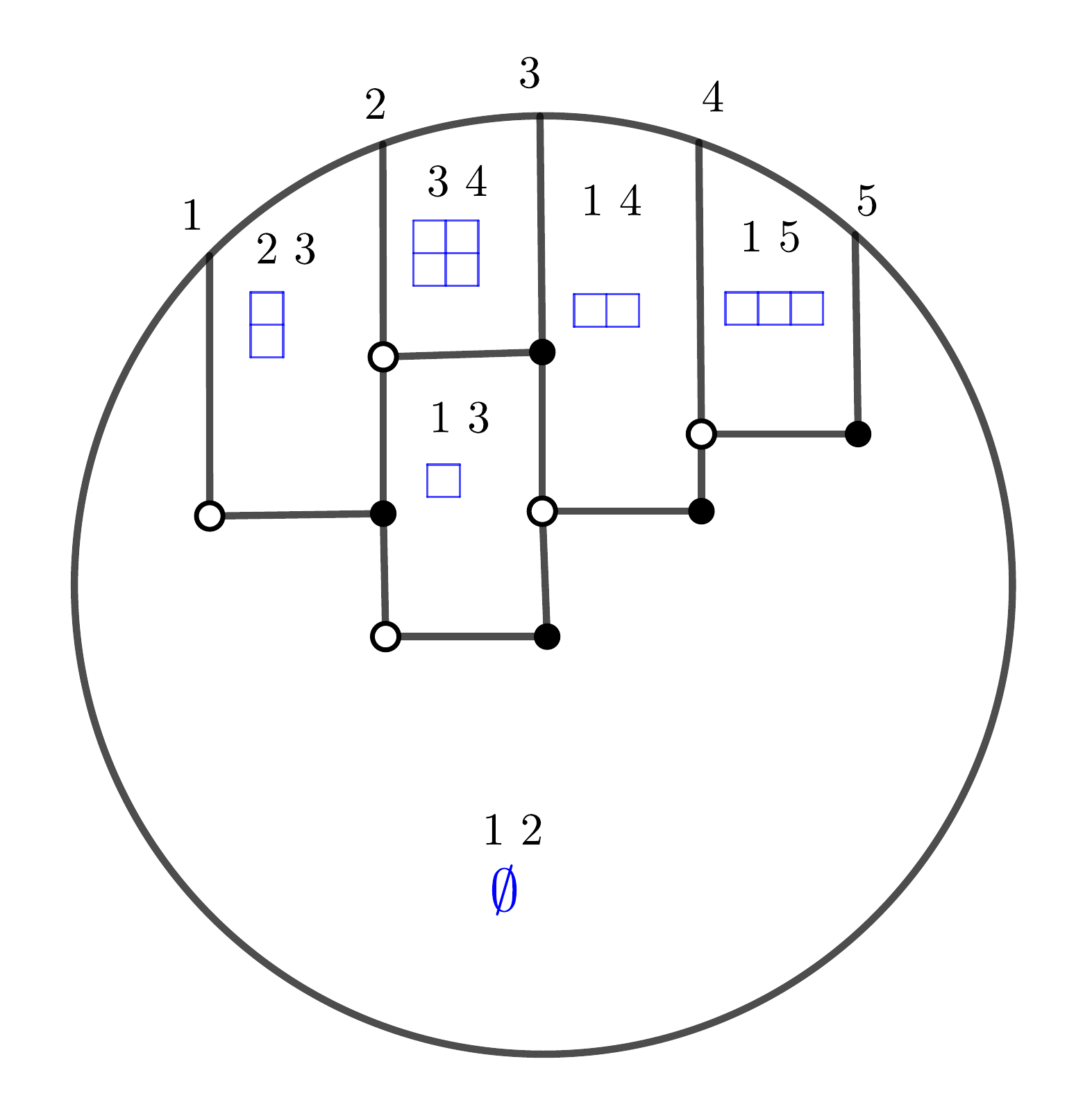}
\end{center}

Note that the (target) face labels of $B_{w_K, \textbf{w}}$ correspond to rectangles that fit inside of $\partne{x([2])}.$
\end{example}

%
%

The structure of $B_{w_K, \textbf{w}}$ will follow entirely from the structure of its Grassmann necklace. First, we need the following simple lemma. 

\begin{lemma} \label{lem:minlengthantiexc}
Let $x\in {^K W}$.

\begin{enumerate}
\item The fixed points of $x$ are $[p] \cup [q, n]$ for some $0 \leq p \leq k < q \leq n+1$.
\item For $i \in [k]$, $x(i)\geq i$.
\end{enumerate}
\end{lemma}

\begin{proof}
For the first statement, recall that $x \in {^K W}$ implies $x(1)<x(2)< \cdots <x(k)$ and $x(k+1)<\cdots <x(n)$. Suppose $x(j)=j$ for some $j \in [k]$. Since for $i <j$, $x(i)<x(j)$, we must have that $x([j])=[j]$. The increasing condition described above then implies that $x(i)=i$ for $i<j$. An analogous argument shows that if $x(j)=j$ for some $j \in [k+1, n]$, then $x(\ell)=\ell$ for all $\ell >j$. 

The second statement is clear from the condition that $x(1)<x(2)< \cdots <x(k)$.
\end{proof}

\cref{prop:grassrect} shows that the Young diagrams associated to the 
Grassmann necklace of $y\in W^K_{\min}$ are the rectangles
which are frozen for $\lambda:=\partne{y^{-1}[k]}$
(cf \cref{def:hook-maximal}),
together with $\emptyset$.

\begin{prop} \label{prop:grassrect} Let $y \in W^K_{\min}$ with fixed points $[p] \cup [q, n]$, and let 
	$\lambda:=\partne{y^{-1}[k]}$. 
	We color the fixed points of $y$ in $[k]$ white and all others black. Let $\mathcal{J}=(J_1, \dots, J_n)$ be the Grassmann necklace of $y$. Then $\partne{J_i}=\emptyset$ for $i \in [p+1] \cup [q, n]$. For other $i$, $\partne{J_i}$ is a rectangle which is frozen for  $\lambda$, and $\partne{J_{i+1}}$ can be obtained from $\partne{J_i}$ by adding a column to $\partne{J_i}$ if the resulting rectangle fits inside of $\lambda$ (that is, if $y(i)>k$) or removing a row from $\partne{J_i}$ if it does not (that is, if $y(i) \leq k$). 
In particular, every $\lambda$-frozen rectangle 
	occurs as one of the $\partne{J_i}$.
\end{prop}

\begin{proof} We induct on the length of $y$. If $y=e$, the white fixed points of $y$ are $[k]$, so $J_i=[k]$ for all $i$, corresponding to the empty set. 

Now, consider $y \neq e$. Note that by \cref{lem:minlengthantiexc}, if $i \in [k]$ is not a fixed point of $y$, then $y^{-1}(i)>i$. This together with our choice of decoration implies that the antiexcedance set of $y$ is $[k]$.

 Suppose the columnar expression for $y$ ends in $s_j$. Then $z:=y s_j$ is an element of $W^K$ corresponding to the partition $\lambda':=\partne{z^{-1}([k])}$, which is $\lambda$ with the bottom box of the 
	rightmost column removed. In other words, in $\lambda$, the $j^{th}$ step is horizontal and the $(j+1)^{th}$ step is vertical, and vice versa in $\lambda'$. Again, we color the fixed points of $z$ in $[k]$ white and the fixed points in $[k+1, n]$ black, and let $\mathcal{I}=(I_1, \dots, I_n)$ be the Grassmann necklace of $z$. 

Note that $J_r =I_r$ for $r \leq j$, since the antiexcedances of both permutations are $[k]$ and $y(r)=z(r)$ for $r \neq j, j+1$. Note also that since $\ell(y)>\ell(z)$, $y(j)> y(j+1)$. As $y$ is a minimum length right coset representative, this implies $y(j)>k \geq y(j+1)$. From this, we can conclude neither $j$ nor $j+1$ are fixed by $y$; otherwise, \cref{lem:minlengthantiexc} would lead to a contradiction. So $J_{j+1}=(I_j \setminus \set{j}) \cup \set{y(j)}$ and $J_{j+2}=(J_{j+1} \setminus \set{j+1}) \cup \set{y(j+1)}$.

By definition, $z(j+1) > k \geq z(j)$. By induction, $\partne{I_j}$ is a rectangle, so $I_j=[a] \cup [b, c]$ for $0 \leq a \leq b, c \leq n$. There are 4 cases, depending on if $j$ or $j+1$ is fixed by $z$. The cases in which at least one of $j$ and $j+1$ is fixed are straightforward, so we just prove the last.

%
%

	Suppose neither $j$ nor $j+1$ are fixed by $z$, so $\lambda'$ is obtained from $\lambda$ by removing a box that is not in the left column or top row. Suppose $I_j=[a] \cup [b, c]$. Since $z(j) \leq k$, $\partne{I_{j+1}}$ is obtained from $\partne{I_j}$ by removing a row, and we have that $I_{j+1}=[a+1] \cup [b+1, c]$. In other words, $j=b$ and $z(j)=a+1$. $\partne{I_{j+2}}$ is obtained from $\partne{I_{j+1}}$ by adding a column, so $I_{j+2}=[a+1] \cup [b+2, c+1]$; that is, $z(j+1)=c+1$. So $J_{j+1}=[a] \cup [b+1, c+1]$, which means that $\partne{J_{j+1}}$ is the rectangle obtained from $\partne{J_j}$ by adding a column. This rectangle fits inside of $\lambda$ because of where we added a box and is also $\lambda$-frozen, since its lower right corner touches the southeastern boundary of $\lambda$. Computation shows that $J_{j+2}=I_{j+2}$, and thus $\partne{J_{j+2}}$ is obtained from $\partne{J_{j+1}}$ by removing a row. Since $I_r=J_r$ for $r \neq j+1$, and all of the rectangles $\partne{I_r}$ are $\lambda$-frozen for $r \neq j+1$, we are done.
\end{proof}

As a corollary, we obtain the structure of the face labels of the plabic graphs we are interested in.

 \begin{cor} \label{cor:facelabels} 
	Let $w\in W$ with a  length-additive factorization $w=xw_K$, where
	$x\in {^KW}$.
	Let $\textbf{x}=s_{i_r} \dots s_{i_1}$ be the columnar expression for $x$ 
	and  $\textbf{w}$ be a standard reduced expression for $w$. 
	 Let $\lambda:=\partne{x([k])}$. 
	 Then the set of face labels of $B_{w_K, \textbf{w}}$ (see \cref{bridgegraph}) with respect to the target labeling is $\set{\vertne{\Rect(b)} \ \vert \ b \text{ a box of } \lambda} \cup \set{\vertne{\emptyset}}$. The boundary face labels correspond to the $\lambda$-frozen rectangles and the empty set.
 \end{cor}

 \begin{proof}

Recall that the bridge sequence of $B_{w_K, \textbf{w}}$ is exactly the simple transpositions in the columnar expression for $x^{-1}$; that is $s_{i_1}, \dots, s_{i_r}$. After placing the $j^{th}$ bridge, we get a plabic graph with trip permutation $s_{i_1} \cdots s_{i_j}$ with fixed points in $[k]$ colored white. Since $s_{i_1} \cdots s_{i_j} \in W^K_{\min}$, by \cref{prop:grassrect}, its Grassmann necklace consists of 
	  rectangles that are frozen for the partition corresponding to $s_{i_1} \cdots s_{i_j}$. The face labels of the boundary faces are the Grassmann necklace of the trip permutation, with $I_j$ labeling the face immediately to the left of $j$. When we add the $(j+1)^{th}$ bridge, we introduce a new boundary face (whose label is a rectangle that is frozen for the partition corresponding to 
	  $s_{i_1} \cdots s_{i_{j+1}}$) and the labels of all other faces stay the same. An old boundary face may be pushed off of the boundary by the new face; this occurs precisely when its label is not frozen for the new partition. Further, it is clear that every rectangle that fits into $\lambda$ is frozen for a partition corresponding to some prefix of $s_{i_1} \cdots s_{i_r}$.
 \end{proof}
 
 We can also describe the dual quiver of $B_{w_K, \textbf{w}}$.
 
 \begin{prop} \label{cor:dualquiver}
Let $w$, $x$, and $\textbf{w}$ be as in \cref{cor:facelabels}, and let $\lambda:=\partne{x([k])}$. Let $\mu, \nu$ be rectangles contained in $\lambda$ which are not the empty partition. In the dual quiver of $B_{w_K, \textbf{w}}$, there is an arrow from the face labeled $\vertne{\mu}$ to the face labeled $\vertne{\nu}$ if
 
 \begin{itemize}
\item $\nu$ is obtained from $\mu$ by removing a row or column
\item $\nu$ is obtained from $\mu$ by adding a hook shape
\end{itemize}

unless both faces are on the boundary, in which case there is no arrow between them. There is also an arrow from the face labeled $\vertne{\mu}$, where $\mu$ is a single box, to the face labeled $[k]$.
\end{prop} 

\begin{proof}
This follows from the proof of \cref{cor:facelabels} by induction on the number of bridges. 

To make the proof more uniform, we color all boundary vertices of $B_{w_K, \textbf{w}}$ adjacent to white (black) internal vertices black (white) and add arrows appropriately in the dual quiver. To obtain the statement of the proposition, just remove all arrows between frozen vertices.

Let $\textbf{x}=s_{i_r} \dots s_{i_1}$ be the columnar expression for $x$, so that $s_{i_1}, \dots, s_{i_r}$ is the bridge sequence for $B_{w_K, \textbf{w}}$. Note that $s_{i_1}=s_k$.

If there is only one single bridge, then $B_{w_K, \textbf{w}}$ has two faces, one face $f$ labeled with $[k] =\vertne{\emptyset}$ and the other face $f'$ labeled with $\vertne{\mu}$, where $\mu$ is a single box. From the coloring of vertices in a bridge, it is clear that the dual quiver has one arrow from $f'$ to $f$.

We examine what occurs after we place the final bridge $s_{i_r}=(j ~ j+1)$. Let $f'$ be the new face created by this bridge. Note that $j$ and $j+1$ cannot both be lollipops. Indeed, it is easy to see from the definition of the columnar reading order for $\lambda$ that $s_{i_r}$ is preceded by either a $s_{i_r -1}$ or a $s_{i_r +1}$ in the bridge sequence. If $j$ or $j+1$ is a lollipop, then the face $f'$ shares 2 edges with $f$, the face labeled with $[k]$. This means there are no edges between these faces in the dual quiver, since 2 shared edges results in an oriented 2-cycle. 

Note also that we do not have to add additional vertices of degree 2 after placing the bridge to make the graph bipartite; this is clear from the previous paragraph if $j$ or $j+1$ is a lollipop. If neither is a lollipop, from the columnar reading order, it is clear that there is a $s_{j-1}$ and a $s_{j+1}$ between each occurrence of $s_j$ in the sequence $s_{i_1}, \dots, s_{i_r}$, so $j$ is adjacent to a black internal vertex and $j+1$ is adjacent to a white internal vertex. This means that there is an arrow in the dual quiver between $f'$ and all adjacent faces that are not labeled with $[k]$. We discuss these arrows in the case when neither $j$ nor $j+1$ are lollipops, as the other cases are similar.

We know that $f'$ is labeled by (the vertical steps of) $\Rect(b)$, where $b$ is the last box of $\lambda$ in the columnar reading order. According to the proof of \cref{cor:facelabels}, to its right is the face labeled by (the vertical steps of) a partition $\nu$ obtained from $\Rect(b)$ by removing a row (since the partition obtained from $\Rect(b)$ by adding a column does not fit in $\lambda$). Similarly, to the left of $f'$ is the face labeled by (the vertical steps of) a partition $\nu'$ obtained from $\Rect(b)$ by removing a column. Below $f'$ is the face labeled by the partition obtained from $\Rect(b)$ by removing a hook shape. This, together with the color of vertices in bridges, gives the proposition.

\end{proof}

\subsection{Obtaining the rectangles seed from a plabic graph}

The goal of this section is to verify \cref{lem:sameseed}, which will 
be used in \cref{section:1.5} to 
deduce \cref{thm:main} from \cref{thm:main2}.

In what follows, when we say ``reflect a (generalized) plabic graph in the mirror", we mean the operation shown in \cref{fig:reflect}.

Now, we return to the setup of \cref{sec:seed}.

\begin{lemma}\label{lem:sameseed}  Let $v \leq w$ where $v \in W^K_{\max}$ and $w=xv$ is length-additive and let $\textbf{w}'$ be a standard reduced expression for $xw_K$. Consider the following generalized plabic graphs, with the indicated face labeling.
\begin{itemize}
\item $G_{v, w}$, obtained by applying $v^{-1}$ to the boundary vertices of $B_{w_K, \textbf{w}'}$, with target labels.
\item $G^{\text{mir}}_{v, w}$, obtained by applying $v^{-1}$ to the boundary vertices of $B_{w_K, \textbf{w}'}$ and reflecting in the mirror, with source labels.
\item $H_{v, w}$, obtained by applying $w^{-1}$ to the boundary vertices of $B_{w_K, \textbf{w}'}$, with source labels.
\item $H_{v, w}^{\text{mir}}$, obtained by applying $w^{-1}$ to the boundary vertices of $B_{w_K, \textbf{w}'}$ and reflecting in the mirror, with target labels.
\end{itemize}

The labeled dual quiver of each of these graphs, with the vertex labeled $v^{-1}([k])$ deleted, is $\Sigma_{v, w}$ (up to reversing all arrows).
\end{lemma}

\begin{proof}

Clearly $G_{v, w}$ and $H_{v, w}$ have the same (unlabeled) dual quiver as $B_{w_K, \mathbf{w}'}$; reflecting in the mirror reverses all arrows in the dual quiver. By \cref{cor:dualquiver}, the dual quiver of all of these graphs is $Q_{v, w}$ (see \cref{def:labeledquiver}), up to reversal of all arrows. 

Since the face labels of $G_{v, w}$ are obtained from those of $B_{w_K, \mathbf{w}'}$ by applying $v^{-1}$, it is clear from \cref{cor:facelabels} that the labeled dual quiver of $G_{v, w}$ is $\Sigma_{v, w}$. So it suffices to show that the face labels of $G_{v, w}$ agree with the face labels of the 3 other graphs.

Recall that the trip permutation of $B_{w_K, \mathbf{w}'}$ is $x^{-1}$. This implies that applying $v^{-1}$ to a target face label of $B_{w_K, \mathbf{w}'}$ gives the same set as applying $v^{-1}x^{-1}=w^{-1}$ to a source face label of $B_{w_K, \mathbf{w}'}$. Thus the face labels of $H_{v, w}$ are the same as the face labels of $G_{v, w}$.

Note that reflecting a generalized plabic graph in the mirror reverses all trips and also exchanges left and right. As a result, the target labels of $G_{v, w}$ are the same as the source labels of $G^{\text{mir}}_{v, w}$, and the source labels of $H_{v, w}$ are the same as the target labels of $H^{\text{mir}}_{v, w}$.

\end{proof}

\begin{figure}
\centering
\includegraphics[width=\textwidth]{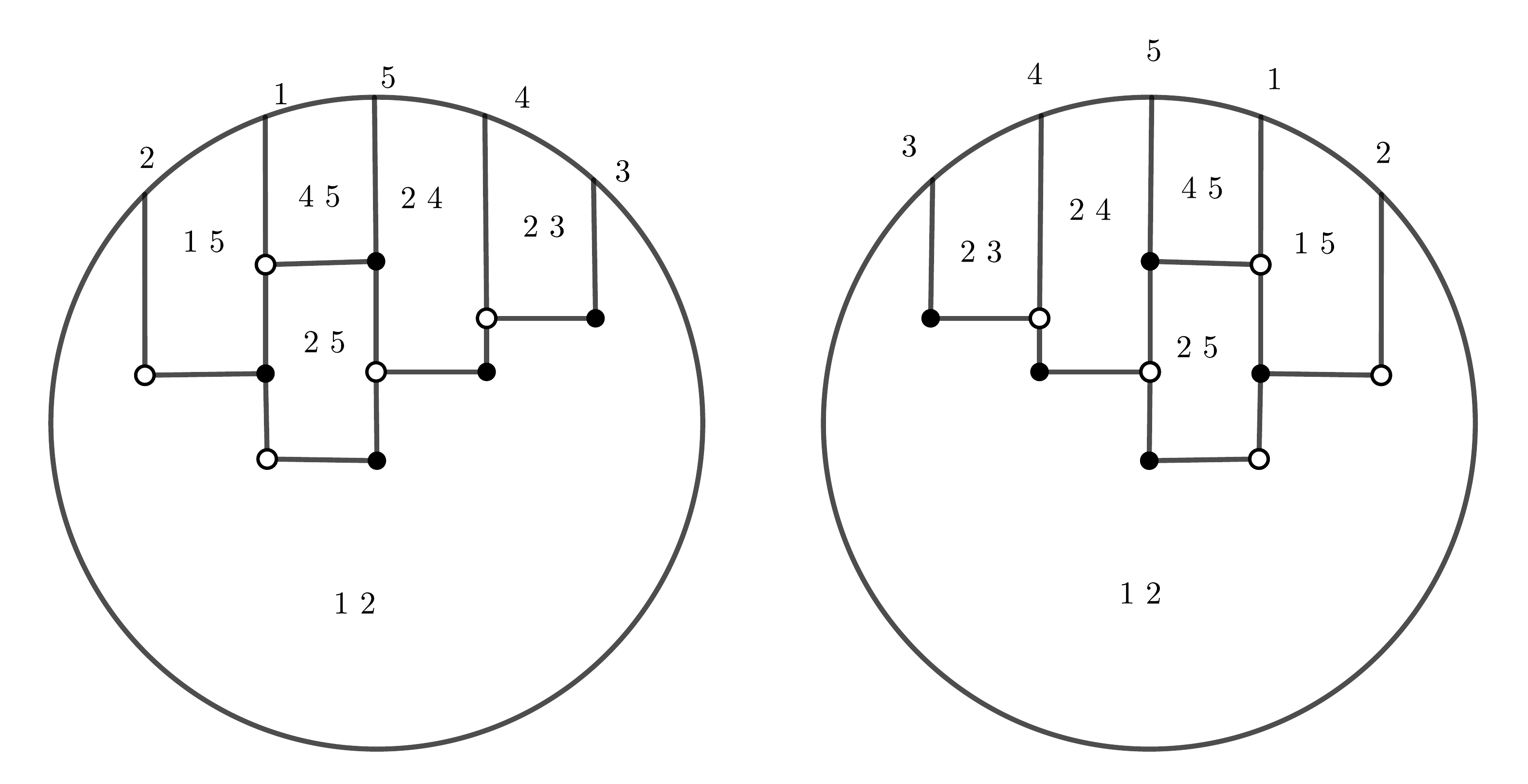}
\caption{\label{fig:reflect} Let $k=2$, $n=5$, $x=(3, 5, 1, 2, 4)$ and $w=xw_K$ as in \cref{ex:bridge}. On the left, we have applied $w_K^{-1}$ to the boundary vertices of $B_{w_K, \textbf{w}}$ to obtain $G_{w_K, w}$ (shown with target labels). On the right, we have ``reflected $G_{w_K, w}$ in the mirror" to obtain $G^{\text{mir}}_{w_K, w}$ (shown with source labels).}
\end{figure}

\begin{remark} \label{rem:lex_min}
Since we are actually interested in the affine cone over $\pi(\mathcal{R}_{v, w})$, we always assume that $\Delta_{v^{-1}([k])}$, the lexicographically minimal nonvanishing Pl\"ucker coordinate, is equal to $1$. This is why we delete the vertex labeled by $v^{-1}([k])$ in \cref{lem:sameseed}.
\end{remark}

Note that if $v=w_K$, $G^{\text{mir}}_{v, w}$ is a ``usual" plabic graph (that is, its boundary vertices are $1, \dots, n$ going clockwise). Similarly, if $w=w_0$, $H^{\text{mir}}_{v, w}$ is a usual plabic graph. So in these cases, the rectangles seed $\Sigma_{v, w}$ gives rise to the cluster structure conjectured in \cite[Conjecture 3.4]{MullerSpeyer0}. In general, $\Sigma_{v, w}$ gives rise to a different cluster structure; the cluster variables may differ and the frozen variables generally do not agree with the labels of the boundary faces (with either source or target labels) of a plabic graph corresponding to the positroid variety. However, the cluster structure given by $\Sigma_{v, w}$ is \emph{quasi-isomorphic} to the cluster structure conjectured in \cite[Conjecture 3.4]{MullerSpeyer0}. This means roughly that seeds in one cluster algebra can be obtained from seeds in the other by rescaling variables by Laurent monomials in frozen variables, in a way that is compatible with mutation (see \cite{Fraser}). Details will appear in \cite{FSB}.

\begin{remark} Applying $v^{-1}$ or $w^{-1}$ to the boundary vertices of $B_{w_K, \textbf{w}'}$ is a mysterious operation. This relabeling 
	takes a plabic graph associated to $\pi_k(\mathcal{R}_{w_K, w_Kx^{-1}})$ to one associated to $\pi_k(\mathcal{R}_{v, xv})$, and hence these positroid
	varieties are isomorphic.  
	We can describe the association of these two positroid varieties combinatorially in terms of \Le-diagrams: to obtain the \Le-diagram of $\pi_k(\mathcal{R}_{v, xv})$, rotate the \Le-diagram of $\pi_k(\mathcal{R}_{e, x^{-1}})$ by $180^{\circ}$, cut off boxes so it has shape $\partsw{v^{-1}([k])}$, and then perform \Le-moves until it satisfies the \Le-property (see \cref{sec:whichpositroids}). 
\end{remark}

\section{Obtaining the rectangles seed from Leclerc's categorical cluster structure}\label{sec:Leclerc}


\subsection{The categorical cluster structure for Richardson varieties}\label{sec:category-cluster}


We describe the categorical cluster structure on the coordinate ring of the Richardson variety $\mathcal{R}_{v, w}$ obtained in \cite{Leclerc}.  It involves representation theory of finite-dimensional algebras, see \cite{ASS, Sch14} 
for some background.  As we are only interested in the case of Grassmannians, we restrict our discussion to the construction in type $A$.  

Let $\Lambda$ be the \emph{preprojective algebra} over $\mathbb{C}$ of type $A$ and rank $n-1$.  It is the finite-dimensional path algebra of the \emph{double quiver}   

$$\overline{Q} = \xymatrix@!{1 \ar@/^/[r]^{\alpha_1}& \ar@/^/[l]^{\alpha_1^*} 2 \ar@/^/[r]^{\alpha_2}& \ar@/^/[l]^{\alpha_2^*} 3 \ar@/^/[r]^{\alpha_3}  &\cdots   \ar@/^/[l]^{\alpha_3^*} \ar@/^/[r]^{\alpha_4} & \ar@/^/[l]^{\alpha_4^*} n-1}$$  

on the vertex set $I=\{1, \dots, n-1\}$, subject to the relations
generated by 
$$\sum_{i} \alpha_i \alpha_i^*-\alpha_i^*\alpha_i=0.$$

In particular, the elements of $\Lambda$ are linear combinations of paths in the quiver modulo the relations, and multiplication is given by concatenation of paths.  Any finite-dimensional module $N$ over $\Lambda$ has an explicit realization in terms of the quiver.  In particular, $N$ is a collection 
$\{N_i\}_{i\in I}$ 
of finite-dimensional vector spaces over $\mathbb{C}$ 
for each vertex $i\in I$, together with a collection of linear maps $\phi_{\beta}: N_i\to N_j$ for every arrow $\beta:i\to j$ in the quiver.  Moreover, the composition of these linear maps must satisfy relations induced by the relations on the corresponding arrows. 

Let $\text{mod}\,\Lambda$ be the category of finite-dimensional $\Lambda$-modules.   For any $N\in \text{mod}\, \Lambda$ let $\abs{N}$ be the number of pairwise non-isomorphic indecomposable direct summands of $N$.  We use $\text{add} \, N$ to denote the additive closure of $N$, and  $\text{ind}\,N$ to denote the set of indecomposable direct summands of $N$.  Given a vertex $i$ in the quiver $\overline{Q}$,
  let $S_i$ denote the corresponding simple module and $Q_i$ denote the associated injective module.  The simple module $S_i$ is obtained by placing $\mathbb{C}$, a one-dimensional vector space, at vertex $i$ and $0$'s at the remaining vertices of the quiver. In this case, $\phi_{\beta}=0$ for all arrows $\beta$.  On the other hand, the injective $\Lambda$-module $Q_i$ also has a distinct structure, and we can represent $Q_i$ by its composition factors as follows.

$$\begin{smallmatrix}
&&&n-i&&&&&\\
& & n-i+1& &n-i-1&&&&\\
& \udots && n-i &&\ddots && \\
n-1 && \udots && \ddots && \ddots \\
 & \ddots &&\ddots && \udots && 1\\
&&\ddots && i && \udots  \\
&&&i+1&& i-1\\
&&&&i
\end{smallmatrix}
$$

In particular, when $n=6$ we obtain the following composition diagrams for the injective modules. 

$$Q_1=\begin{smallmatrix}5\\&4\\&&3\\&&&2\\&&&&1\end{smallmatrix}\hspace{.7cm}
Q_2=\begin{smallmatrix}&4\\5&&3\\&4&&2\\&&3&&1\\&&&2\end{smallmatrix}
\hspace{.7cm}
Q_3=\begin{smallmatrix}&&3\\&4&&2\\5&&3&&1\\&4&&2\\&&3\end{smallmatrix}
\hspace{.7cm}
Q_4=\begin{smallmatrix}&&&2\\&&3&&1\\&4&&2\\5&&3\\&4\end{smallmatrix}
\hspace{.7cm}
Q_5=\begin{smallmatrix}&&&&1\\&&&2\\&&3\\&4\\5\end{smallmatrix}
$$
These numbers can be interpreted as basis vectors or as composition factors (see \cite[Section 2.4]{GLS-KM}).  For example, 
the module $Q_2$ is an $8$-dimensional $\Lambda$-module with dimension vector $(d_1, d_2, d_3, d_4, d_5) = (1, 2, 2, 2, 1)$.  
In general, for every occurrence of $j\in I$ above we obtain the corresponding one-dimensional vector space $V_j\cong \mathbb{C}$ at vertex $j$ of the quiver.  Moreover, whenever we see a configuration $\begin{smallmatrix}j+1\\&j\end{smallmatrix}$ or $\begin{smallmatrix}& j-1\\j\end{smallmatrix}$ then the linear map between the corresponding spaces $V_{j+1}\to V_{j}$ or $V_{j-1}\to V_j$ is the identity.   We will often use this notation to denote other modules of $\Lambda$ that have a similar structure.   Moreover, in this notation, it is easy to see the top and socle of a given module $N$.   The top (resp. socle) of $N$ is a direct sum of simple modules $S_{i}$ such that the corresponding entry $i$ in the associated composition factor diagram lies at the top (resp. bottom).  In other words, there are no $i-1$ and no $i+1$ appearing directly above (resp. below) this $i$.  
For more information on preprojective algebras and their representation theory see \cite{GLS08,Rin96}. 



Next, for every $i\in I$ and $s_i\in W$ (where $W$ is the symmetric group on $n$ letters) we define two functors $\mathcal{E}_i=\mathcal{E}_{s_i}$ and $\mathcal{E}_i^{\dagger}=\mathcal{E}_{s_i}^{\dagger}$ on the category $\text{mod}\,\Lambda$.   Given $N\in \text{mod}\,\Lambda$ let $\mathcal{E}_i(N)$ be the kernel of a surjection 

$$\xymatrix{N\ar@{->>}[r]& S_i^a}$$

where $a$ is the multiplicity of $S_i$ in the top of $N$.  Note that $\mathcal{E}_i(N)$ is well-defined; it is obtained from $N$ by removing the $S_i$-isotypical part of its top.  Similarly, let $\mathcal{E}_i^{\dagger}(N)$ be the cokernel of an injection 

$$\xymatrix{S_i^b \,\,\ar@{^{(}->}[r]& N}$$

where $b$ is the multiplicity of $S_i$ in the socle of $N$.  The module $\mathcal{E}_i^{\dagger}(N)$ results from $N$ by taking away the $S_i$-isotypical part of its socle.  In terms of the corresponding composition factor diagrams, the diagram for $\mathcal{E}_i(N)$ (resp. $\mathcal{E}_i^{\dagger}(N)$) is obtained from that of $N$ by removing all entries $i$ appearing in the top (resp. bottom).  Moreover, for every $w\in W$ we can extend the definition to $\mathcal{E}_w, \mathcal{E}_w^{\dagger}$ by composing the functors associated to the simple reflections in a reduced expression for $w$.  It was shown in \cite[Proposition 5.1]{GLS08} that this definition does not depend on the choice of a reduced expression.  

Given $w\in W$, consider
$\mathcal{C}_w=\mathcal{E}_{w^{-1}w_0}(\text{mod}\,\Lambda)$ and $\mathcal{C}^w=\mathcal{E}_{w^{-1}}^{\dagger} (\text{mod}\,\Lambda)$, two subcategories of $\text{mod}\,\Lambda$ associated to $w$.  With this notation we can summarize the main theorem of \cite{Leclerc}.

\begin{theorem}\cite[Theorem 4.5]{Leclerc}\label{Lec-main-thm}
For every $v,w\in W$ with $v\leq w$, the subcategory $\mathcal{C}_{v,w}:= \mathcal{C}^v\cap \mathcal{C}_w$ has a cluster structure in the sense of \cite{BIRS}.  Moreover, $\mathcal{C}_{v,w}$ induces a cluster subalgebra in the coordinate ring $\mathbb{C}[\mathcal{R}_{v,w}]$, where the cardinality of the 
	extended cluster 
	is  equal to $\textup{dim}\,\mathcal{R}_{v,w}$.
\end{theorem}

In particular, the theorem says that $\mathcal{C}_{v,w}$ is a Frobenius category that admits a cluster-tilting object.  
Given a basic cluster-tilting module $T$ we can associate
the \emph{endomorphism quiver} $\Gamma_T$ as follows.
The vertices of $\Gamma_T$ are in bijection with indecomposable direct summands $T_i$ of $T$.  The number of arrows $T_i\to T_j$ in $\Gamma_T$ corresponds to the dimension of the space of irreducible morphisms $T_i \to T_j$ in $\text{add}\,T$.

Given a basic cluster-tilting module $T\in \mathcal{C}_{v,w}$, there is a notion of mutation of $T$ at an indecomposable summand $T_i$ of $T$, provided that $T_i$ is not projective-injective in $\mathcal{C}_{v,w}$.  
The {\it mutation} of $T$ at $T_i$ is a new cluster-tilting module $\mu_{T_i}(T):=T/T_i\oplus T_i'$, obtained by replacing $T_i$ by a unique different indecomposable module $T_i'\in \mathcal{C}_{v,w}$. Moreover, $T_i'$ is defined by the two short exact sequences 
$$\xymatrix{0\ar[r]&T_i' \ar[r] & B \ar[r]^g & T_i \ar[r] & 0 && 0\ar[r]&T_i \ar[r]^f & B' \ar[r] & T_i' \ar[r] & 0}$$
where $g$ and $f$ are minimal right and left $\text{add}\,(T/T_i)$-approximations of $T_i$. 
Thus, $B$ is a direct sum of $T_j \in\textup{ind}\,T$ for every arrow $T_j\to T_i$ in $\Gamma_T$, and  $B'$ is a direct sum of $T_j \in\textup{ind}\,T$ for every arrow $T_i\to T_j$ in $\Gamma_T$.

Furthermore, there is a cluster character $\varphi: \textup{obj}\,\mathcal{C}_{v,w}\to 
\mathbb{C}[\mathcal{R}_{v,w}]$ 
that maps modules $N\in \mathcal{C}_{v,w}$ to functions $\varphi_N \in 
\mathbb{C}[\mathcal{R}_{v,w}]$. 
   While the definition of $\varphi$ is rather complicated, $\varphi$ 
   satisfies several nice properties.  
   In particular, 
   for every $N,N'\in \mathcal{C}_{v,w}$, we have 
$$\varphi_{N\oplus N'}=\varphi_{N}\varphi_{N'}.$$
Moreover, for every mutation $\mu_{T_i}$ of a cluster-tilting module $T$, 
we obtain an exchange relation in $\mathbb{C}[\mathcal{R}_{v,w}]$:  
$$\varphi_{T_i}\varphi_{T_i'} = \varphi_{B}+\varphi_{B'},$$
where $B$ and $B'$ come from the short exact sequences above.
In this way the cluster character $\varphi$ induces a cluster algebra structure in $\mathbb{C}[\mathcal{R}_{v,w}]$ from a categorical cluster structure in $\mathcal{C}_{v,w}$.  

Next we want to give a more explicit version of \cref{Lec-main-thm}.

\begin{definition}\label{def:Uj}
Given $v \leq w$ in $W$ and a reduced expression 
 ${\bf w}=s_{i_t}\cdots s_{i_2}s_{i_1}$ 
	for $w$, we construct a set of modules $\{U_j\}$ 
	which will give rise to 
	a cluster in $\mathbb{C}[{\mathcal R}_{v,w}]$. 
Let $\bf{v}$ be the reduced subexpression for $v$ in $\bf{w}$ that is ``rightmost" in $\bf{w}$, called the \emph{positive distinguished subexpression} for $v$ in $\mathbf{w}$ (see \cref{def:posdistsubexpression}).  
Set $w_{(j)}=s_{i_j}\cdots s_{i_2}s_{i_1}$ for $1\leq j \leq t$, and let $w_{(j)}^{-1}:=(w_{(j)})^{-1}$. 
Let $v_{(j)}$ be the product of all simple reflections in $w_{(j)}$ that are  part of $\bf{v}$.   Define $J \subseteq \{1, \dots, t\}$ to be the collection of indices $j$ such that the corresponding reflection $s_{i_j}$ in the expression $\bf{w}$ is \emph{not} a part of $\bf{v}$.    

For every $j\in J$ we construct a module $U_j$ from the injective module $Q_{i_j}$.  For $N\in\text{mod}\, \Lambda$ let $\text{Soc}_{s_i} (N)$ be the direct sum of all submodules of $N$ isomorphic to the simple module $S_i$. Given a reduced word 
$z=s_{i_r}\dots s_{i_2}s_{i_1}$ in $W$, there is a unique sequence

$$0=N_0\subseteq N_1 \subseteq \cdots \subseteq N_r \subseteq N$$
of submodules of $N$ such that $N_p/N_{p-1}=\text{Soc}_{s_{i_p}}(N/N_{p-1})$.  Define $\text{Soc}_{z}(N)=N_r$.  
 For every $j\in J$,
let 
	\begin{equation}
		V_j = \text{Soc}_{w_{(j)}^{-1}} (Q_{i_j}) \hspace{1.5cm} \text{and} \hspace{1.5cm} U_j= \mathcal{E}^{\dagger}_{v_{(j)}^{-1}} V_j.
	\end{equation}
\end{definition}
\cref{ex:running} gives a detailed construction of a module  $U_j$.

The following theorem describes the cluster algebra structure in the coordinate ring of $\mathcal{R}_{v,w}$ and its additive categorification provided by $\mathcal{C}_{v,w}$.
	
\begin{theorem}\cite[Theorem 4.5 and Proposition 5.1]{Leclerc}\label{Lec-seed}
	Each pair $(v,{\bf w})$ as in \cref{def:Uj} gives a 
	cluster-tilting module 
		$U_{v,{\bf w}}:= 
		\bigoplus_{j\in J} U_j$ in $\mathcal{C}_{v,w}$, 
		that corresponds via the cluster character $\varphi$ to 
	a seed in $\C[\mathcal{R}_{v, w}]$ 
	as follows. 

\begin{itemize}

	\item[(a)] The cluster variables in $\C[\mathcal{R}_{v,w}]$ are the irreducible factors of $\varphi_{U_j} = \Delta_{v^{-1}_{(j)}([i_j]),w^{-1}_{(j)}([i_j])}$ for $j\in J$; they correspond to the indecomposable summands of $U_j$.  


\item[(b)] The frozen variables are the irreducible factors of $\Pi_{i\in I}\Delta_{v^{-1}([i]),w^{-1}([i])}$; they correspond to the  
	indecomposable summands of $\bigoplus_{i\in I} \mathcal{E}^{\dagger}_{v^{-1}} \mathcal{E}_{w^{-1} w_0} (Q_i)$ (which 
		are the projective-injective objects).

\item[(c)] The 
	 \emph{extended cluster} is the set of cluster and frozen variables, 
		which has cardinality $\textup{dim}\,\mathcal{R}_{v,w} = 
	\ell(w)-\ell(v)  = 
		\abs{U_{v,{\bf w}}}$.

	\item[(d)] The quiver associated to the seed is the endomorphism quiver $\Gamma_{U_{v,{\bf w}}}$ of the cluster-tilting module.  Moreover, the quiver has no loops and no 2-cycles, and the mutation of $U_{v,{\bf w}}$ induces  mutation on the quiver $\Gamma_{U_{v,{\bf w}}}$, in the sense of \cref{def:mutation}.
	\item[(e)] The cluster algebra
		$\tilde{R}_{v,w}$ 
		generated by all cluster variables 
		is a subalgebra of 
	 $\C[\mathcal{R}_{v, w}]$; when $w$ can be factored as $w=xv$ with 
		$\ell(w) = \ell(x)+\ell(v)$, the cluster algebra 
		$\tilde{R}_{v,w}$  is equal to 
	 $\C[\mathcal{R}_{v, w}]$.
\end{itemize}
\end{theorem}

\begin{example}\label{ex:running}
Let $n=7$ and consider the pair $(v,w)$ corresponding to a cell in $\text{Gr}(3,7)$, where $v=w_Ks_3$ and $w$ is given by the reduced expression
	$${\bf w}=s_5s_6s_4s_5s_2s_3s_4s_1s_2s_3{\bf s_1s_2s_1s_4s_5s_4s_6s_5s_4s_3} = s_{i_{20}} \dots s_{i_2} s_{i_1}.$$

The positive distinguished subexpression for $v$ in ${\bf w}$ is indicated in bold,
	and corresponds to the last ten transpositions at the end of ${\bf w}$.  The remaining transpositions determine the index set $J=\{11, 12, \dots, 20\}$, and for each $j\in J$ we obtain a summand $U_j$ of the cluster-tilting module $U_{v,{\bf w}}$.  Because the subexpression for $v$ appears at the end of ${\bf w}$, we have $v_{(j)}=v$ for all $j\in J$. We first compute $U_{14}$, denoting modules by their composition factors throughout.  Recall that 
	$$Q_4= {\begin{smallmatrix}&&&3&&\\&&4&&2&\\
	& 5 && 3 && 1\\6&&4&&2\\&5&&3\\&&4\end{smallmatrix}}. $$

Informally, we ``build up'' the composition diagram of $V_{14}=\text{Soc}_{w_{(14)}^{-1}}(Q_4)$ by adding composition factors from the diagram of $Q_4$, working from the bottom up. This process is illustrated below. We add composition factors in the order specified by the reduced expression $w_{(14)}^{-1}=\underline{s_3} s_4 \underline{s_5} \underline{s_6} \underline{s_4} \underline{s_5} s_4 \underline{s_1} \underline{s_2} s_1 \underline{s_3} s_2 s_1 \underline{s_4}$ (reading right to left). The underlined $s_i$'s indicate when an $i$ is added.

$$
\begin{smallmatrix}& && && \\&&&&\\&&&\\&&4\end{smallmatrix} 
\to 
\begin{smallmatrix}&  &&  && \\&&&&\\&&&3\\&&4\end{smallmatrix}
\to 
\begin{smallmatrix}& && && \\&&&&2\\&&&3\\&&4\end{smallmatrix}
\to 
\begin{smallmatrix}& && && 1\\&&&&2\\&&&3\\&&4\end{smallmatrix}
\to 
\begin{smallmatrix}& && && 1\\&&&&2\\&5&&3\\&&4\end{smallmatrix}
\to
\begin{smallmatrix}&  &&  && 1\\&&4&&2\\&5&&3\\&&4\end{smallmatrix}
\to
\begin{smallmatrix}& && && 1\\6&&4&&2\\&5&&3\\&&4\end{smallmatrix}
\to
\begin{smallmatrix}& 5 && && 1\\6&&4&&2\\&5&&3\\&&4\end{smallmatrix}
\to
\begin{smallmatrix}& 5 && 3 && 1\\6&&4&&2\\&5&&3\\&&4\end{smallmatrix}=V_{14}
$$

To get the composition diagram of $U_{14}=\mathcal{E}^{\dagger}_{v^{-1}}V_{14}$, we remove composition factors from the diagram of $V_{14}$, as illustrated below. We remove these factors from the bottom up, in the order specified by reading the reduced expression $v^{-1}=\underline{s_3} s_4 s_5 \underline{s_6} s_4 \underline{s_5} \underline{s_4} s_1 s_2 s_1$ right to left. The underlined $s_i$'s indicate when an $i$ is removed. 

$$V_{14}=\begin{smallmatrix}& 5 && 3 && 1\\6&&4&&2\\&5&&3\\&&4\end{smallmatrix}
\to 
\begin{smallmatrix}& 5 && 3 && 1\\6&&4&&2\\&5&&3\end{smallmatrix}
\to
\begin{smallmatrix}& 5 && 3 && 1\\6&&4&&2\\&&&3\end{smallmatrix}
\to
\begin{smallmatrix}& 5 && 3 && 1\\&&4&&2\\&&&3\end{smallmatrix}
\to
\begin{smallmatrix}& 5 && 3 && 1\\&&4&&2\end{smallmatrix}=U_{14}
$$

Performing similar computations for the remaining elements of $J$ we obtain the following set of modules:

\[U_{11}={\begin{smallmatrix}6\\&5&&3&&1\\&&4&&2\end{smallmatrix}} \hspace{1cm} U_{12}={\begin{smallmatrix}6\\&5&&3\\&&4&&2\end{smallmatrix}} 
\hspace{1cm} U_{13}={\begin{smallmatrix}6\\&5\\&&4\end{smallmatrix}}
\hspace{1cm} U_{15}={\begin{smallmatrix}&5&&3\\6&&4&&2\\&5&&3&&1\\&&4&&2\end{smallmatrix}}\]

\[U_{16}={\begin{smallmatrix}&5\\6&&4\\&5&&3\\&&4&&2\end{smallmatrix}}
\hspace{1cm}U_{17}={\begin{smallmatrix}&3&&1\\4&&2\end{smallmatrix}}
\hspace{1cm}U_{18}={\begin{smallmatrix}&& 3\\ & 4&&2\\5&&3&&1\\&4&&2\end{smallmatrix}}
\hspace{1cm}U_{19}={\begin{smallmatrix}&1\\2\end{smallmatrix}}
\hspace{1cm}U_{20}={\begin{smallmatrix}&&2\\&3&&1\\4&&2\end{smallmatrix}}\]

	The projective-injective objects of $\mathcal{C}_{v,w}$ are precisely $U_{13}, U_{15}, U_{16}, U_{18}, U_{19}, U_{20}$.

	

		The endomorphism quiver $\Gamma_{U_{v,{\bf w}}}$ is given below. 

$$\xymatrix{
{\begin{smallmatrix}6\\&5&&3&&1\\&&4&&2\end{smallmatrix}} \ar[dr] & {\begin{smallmatrix}5&&3&&1\\&4&&2\end{smallmatrix}} \ar[l]\ar[dr] & {\begin{smallmatrix}&3&&1\\4&&2\end{smallmatrix}} \ar[l]\ar[dr] & {\begin{smallmatrix}&1\\2\end{smallmatrix}} \ar[l]
\\
{\begin{smallmatrix}6\\&5&&3\\&&4&&2\end{smallmatrix}} \ar[u] \ar[dr] & {\begin{smallmatrix}&5&&3\\6&&4&&2\\&5&&3&&1\\&&4&&2\end{smallmatrix}} \ar[l]\ar[u] &{\begin{smallmatrix}&& 3\\ & 4&&2\\5&&3&&1\\&4&&2\end{smallmatrix}} \ar[l]\ar[u] &{\begin{smallmatrix}&&2\\&3&&1\\4&&2\end{smallmatrix}} \ar[u]\ar[l] 
\\
{\begin{smallmatrix}6\\&5\\&&4\end{smallmatrix}} \ar[u] & {\begin{smallmatrix}&5\\6&&4\\&5&&3\\&&4&&2\end{smallmatrix}} \ar[l]\ar[u]
}$$

In general, it is difficult to construct the endomorphism quiver 
		$\Gamma_{U_{v,{\bf w}}}$, because it is 
	 difficult to determine whether a given morphism is irreducible in $\text{add}\,U_{v,{\bf w}}$.
For example, there is a nonzero morphism $f: U_{15}\to U_{11}$ with image ${\begin{smallmatrix}5&&3\\&4&&2\end{smallmatrix}}$ but it factors through $U_{12}$.  Thus, $f$ does not induce an arrow in 
	 $\Gamma_{U_{v,{\bf w}}}$.
\end{example}

Our goal is to find an explicit description of the seed associated to 
a pair $(v,\mathbf{w})$, where 
$v\in W^K_{max}$, 
$w=xv$ is a length-additive factorization, and 
$\mathbf{w}= \mathbf{x} \mathbf{v}$ is a standard expression for $w$.  
In \cref{Lec-variables} we will analyze the cluster variables 
coming from \cref{Lec-seed}
(interpreting generalized minors
as Pl\"ucker coordinates), 
and in \cref{sec:modules} and \cref{sec:morphisms} we will 
analyze the modules $U_j$ and the morphisms between them, so as to obtain the quiver.  
The modules $U_j$ were previously defined constructively, so we 
need to develop a more explicit construction, which then enables us to understand the morphisms.
  In the case $v=w_K$, the modules have a particularly nice structure, which 
   allows us to explicitly compute the irreducible morphisms in $\text{add}\,U_{w_K,{\bf w}}$.  
  Next, we use a result of \cite{BKT} we show that there is a bijection between morphisms 
  $U_i\to U_j$ in $\text{add}\,U_{w_k,{\bf w}}$ and morphisms  $U'_i\to U'_j$ in $\text{add}\,U_{w_kv',{\bf wv'}}$. 
    Then we conclude that the quiver for the pair $(v,{\bf w})$ 
  coming from this representation theoretic construction agrees with the quiver coming from a plabic graph.

\subsection{Projecting the categorical cluster variables to Grassmannians}\label{Lec-variables}

When $v \leq w$ and 
$v\in W^K_{\max}$, the Richardson variety $\mathcal{R}_{v,w}$ in the complete flag variety
projects isomorphically to a 
positroid variety $\pi_k(\mathcal{R}_{v,w})$ in the Grassmannian $Gr_{k,n}$.
(Concretely, elements of this positroid variety are given by the span
of rows $v^{-1}\{1,\dots,k\}=v^{-1}[k]$ in a matrix representative $g$ for 
$Bg \in \mathcal{R}_{v,w}$).
When additionally there is a 
length-additive
factorization 
 $\mathbf{w}=\mathbf{xv}$, the positroid variety is a skew Schubert variety, and 
 \cref{Lec-seed} produces a cluster algebra which is equal to the coordinate ring
 $\CC[\mathcal{R}_{v,w}]$. 
 In this section we will determine how to interpret the cluster 
 variables in 
 $\CC[\mathcal{R}_{v,w}]$ as functions on the Grassmannian.
In particular, since 
 each generalized minor from \cref{Lec-seed} is a minor of a unipotent matrix, 
we can restrict that matrix to rows $v^{-1}[k]$,  and then identify
the minor with a Pl\"ucker coordinate of the resulting
$k \times n$ matrix.

For example, continuing \cref{ex:running} with
$v=w_Ks_3$ and $w$ given by 
$${\bf w}=s_5s_6s_4s_5s_2s_3s_4s_1s_2s_3{\bf s_1s_2s_1s_4s_5s_4s_6s_5s_4s_3},$$
 we obtain generalized minors which map
to the following Pl\"ucker coordinates:
\begin{enumerate}
\item $\Delta_{v^{-1}[3], v^{-1}s_3 [3]} = \Delta_{124, 247} = \Delta_{247}.$
\item $\Delta_{v^{-1}[2], v^{-1} s_3 s_2 [2]} = \Delta_{24, 47} = \Delta_{147}.$
\item $\Delta_{v^{-1}[1], v^{-1} s_3 s_2 s_1 [1]} = \Delta_{4, 7} = \Delta_{127}.$
\item $\Delta_{v^{-1}[4], v^{-1} s_3 s_2 s_1 s_4 [4]} = \Delta_{1247,2476} = \Delta_{246}.$
\item $\Delta_{v^{-1}[3], v^{-1} s_3 s_2  s_1 s_4 s_3 [3]} = \Delta_{124, 467} = \Delta_{467}.$
\item $\Delta_{v^{-1}[2], v^{-1} s_3 s_2 s_1 s_4 s_3 s_2 [2]} = \Delta_{24, 67} = \Delta_{167}.$
\item $\Delta_{v^{-1}[5],v^{-1} s_3 s_2 s_1 s_4 s_3 s_2 s_5 [5]} = \Delta_{12467, 24567} = \Delta_{245}.$
\item $\Delta_{v^{-1}[4], v^{-1} s_3 s_2 s_1 s_4 s_3 s_2 s_5 s_4 [4]} = 
\Delta_{1247, 4567} = \Delta_{456}.$
\item $\Delta_{v^{-1}[6], v^{-1} s_3 s_2 s_1 s_4 s_3 s_2 s_5 s_4 s_6 [6]} = 
\Delta_{124567, 234567} = \Delta_{234}.$
\item $\Delta_{v^{-1}[5], v^{-1} s_3 s_2 s_1 s_4 s_3 s_2 s_5 s_4 s_6 s_5 [5]} = 
\Delta_{12467, 34567} = \Delta_{345}.$
\end{enumerate}



\begin{remark}\label{rem:genminor} 
	Let $J \subseteq [n]$ with $|J| = \ell$.
If we project an $n \times n$ unipotent matrix $g$ to the Grassmannian element represented by the 
span of rows 
$v^{-1}[\ell]$ of $g$, the generalized minor
	$\Delta_{v^{-1}[\ell],J}$ of $g$
	equals the following Pl\"ucker coordinate of $Gr_{k,n}$:
\begin{itemize}
\item If $\ell<k$ and $|J \cup 
	v^{-1}([k]\setminus [\ell])|
		=k$, then 
	$\Delta_{v^{-1}[\ell],J} = 
		\Delta_{J \cup 
		v^{-1}([k]\setminus [\ell])}$.
\item If $\ell=k$ then 
	$\Delta_{v^{-1}[\ell],J} = 
\Delta_{J}$.
\item If $\ell>k$ and 
	$|J \setminus v^{-1}([\ell]\setminus [k])| = k$, then	
	$\Delta_{v^{-1}[\ell],J} = 
		\Delta_{J \setminus v^{-1}([\ell]\setminus [k])}$.	
\end{itemize}
\end{remark}

Using \cref{rem:genminor},
the following lemma implies that Leclerc's cluster variables in the seed corresponding to $(v, \mathbf{w})$ coincide
with those obtained from the rectangles seed defined in 
\cref{sec:seed}.


\begin{lemma}\label{lem:pluecker}
Choose a Young diagram contained in a $k \times (n-k)$ rectangle,
and label its boxes by simple reflections as in the right of 
\cref{fig1Leclerc}.
Choose a reading order for the boxes as in the left  of 
\cref{fig1Leclerc}.  
Choose any box $b$ and let $s_{\ell}$ be its label.
Let $w_b$ be the word obtained by reading boxes in order up through $b$
and recording the corresponding simple reflections.  For example
if $b$ is the box indicated by the bold $s_6$ in the right of 
\cref{fig1Leclerc}, then $w_b = (s_5 s_4  s_3 s_2) (s_6 s_5 s_4)(s_7 s_6)$.

\begin{figure}[h]
\setlength{\unitlength}{0.7mm}
\begin{center}
 \begin{picture}(50, 44)
  \put(5,41){\line(1,0){45}}
  \put(5,32){\line(1,0){45}}
  \put(5,23){\line(1,0){45}}
  \put(5,14){\line(1,0){36}}
  \put(5,5){\line(1,0){9}}
  \put(5,-4){\line(0,1){45}}
  \put(14,5){\line(0,1){36}}
  \put(23,14){\line(0,1){27}}
  \put(32,14){\line(0,1){27}}
  \put(41,14){\line(0,1){27}}
  \put(50, 23){\line(0, 1){18}}
  \put(8,35){$1$}
  \put(17,35){$5$}
  \put(26,35){$8$}
  \put(35,35){$11$}
  \put(44, 35){$14$}
  \put(8,26){$2$}
  \put(17,26){$6$}
  \put(26,26){$9$}
  \put(35,26){$12$}
  \put(44, 26){$15$}
  \put(8,17){$3$}
  \put(17,17){$7$}
  \put(26, 17){$10$}
  \put(35, 17){$13$}
  \put(8,8){$4$}
 \end{picture}
    \qquad
	\begin{picture}(50,44)
  \put(5,41){\line(1,0){45}}
  \put(5,32){\line(1,0){45}}
  \put(5,23){\line(1,0){45}}
  \put(5,14){\line(1,0){36}}
  \put(5,5){\line(1,0){9}}
  \put(5,-4){\line(0,1){45}}
  \put(14,5){\line(0,1){36}}
  \put(23,14){\line(0,1){27}}
  \put(32,14){\line(0,1){27}}
  \put(41,14){\line(0,1){27}}
  \put(50, 23){\line(0, 1){18}}
  \put(8,35){$s_k$}
  \put(14,35){$s_{k+1}$}
  \put(23,35){$s_{k+2}$}
  \put(35,35){$\cdots$}
  \put(5,26){$s_{k-1}$}
  \put(17,26){$s_k$}
  \put(23,26){$s_{k+1}$}
  \put(35,26){$\cdots$}
  \put(5,17){$s_{k-2}$}
  \put(17,17){$\vdots$}
  \put(26, 17){$\vdots$}
  \put(8,8){$\vdots$}
  \end{picture}
   \qquad
	\begin{picture}(50,44)
  \put(5,41){\line(1,0){45}}
  \put(5,32){\line(1,0){45}}
  \put(5,23){\line(1,0){45}}
  \put(5,14){\line(1,0){36}}
  \put(5,5){\line(1,0){9}}
  \put(5,-4){\line(0,1){45}}
  \put(14,5){\line(0,1){36}}
  \put(23,14){\line(0,1){27}}
  \put(32,14){\line(0,1){27}}
  \put(41,14){\line(0,1){27}}
  \put(50, 23){\line(0, 1){18}}

\linethickness{.5mm}
 \put(5,-4){\line(0,1){27}}
 \put(5,23){\line(1,0){27}}
 \put(32,23){\line(0,1){18}}
 \put(32, 41){\line(1, 0){18}}

  \put(8,35){$s_5$}
  \put(17,35){$s_6$}
  \put(26,35){$s_7$}
  \put(8,26){$s_4$}
  \put(17,26){$s_5$}
  \put(26,26){$\mathbf{s_6}$}
  \put(8,17){$s_3$}
  \put(17,17){$s_4$}
  \put(8,8){$s_2$}
  
	\put(1, -1){$1$}
	\put(1, 8){$2$}
	\put(1, 17){$3$}
	\put(33, 26){$7$}
	\put(33 , 35){$8$}
  \end{picture}
\end{center}

\caption{}
\label{fig1Leclerc}
\end{figure}

Also let $J(b):=\vertne{\Rect(b)}$ (see \cref{def:labeledquiver}).  In 
the right of 
\cref{fig1Leclerc}, 
$J(b) = \{1,2,3,7,8\}$.

Then for any $b$ and $\ell$ as above, let $J = w_b [\ell]$.
We have that:
\begin{enumerate}
\item  If $\ell < k$, then $J(b) = J \cup ([k]\setminus [\ell]) = 
J \cup \{\ell+1,\ell+2,\dots, k\}$.
\item If $\ell=k$, then $J(b)=J$.
\item If $\ell>k$, then $J(b) = J \setminus ([\ell]\setminus [k]) = 
J \setminus \{k+1, k+2,\dots, \ell\}$.
\end{enumerate}
\end{lemma}

\begin{proof}
The proofs of the three cases are quite analogous, so we will 
just prove the first one, where $\ell<k$.

Let box $b$ be in row $r$ and column $c$, as in 
\cref{fig2Leclerc}, so that its label is $s_{\ell} = s_{k-r+c}$.
We have that $r>c$.

\ytableausetup{boxsize=6.5mm}

\begin{figure}[h]
\centering
\setlength{\unitlength}{3.25mm}
\begin{picture}(30, 34)
\put(4, 30){\begin{ytableau}
\scriptstyle{s_k} & \scriptstyle{s_{k+1}} & \scriptstyle{s_{k+2}} & & & & & & & \\
\scriptstyle{s_{k-1}} & \scriptstyle{s_{k}} & \scriptstyle{s_{k+1}} & & & & & & & \\
\scriptstyle{s_{k-2}} &  & \scriptstyle{s_{k}} & & & & & & & \\
 & & & \scriptstyle{s_{k}} & & & & & & \\
 & & & & \scriptstyle{s_{k}} & & & & & \\
 & & & & & \scriptstyle{s_{k}}& & & \\ 
 & & & & & & \scriptstyle{s_{k}} & &\\ 
 & & & & &\text{\rotatebox[origin=c]{45}{$\scriptstyle{s_{k-r+c}}$}} &  &\scriptstyle{s_{k}}\\ 
 & & & & & \\
 & & & & & \\
 & & \scriptstyle{s_{i_3}}& \cdots&\scriptstyle{s_{i_{c-1}}} \\
\scriptstyle{s_{i_1}} &\scriptstyle{s_{i_2}} \\
\end{ytableau}}
\linethickness{.5mm}
\put(16.25, 32){\line(1, 0){16}}
\put(4, 15.6){\line(1, 0){12.25}}
\put(4, 4){\line(0, 1){11.6}}
\put(16.3, 15.6){\line(0, 1){16.4}}
\put(0, 30.5){Row 1}
\put(0, 28.5){Row 2}
\put(0, 16.5){Row $r$}
\put(0, 4.5){Row $k$}
\put(4, 33.5){Col}
\put(4.5, 32.5){1}
\put(6, 33.5){Col}
\put(6.5, 32.5){2}
\put(12, 33.5){Col}
\put(12, 32.5){$c-1$}
\put(14.5, 33.5){Col}
\put(15, 32.5){$c$}
\put(30, 33.5){Col}
\put(30, 32.5){$n-k$}
\end{picture}

\caption{}
\label{fig2Leclerc}
\end{figure}

Then $J(b) = \{1,2,\dots, k-r\} \cup 
\{k-r+c+1, k-r+c+2,\dots, k+c\}$,
and $J(b) \setminus \{\ell+1,\ell+2,\dots, k\} = 
J(b) \setminus \{k-r+c+1, k-r+c+2,\dots, k\} = 
\{1,2,\dots, k-r\} \cup \{k+1, k+2,\dots, k+c\}$.
We need to show that 
$$w_b \{1,2,\dots, k-r+c\} = 
J(b) \setminus \{k-r+c+1, k-r+c+2,\dots, k\} = 
\{1,2,\dots, k-r\} \cup \{k+1, k+2,\dots, k+c\}.$$

Let the labels of the simple generators in the bottom boxes
of columns $1, 2, \dots, c-1$ be 
$i_1, i_2, \dots, i_{c-1}$, respectively.  We also write
$i_c = k-r+c$.  Then we have that 
\begin{equation} \label{wb}
w_b = (s_k s_{k-1} s_{k-2} \dots, s_{i_1})(s_{k+1} s_k s_{k-1} \dots, s_{i_2}) \dots (s_{k+c-2} s_{k+c-3}, \dots, s_{i_{c-1}}) (s_{k+c-1} s_{k+c-2} \dots, s_{i_c}).
\end{equation}
Note that 
$$1 \leq i_1 < i_2 < i_3 \dots < i_{c-1} < i_c = k-r+c$$ so that 
$i_s \leq k-r+s$ for all $1 \leq s \leq c$.

We will now explicitly analyze $w_b(j)$ for $1 \leq j \leq k-r+c$.
 Towards this end, it's useful to observe
that for $a<b$,  the product
$s_b s_{b-1} s_{b-2} \dots s_a$  is equal to the cycle
$(b+1, b, b-1, \dots, a+1, a)$ (in cycle notation).

Then looking at \eqref{wb}, we see that:
\begin{itemize}
\item for $1 \leq j \leq i_1-1$, $w_b(j) = j \in \{1,2,\dots, k-r\}$.
\item for $j \in \{i_1, i_2, \dots, i_c\}$, 
$w_b(j) \in \{k+1, k+2,\dots, k+c\}$.
\end{itemize}
We also see that 
\begin{itemize}
\item for $i_1 < j < i_2$, $w_b(j) = j-1$
\item for $i_2 < j < i_3$, $w_b(j) = j-2$
\item $\vdots$
\item for $i_{c-1} < j < i_c$, $w_b(j) = j-(c-1)$.
\end{itemize}
So for $i_{s-1} < j < i_s$, we have that 
$w_b(j) = j - (s-1) < i_s - (s-1) \leq k-r+s-(s-1) = k-r+1$, 
and so $w_b(j) \leq k-r$.
This shows that for each 
$j \in \{1,2,\dots, k-r+c\}$, 
$w_b(j) \in \{1,2,\dots, k-r\} \cup \{k+1, k+2,\dots, k+c\}$,
and so $w_b[k-r+c] = 
 \{1,2,\dots, k-r\} \cup \{k+1, k+2,\dots, k+c\}$.  
This completes the proof of the lemma.
\end{proof}

\begin{corollary}\label{cor:quiverlabels}
Consider a  skew Schubert variety $\pi_k(\mathcal{R}_{v, w}) \subset Gr_{k,n}$, 
where $v\leq w$, $v\in W^K_{max}$, and with $w=xv$ length-additive.   Consider 
the seed for $\mathcal{R}_{v,w}$ 
given by \cref{Lec-seed} which is associated to a standard (columnar) reduced expression 
$\mathbf{w} = \mathbf{xv}$.
When we project the cluster variables
to $\pi_k(\mathcal{R}_{v,w})$, we obtain precisely the set of Pl\"ucker coordinates from the 
rectangles seed (\cref{def:labeledquiver}).  In other words, they are indexed by 
boxes $b$ in 
	$\partne{x([k])}$, 
	and are equal to the Pl\"ucker coordinates 
	$\Delta_{v^{-1}(J(b))}$ in the Grassmannian. 
\end{corollary}

\begin{proof}
Let $\mathbf{x}$ be the columnar expression for $x$ and $\mathbf{w}$ be a standard reduced expression for $w$. 
	Let $b$ be a box in 
	$\partne{x([k])}$, 
	and let $s_\ell$, $w_b$, and $J(b)$ be as defined in \cref{lem:pluecker}. Note that $w_b=x_{(i)}^{-1}$ for some $1\leq i \leq \ell(x)$, so $v^{-1}w_b=w_{(j)}^{-1}$ for some $j$. 
	Using \cref{rem:genminor} and applying $v^{-1}$ to \cref{lem:pluecker} implies 
	that the generalized minor $\Delta_{v^{-1}([\ell]), w_{(j)}^{-1}([\ell])}$ equals the Pl\"ucker coordinate
	$\Delta_{v^{-1}(J(b))}$ in the Grassmannian. 
	Each of these Pl\"ucker coordinates is irreducible in  $\CC[\widehat{\pi_k(\mathcal{R}_{v,w})}]$ (\cref{rmk:irreducible}), so we are done.
\end{proof}

It is not hard to see which Pl\"{u}cker coordinates are frozen in the rectangles seed.

\begin{lem} Let $x$ be a Grassmannian permutation of type $(k, n)$. Let $b$ be a $\lambda$-frozen box of $\lambda=\partne{x([k])}$, and let $s_\ell$ and $w_b$ be as defined in \cref{lem:pluecker}. Then $w_b([\ell])=x^{-1}([\ell])$. Thus $\Delta_{v^{-1}(J(b))}$ is frozen in the rectangles seed.
\end{lem}

\begin{proof}
It is clear from the filling of $\partne{x([k])}$ that the boxes in columns to the right of the column of $b$ are filled with $s_i$ such that $i > \ell$. So $x^{-1}=w_b u$, where $u$ is a permutation that fixes $[\ell]$ pointwise, so $w_b([\ell])=x^{-1}([\ell])$.

Using \cref{rem:genminor} and applying $v^{-1}$ to \cref{lem:pluecker} implies that $\Delta_{v^{-1}(J(b))}$ is the projection of $\Delta_{v^{-1}([\ell]), v^{-1}x^{-1}([\ell])}$ to the Grassmannian, which is frozen by \cref{Lec-seed}.
\end{proof}

\subsection{An explicit description of the modules $U_j$ when $w=xv$ 
is length-additive}
\label{sec:modules}
Throughout this section we fix a pair $(v,w)$, where $w=xv$ is a length additive factorization and $v\in W^K_{max}$.  
Let ${\bf w}$ be a standard reduced expression for $w$ 
(see Definition~\ref{def:stdredexpression}).
Thus, we can write $v=w_K v'$ where $v' \in W^K_{min}$, and choose reduced expressions ${\bf x}, {\bf v}'$ for ${x, v'}$ respectively as described in \cref{lem:redexpression}.  
Our goal in \cref{sec:modules} and 
\cref{sec:morphisms}
is to  prove that in this case,
the quiver from \cref{Lec-seed} agrees with the quiver from the rectangles seed.
Recall that the vertices of the quiver from \cref{Lec-seed} are indexed by 
 modules $U_j$.  In this section we will give an explicit
(non-recursive) description of the composition diagrams 
of these modules.



Let 
$$\mathbf{w} = \mathbf{x w_k v'} = (s_{i_t} \dots s_{i_{r+1}}) (s_{i_r} \dots s_{i_{p+1}}s_{i_p} \dots s_{i_{l+1}}) (s_{i_l} \dots s_{i_1})$$
where the parenthesis separate the subexpressions ${\bf x}, {\bf w_k}, {\bf v'}$.
In what follows, we will define a diagram $\mathcal{D}_{v,w}$ 
   (see \cref{fig:diagD})
whose boxes are filled with simple reflections in such a way 
that a natural reading order of the boxes gives the reduced expression $\mathbf{w} = \mathbf{x w_K v'}$.  Then to 
each $j\in J$ (see \cref{def:Uj}), we will associate a subdiagram $D_j$ with the property that if we replace each 
$s_i$ by $i$, $D_j$ is precisely the composition diagram of the module $U_j$.  More precisely, given a subdiagram of $\mathcal{D}_{v,w}$ the associated module is obtained as follows: for every $s_i$ directly followed by $s_{i+1}$ to the right (resp. $s_{i-1}$ below) in the subdiagram we obtain $\begin{smallmatrix}i\\&i+1\end{smallmatrix}$ (resp. $\begin{smallmatrix}&i\\i-1\end{smallmatrix}$) in the composition factor diagram of the module (see Figure~\ref{fig:ex2}). 


\begin{definition} (Diagram $\mathcal{D}_{v,w}$)
	Extending ideas from \cref{lem:redexpression}, we will build a diagram 
	$\mathcal{D}_{v,w}$ which encodes the reduced expression $\mathbf{w}$.
	We start by taking the 
	union of diagrams
	$R^*(v') \cup R(w_K) \cup R(x)$, 
	glued as in \cref{fig:diagD}, where
\begin{itemize}
\item $R^*(v')$ is a (rotated) Young diagram filled with simple reflections which give
	a reduced expression for $v'$, when read in the reading order indicated at the right of  \cref{fig:diagD};
\item $R(w_K)$ is a pair of staircase Young diagrams filled with simple reflections
	which give a reduced expression for $w_K$;
\item 	$R(x)$ is a Young diagram filled with simple reflections which give a 
	reduced expression for $x$.
\end{itemize}
	We additionally embed $R^*(v')$ into an $(n-k)\times k$ rectangle
	$D^*$ (with boxes filled with simple reflections as shown in \cref{fig:diagD}) and embed
        $R(x)$ into a $k \times (n-k)$ rectangle $D$ (with boxes filled
	with simple reflections as shown in \cref{fig:diagD}).
	We let $\mathcal{D}_{v,w}$ denote the union of $D$, $D^*$ and 
	$R(w_K)$, together with the paths defining $R^*(v')$ and $R(x)$.
	Note that $R^*(v') \cup R(w_K) \cup R(x)$ 
	encodes the reduced expression $\mathbf{w}$.

	Note that $R^*(v')$ is defined by the path 
	 $\pathsw{v'^{-1}([k])}$ rotated clockwise 90 degrees and then reflected across a vertical axis,
	 while 
	 $R(x)$ is defined by the path
	$\pathne{x ([k])}$.
	Finally we define the region $R(v')$ to be the subset of boxes of $D$ below  
	 $\pathsw{v'^{-1}([k])}$ (up to a rotation and and reflection, it agrees with $R^*(v')$).  Note that $v'^{-1}([k])=v^{-1}([k])$, so $R(x)\cap R(v')=\emptyset$ by \cref{prop:lengthsadd}.  
%
%
\end{definition}

\begin{figure}
\hspace*{0cm}\scalebox{.6}{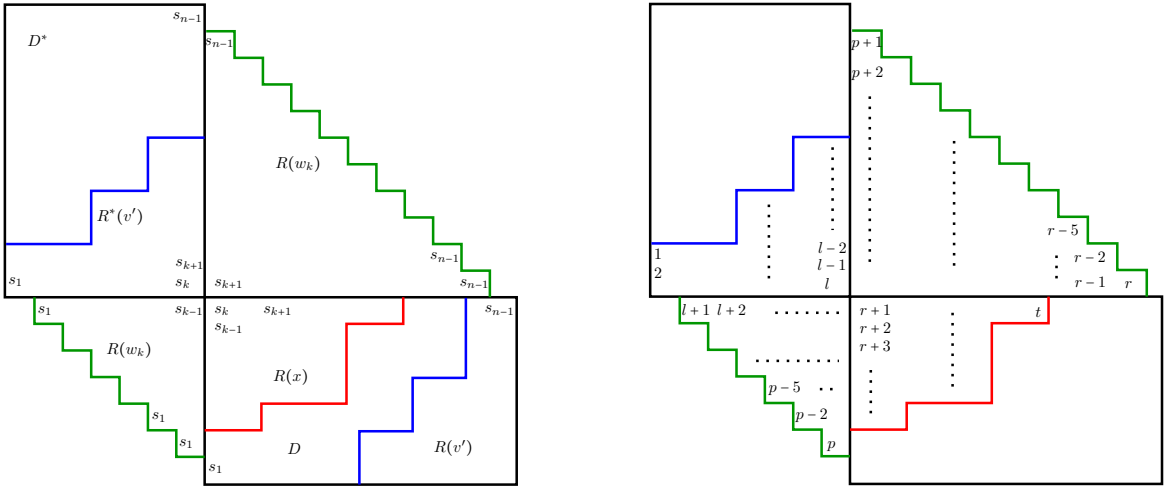}
 \caption{Diagram $\mathcal{D}_{v,w}$ (left) and reading order in each region (right).}
   \label{fig:diagD}
\end{figure}

Let $J, L$ be lattice paths from $(0, 0)$ to $(n-k, k)$ taking steps north and east and suppose $\vertne{J}=\set{j_1 < \cdots < j_k}$ and $\vertne{L}=\set{l_1 < \cdots <l_k}$. We say $J \leq L$ if $j_r \leq l_r$ for all $r$; that is, $J$ ``lies above'' $L$ when drawn in the plane (see \cref{fig:lengthsadd}). 
We leave the proof of \cref{prop:lengthsadd} to the reader.
\begin{lemma} \label{prop:lengthsadd} Let $\A=\{(v, w)\ \vert \ v \in W^K_{\max}, w\in W \text{ with length-additive factorization } w=xv\}$. Then the following map is a bijection:
\begin{align*} \A &\to \{(J, L)\ \vert \  J \leq L \text{ lattice paths from } (0, 0) \text{ to }(n-k, k)\}\\
(v, xv) & \mapsto (\pathne{x([k])}, \pathsw{v^{-1}([k])}).
\end{align*}

In particular, if $(v, xv) \in \A$ then $\pathne{x([k])} \leq \pathsw{v^{-1}([k])}$.
\end{lemma}

\begin{figure}
\setlength{\unitlength}{0.7mm}
\centering
 \begin{picture}(52, 30)
 \put(5,28){\line(1,0){45}}
 \put(5,19){\line(1,0){45}}
 \put(5,10){\line(1,0){45}}
 \put(5,1){\line(1,0){45}}
 \put(5, 1){\line(0, 1){27}}
 \put(14, 1){\line(0, 1){27}}
 \put(23, 1){\line(0, 1){27}}
 \put(32, 1){\line(0, 1){27}}
 \put(41, 1){\line(0, 1){27}}
 \put(50, 1){\line(0, 1){27}}
 \linethickness{.5mm}
 \put(4.75, 1){\color{red}\line(0, 1){9.25}}
 \put(4.75, 10.25){\color{red}\line(1, 0){9.25}}
 \put(14, 10.25){\color{red}\line(0, 1){8.75}}
 \put(14, 19){\color{red}\line(1, 0){18}}
 \put(32, 19){\color{red}\line(0, 1){9.25}}
 \put(32, 28.25){\color{red}\line(1, 0){18}}
 \put(1, 3){\color{red} \small $1$}
 \put(7, 11.5){\color{red} \small $2$}
 \put(11, 13){\color{red} \small $3$}
 \put(44, 30){\color{red} \small $8$}
 \put(0, 14){\color{red} $J$}
 \put(5.25, 1){\line(0, 1){9}}
 \put(5.25, 9.75){\line(1, 0){35.75}}
 \put(41, 9.75){\line(0, 1){18}}
 \put(41, 27.75){\line(1, 0){9}}
 \put(45, 23){\small $1$}
 \put(42, 21){\small $2$}
 \put(42, 12){\small $3$}
 \put(6,3){\small $8$}
 \put(52, 20){$L$}
 \end{picture}
\caption{\label{fig:lengthsadd} Let $k=3$ and $n=7$. Let $x=(1, 3, 6, 2, 4, 5, 7, 8) \in {^K}W$, $v=(8, 3, 2, 7, 6, 5, 4, 1) \in W^K_{\max}$ and $w=xv$. The upper lattice path $J$ is $\pathne{x([k])}$, with $\vertne{J}=x([3])=\set{1, 3, 6}$; the lower lattice path $L$ is $\pathsw{v^{-1}([k])}$, with $\vertsw{L}=v^{-1}([3])=\set{2, 3, 8}$. Since  $w=xv$ is  length-additive, the bijection of \cref{prop:lengthsadd} sends $(v, w)$ to $(J, L)$.}
\end{figure}
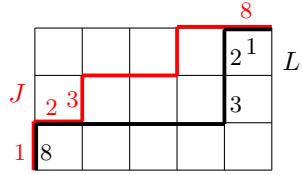

Next, we associate specific regions $D^*_j, D_j$ in $\mathcal{D}_{v,w}$ to the modules  $V_j, U_j$. 

{\noindent\bf Construction of $D^*_j\subset \mathcal{D}_{v,w}$.}
Given $j\in J$ there exists a corresponding box $b_j \in R(x)$ filled with the simple generator $s_{i_j}$.

By \cref{def:Uj}, 
$V_j = \text{Soc}_{w_{(j)}^{-1}}(Q_{i_j})$; we will construct a subdiagram $D^*_j$ 
of $\mathcal{D}_{v,w}$ that yields a composition diagram for the module $V_j$.  See Figure~\ref{fig:RegV}.

\begin{figure}
\hspace*{0cm}\scalebox{.65}{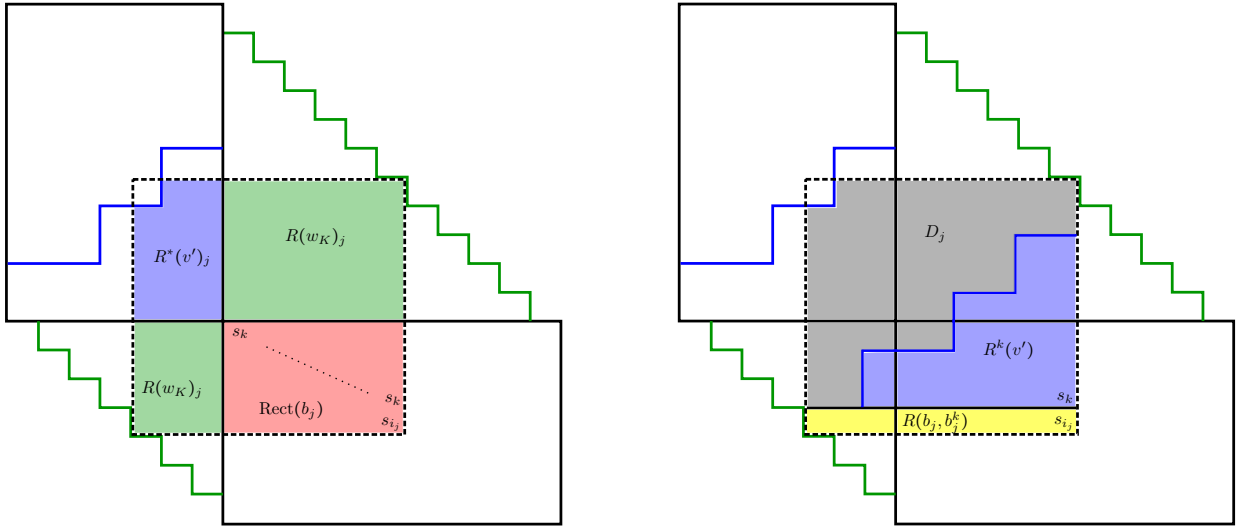}
 \caption{Construction of subdiagrams $D^*_j$ (left) and $D_j$ (right).}
   \label{fig:RegV}
\end{figure}

We can write $w_{(j)}=x_{(j)}w_Kv'$, where ${\bf x}_{(j)}$ comes from entries in $R(x)$ to the left and above $b_{j}$.    
In the definition of $V_j$, we begin with the injective module $Q_{i_j}$.   Note that $Q_{i_j}$ corresponds to the maximal rectangle 
$R(Q_{i_j})$ 
in $\mathcal{D}_{v,w}$ whose lower right corner is $b_{j}$. 
Next, consider the subdiagram associated to the module $\text{Soc}_{x_{(j)}^{-1}}(Q_{i_j})$.
The columnar expression for $x_{(j)}$ can be written as follows
$$x_{(j)}=s_{i_j}s_{i_j+1} s_{i_j+2} \dots s_a s_{i_j-1} s_{i_j} s_{i_j +1} \dots s_{a-1} \dots s_b s_{b+1} s_{b+2}\dots s_k$$
where $s_a$ (resp. $s_b$) is the filling of the box in the first row (resp. column) of $D$ and in the same column (resp. row) as $b_{j}$.  
It is compatible with the structure of $Q_{i_j}$ depicted in Section~\ref{sec:category-cluster} in the following sense. 

$$\text{Soc}_{s_{i_j}} (Q_{i_j})= \begin{smallmatrix}i_j \end{smallmatrix} \hspace{1.5cm} 
\text{Soc}_{s_{i_j+1}s_{i_j}}(Q_{i_j})=\begin{smallmatrix}i_j+1\\&i_j \end{smallmatrix}  \hspace{1.5cm}
\text{Soc}_{s_a\dots s_{i_j+2}s_{i_j+1}s_{i_j}}(Q_{i_j})=\begin{smallmatrix}a \\ & \ddots \\ && i_j+1 \\&&&i_j \end{smallmatrix}
$$ 

$$\text{Soc}_{s_{i_j-1} s_a\dots s_{i_j+2}s_{i_j+1}s_{i_j}}(Q_{i_j})=\begin{smallmatrix}a \\ & \ddots \\ && i_j+1 && i_j-1\\&&&i_j \end{smallmatrix} 
\hspace{1.5cm}
\text{Soc}_{s_{a-1}\dots s_{i_j} s_{i_j-1} s_a\dots s_{i_j+2}s_{i_j+1}s_{i_j}}(Q_{i_j})=\begin{smallmatrix}& a-1 \\a && a-2 \\ & \ddots & & \ddots \\ && i_j+1 && i_j-1\\&&&i_j \end{smallmatrix} 
$$

Continuing in this way, we see that the module $\text{Soc}_{x_{(j)}^{-1}}(Q_{i_j})$ is given by the rectangle $\Rect(b_{j})\subseteq D$,
whose southeast box is $b_{j}$.
Next, in the definition of $V_j$ we need to compute $\text{Soc}_{w_Kx_{(j)}^{-1}}(Q_{i_j})$.  First, observe that $s_k$ does not appear in a reduced expression for $w_K$, therefore the subdiagram of  $\mathcal{D}_{v,w}$ associated to $\text{Soc}_{w_Kx_{(j)}^{-1}}(Q_{i_j})$ has trivial intersection with 
$D^*$.  Recall that
the boxes in $R(w_K)$ yield a reduced expression for $w_K$.  Thus, we see that the subdiagram for $\text{Soc}_{w_Kx_{(j)}^{-1}}(Q_{i_j})$ is obtained by extending $\Rect(b_{j})$ to the north and west as much as possible, while avoiding boxes with entries $s_k$.  
In particular, we have 
$$R( \text{Soc}_{w_Kx_{(j)}^{-1}}(Q_{i_j}))=\Rect(b_{j})\cup R(w_K)_{j}$$
where $R(w_K)_{j}= R(w_K)\cap R(Q_{i_j})$.
Finally, $D^*_j$ is obtained from $R( \text{Soc}_{w_Kx_{(j)}^{-1}}(Q_{i_j}))$ by adding as many boxes in $R^*(v')$ as possible, such that the result is still contained in $R(Q_{i_j})$.  Let $R^*(v')_{j}=R(Q_{i_j})\cap R^*(v')$.  Then 
\begin{equation}\label{eq:Vj}
D^*_j= \Rect(b_{j})\cup R(w_K)_{j}\cup R^*(v')_{j}.
\end{equation}

\begin{remark}
The subdiagram $D^*_j$ can also be obtained as follows.  Given the box $b_j$, let $R$ be the maximal rectangle contained in $\mathcal{D}_{v,w}$ with $b_j$ as its southeast corner.   Then $D^*_j$ results from $R$ by removing boxes of $R\cap D^*$ which are not in $R^*(v')$.  
\end{remark}

{\noindent\bf Construction of $D_j\subset \mathcal{D}_{v,w}$.}  
By \cref{def:Uj}, 
 $U_j=\mathcal{E}^{\dagger}_{v_{(j)}^{-1}} V_j$ is obtained by removing simple modules from the socle of $V_j$ according to the simple reflections appearing in a reduced expression for $v_{(j)}^{-1}$.  Note that in this case $v_{(j)}=w_Kv'$.   Again since $w_K$ does not have $s_k$ in its reduced expression and it is the longest element of $W_K$, we see that $\mathcal{E}^{\dagger}_{w_K} V_j$ is obtained from $V_j$ by quotienting out the largest 
submodule of $V_j$ not supported at vertex $k$.  In particular, if $i_j=k$ then $\mathcal{E}^{\dagger}_{w_K} V_j=V_j$.  
If $i_j < k$ then there exists a unique box $b^k_{j}\in \Rect(b_{j})$ with entry $s_k$ located above $b_{j}$ and in the same column as $b_j$. 
Let $R(b_{j}, b^k_{j})$ be the maximal rectangle in $D^*_j$ with lower right corner $b_j$ and height $k-i_j$ (see Figure~\ref{fig:RegV}).  That is, the upper right corner of $R(b_{j}, b^k_{j})$ is the box directly below $b^k_{j}$.
We see that the module associated to $R(b_{j}, b^k_{j})$ is the largest submodule of $V_j$ not supported at vertex $k$.  Therefore, $R(\mathcal{E}^{\dagger}_{w_K} V_j)=D^*_j\setminus R(b_{j}, b^k_{j})$.  Similarly, if $i_j>k$ then there exists a unique box $b^k_{j}\in \Rect(b_{j})$ with entry $s_k$ located to the left of $b_{j}$ and in the same row as $b_j$.
Let $R(b_{j}, b^k_{j})$ be the maximal rectangle in $D^*_j$ with lower right corner $b_j$ of width $i_j-k$.  That is, the lower left corner of $R(b_{j}, b^k_{j})$ is the box directly to the right of $b^k_{j}$, and as before we obtain $R(\mathcal{E}^{\dagger}_{w_K} V_j)=D^*_j\setminus R(b_{j}, b^k_{j})$.


Finally, it remains to compute $\mathcal{E}^{\dagger}_{v'^{-1}} \mathcal{E}^{\dagger}_{w_K} V_j$.  Note that the columnar expression for $v'$ is again compatible with the structure of $\mathcal{E}^{\dagger}_{w_K} V_j$, see the computations below.  We have 
$$v'=s_k s_{k+1} \dots s_{l_1} s_{k-1} s_k \dots s_{l_2} \dots s_{k-q-1} s_{k-q} \dots s_{t_q}$$
where $s_{l_i}$ is the filling of the upper-most box in column $i$ of region $R^*(v')\subset \mathcal{D}_{v,w}$, if we label the columns of $R^*(v')$ by $1, 2, \dots, q$ right to left.

The socle of $\mathcal{E}^{\dagger}_{w_K} V_j$ is precisely $S_k$, and $s_k$ is the first reflection in ${\bf v}'$. Thus, $\mathcal{E}^{\dagger}_{s_kw_K} V_j$  is obtained from $\mathcal{E}^{\dagger}_{w_K} V_j$ by removing the simple module $S_k$ from the socle.  Similarly, $S_{k+1}$ is in the socle of $\mathcal{E}^{\dagger}_{s_kw_K} V_j$, and $s_{k+1}$ is the second reflection in ${\bf v}'$, provided the first column of $R^*(v')$ has at least two boxes.  Therefore, $\mathcal{E}^{\dagger}_{s_{l_1}\dots s_{k+1}s_kw_K} V_j$ is obtained from $\mathcal{E}^{\dagger}_{w_K} V_j$ by removing a portion of the left-most diagonal between labels $k$ and $l_1$ from the composition diagram for $\mathcal{E}^{\dagger}_{w_K} V_j$.  


$$\mathcal{E}^{\dagger}_{w_K} V_j = \begin{smallmatrix} \ddots && \vdots && && &&& &\vdots && \udots \\ & l_1+1 &  & l_1+3&& && &&& & k-q\\ && l_1 && &&& && & k-q-1\\ &&&\ddots && \ddots && \vdots & && \udots \\ &&& &k+2&  & k && k-2 \\ &&&& &k+1&& k-1\\ &&&&&&k\end{smallmatrix}
\hspace{1cm}
\mathcal{E}^{\dagger}_{s_{l_1} \dots s_{k+1}s_k w_K} V_j = \begin{smallmatrix} \ddots && \vdots && && &&& &\vdots && \udots \\ & l_1+1 &  & l_1+3&& && &&& & k-q\\ &&  && &&& && & k-q-1\\ &&& && \ddots && \vdots & && \udots \\ &&& &&  & k && k-2 \\ &&&& &&& k-1\\ &&&&&&\end{smallmatrix}$$
 
Continuing in this way, we see that $D_j$, the subdiagram associated to $U_j$, results from $R(\mathcal{E}^{\dagger}_{w_K} V_j)$ by removing the diagram $R^k(v')$, where $R^k(v')$ is obtained from $R^*(v')$ by 
shifting $R^*(v')$ southeast until its bottom right corner box is $b^k_{j}$ (see \cref{fig:RegV}).  This completes the construction of the region 
\begin{equation}\label{eq:Uj}
D_j=D^*_j\setminus (R(b_{j}, b^k_{j}) \cup R^k(v') ).
\end{equation}

Now, we use the above constructions to find a simple description of $D_j$ as a subdiagram of $D$.  See Figure~\ref{fig:ex2} for an example of such a transformation.

\begin{figure}
\hspace*{0cm}\scalebox{.8}{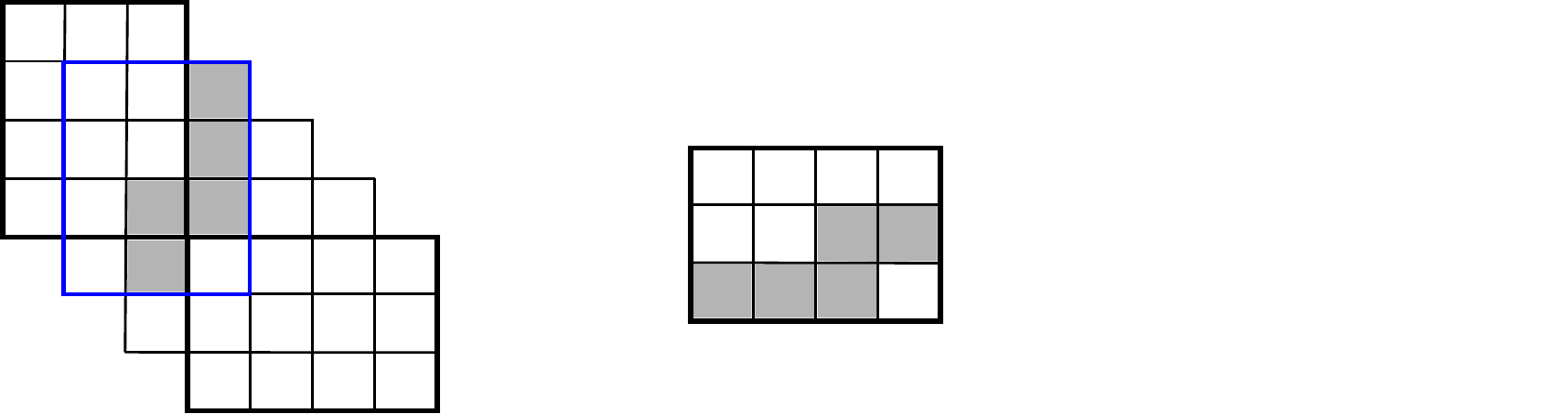}
 \caption{$D_{12}$ from Example~\ref{ex:running}.}
   \label{fig:ex2}
\end{figure}


Let $b_{j}$ be a box in $\partne{x([k])}$ (or equivalently a box in $R(x)$).
This box corresponds to the module $U_j$ and 
by \cref{cor:quiverlabels} the associated  
 Pl\"ucker coordinate is 
 $\Delta_{P_j}$ where $P_j = v^{-1} (J(b_{j}))$
 (see \cref{def:labeledquiver}).


Let $r(D)$ be the $k\times(n-k)$ diagram obtained by rotating $D$ clockwise 180 degrees. 
Define $R(P_j)$ to be the region in $r(D)$
bounded by the lattice paths $\pathsw{P_j}$ and $\pathsw{v^{-1}([k])}$ (see Figure~\ref{fig:RegJ}).  

\begin{figure}
\hspace*{0cm}\scalebox{.8}{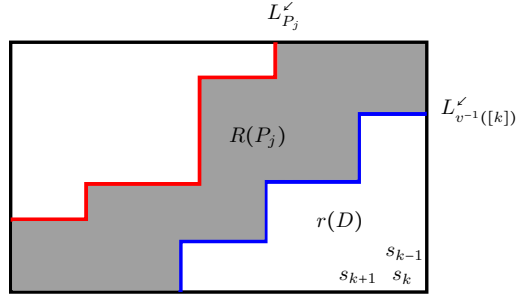}
 \caption{Region $R(P_j)$}
   \label{fig:RegJ}
\end{figure}

\begin{theorem}\label{thm-module-structure}
Let $w=xv$ where $v\in W^K_{max}$ and $\ell(w)=\ell(x)+\ell(v)$.  Given a pair $(v, {\bf w})$, where ${\bf w}$ is a standard reduced expression for $w$, for each $j\in J$ the region $R(P_j)$ gives the composition diagram for $U_j$.
\end{theorem}

\begin{proof}  
	In \eqref{eq:Uj} we defined a region $D_j$ in $\mathcal{D}_{v,w}$ that yields the desired module $U_j$, and we want to realize it as a region in $r(D)$. 

	Let $r(D)^*_{j}$ be a diagram in $\mathcal{D}_{v,w}$, that has the same shape as $D^*$ but whose southeast corner box coincides with the box $b^k_{j}$ (\cref{fig:RegV})

	
	We see that $D_j$ is contained in $r(D)^*_{j}$.  Moreover, by construction, region $D_j\subset r(D)^*_{j}$ is determined by two contours, 
see \cref{fig:RegV}.  Next, we provide an explicit formula for these contours, and then realize them as lattice paths in the diagram $r(D)$.  Note that the bottom contour of $D_j$ is always below or at most coincides with the top contour of $D_j$, therefore we can consider them separately.    

{\it The bottom contour of $D_j$.}  
By (\ref{eq:Uj}) the bottom contour of $D_j$ is determined by the Young diagram $R^k(v')$ associated to $v'$ (see Figure~\ref{fig:RegV}).  Observe that the bottom contour of $D_j$ can be realized as a northeast lattice path in $r(D)^*_{j}$.   By Lemma~\ref{lem:redexpression} the horizontal steps of this path are labeled by $(v')^{-1}([k])$.  Since $(v')^{-1}([k]) = v^{-1}([k])$, we obtain a desired description of the horizontal steps of the bottom contour of $D_j$ in $r(D)^*_j$.

{\it The top contour of $D_j$.}  
We proceed by induction on the length of $v'$.  
Let $\pathne{r(D)^*_{j}}$ denote the lattice path in $r(D)^*_{j}$ resulting in the top contour of $D_j$.  First, we consider the base case.
If $v'=e$, then by construction $R^*(v')=R^k(v')=\emptyset$, therefore by (\ref{eq:Uj})
$$D_j=\Rect(b_{j})\cup R(w_K)_{j} \setminus R(b_{j},b^k_{j}).$$
We depict $D_j$ in Figure~\ref{fig:top-contour}, where we consider the case $i_j\leq k$.  The remaining case can be proved similarly.  Recall that $\Rect(b_{j})$ is a rectangle in $D$, with entries at the corners being $s_{i_j}, s_k, s_a, s_b$ for some $k\leq a\leq n-1$ and $1\leq b \leq i_j$.  
We see that the horizontal steps of $\pathne{r(D)^*_{j}}$ consist of three intervals 
$$P=\{1, 2 \dots, e-1\}\cup \{d+1, d+2, \dots, k\}\cup \{c+1, c+2, \dots, n\}.$$
We claim that $P_j=P$, where $\Delta_{P_j}$ denotes the Pl\"ucker coordinate associated to the box $b_{j}$.

\begin{figure}
\hspace*{0cm}\scalebox{1.2}{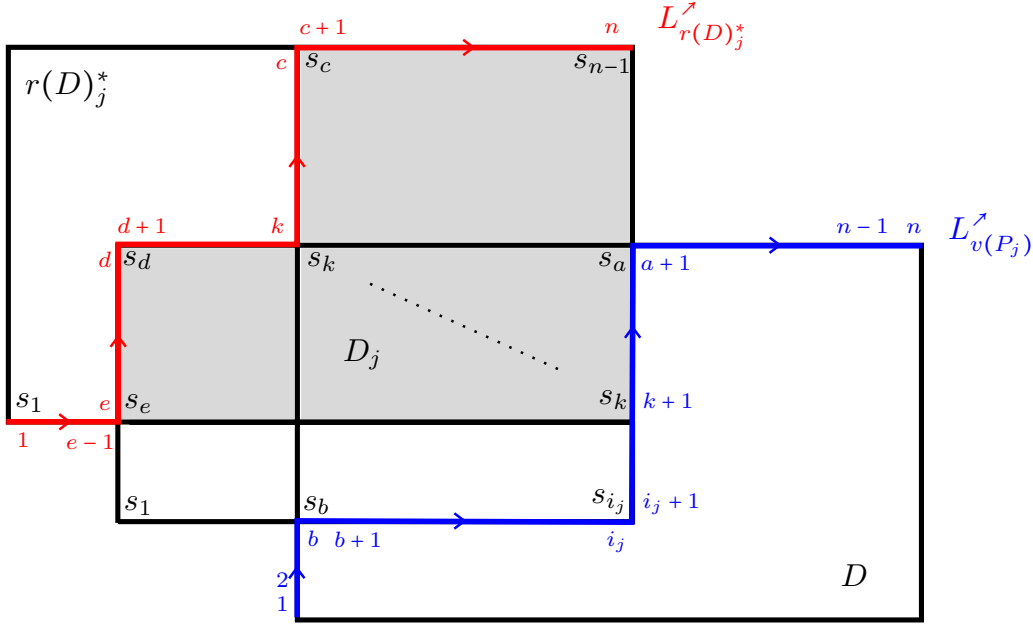}
 \caption{Base case $v'=e$}
   \label{fig:top-contour}
\end{figure}

By Lemma~\ref{lem:pluecker}, the lattice path $\pathne{v(P_j)}$ in $D$ cuts out a rectangle in the northwest corner of $D$.
In our case, this statement implies that 
$$v(P_j)=w_K(P_j) = \{1, 2, \dots, b-1\}\cup \{i_j+1, i_j+2, \dots, a+1\}.$$
Applying $w_K$ to this set where 
$$w_K=\Big( \begin{array}{cccccccccccc}1 & 2 & \dots & i & \dots & k & k+1 & \dots & j &\dots & n-1 & n \\ k & k-1 & \dots & k+1-i & \dots & 1 & n & \dots & n-(j+k-1) &\dots & k+2 & k+1\end{array}\Big)$$
we see that 
$$P_j = \{k, k-1, \dots, k-b+2\}\cup \{k-i_j, k-i_j-1, \dots, 1\}\cup \{n, n-1, \dots, n-a+k\}.$$
It follows from Figure~\ref{fig:top-contour} that $e=k-i_j+1, d=k-b+1, c=n-a+k-1$, which implies that $P_j=P$.   This completes the proof that the top contour of $D_j$ given by $\pathne{r(D)^*_{j}}$  has horizontal steps $P_j$  in the case $v'=e$. 

Now, consider the pair $(v,{\bf w})$ and some $(vs_t, {\bf ws_t})$ such that the length of each element increases by one.  By induction hypothesis assume that the horizontal steps of top contour of $D_j$, coming from the pair $(v, {\bf w})$, are given by ${P_j}$.  

Let $U'_{j}$ be the module associated to the box $b_{j}$ and coming from the pair $(vs_t,ws_t)$.  Let $P'_j$ denote the corresponding element of $\binom{[n]}{k}$ for the pair $(vs_t,ws_t)$.  Observe that by changing $v$ to $vs_t$ the top contour of $D'_j$ is obtained from the top contour of $D_j$ by adding a box $b'$ with entry $s_t$ to $R^*(v')_{j}$ provided that this box lies in $R(Q_{i_j})$; otherwise the top contour does not change.  If the box  $b'\in R(Q_{i_j})$ then the south and east edges of $b'$ are part of the top contour for $D_j$.  The south edge of $b'$ is a  horizontal step of $\pathne{r(D)^*_{j}}$ with label $t$, while the east edge is a vertical step with label $t+1$.  By induction hypothesis, $t\in P_j$ and $t+1\not\in P_j$.   At the same time since the contour of $D'_j$ changes by adding this box $b'$, we see that $t+1\in P'_j$ and $t\not\in P'_j$.  On the other hand, we have $P'_j=s_t (P_j)$, which precisely interchanges $t\in P_j$ for $t+1$.  Thus, we see that the two agree and in the case $b' \in R(Q_{i_j})$ the claim holds.  

\begin{figure}
\hspace*{0cm}\scalebox{1}{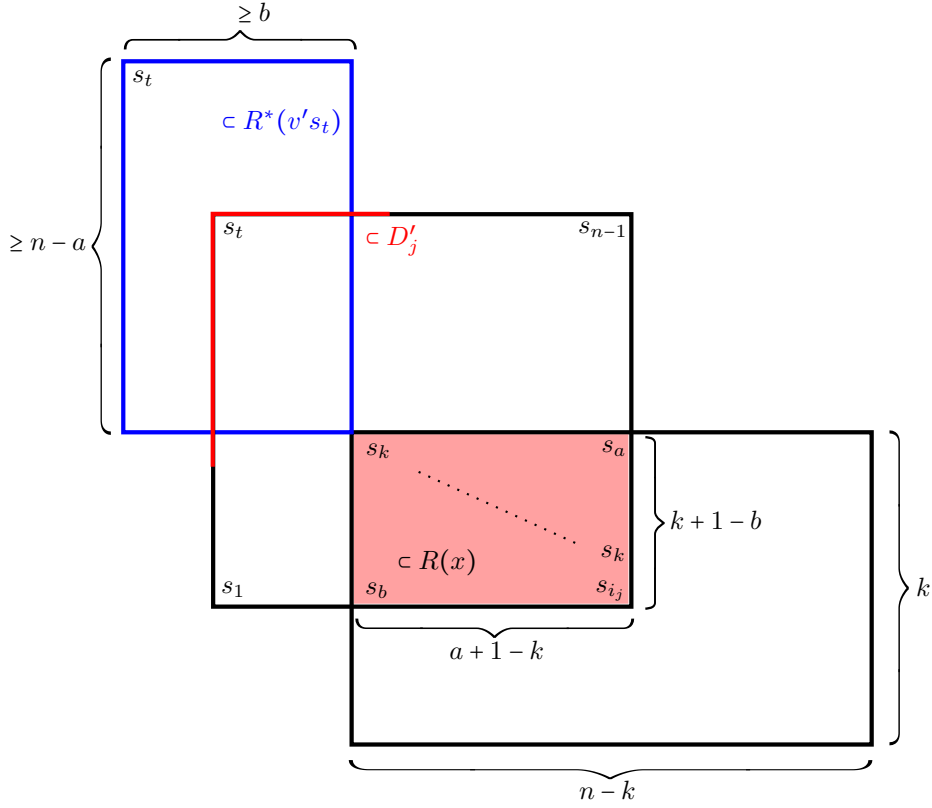}
 \caption{Inductive step for the top contour}
   \label{fig:top-contour2}
\end{figure}

Now, suppose that the box $b'\not\in R(Q_{i_j})$.  Then we know that by construction $D_j$ and $D'_j$ have the same top contour, and we want to show $P_j=P'_j$.  Clearly, if $t,t+1\in P'_j$ or $t,t+1\not\in P'_j$ then $P'_j=s_t(P_j)=P'_j$ as desired.  The remaining possibility is that $t+1\in P_j$ but $t\not\in P_j$.  Thus, we must be in the situation as depicted in \cref{fig:top-contour2}. Note that $R^*(v's_t)$, considered as a region of $D^*$, must contain a rectangle of height at least $b$ and of width at least $n-a$.  Similarly, $R(x)\in D$ must contain a rectangle of height at least $k+1-b$ and width at least $a+1-k$.  Since $D$ has height $k$ and width $n-k$, this contradicts Lemma~\ref{prop:lengthsadd} saying that $R(v's_t)\cap R(x)=\emptyset$ in $D$.  Therefore, it is not possible that $t+1\in P_j$, $t\not\in P_j$, and $b'\not\in R(Q_{i_j})$.  This completes the proof of the induction step for the top contour of $D_j$.
\smallskip

Thus, we showed that the bottom and top contour of $D_j\subset r(D)^*_{j}$ has horizontal steps given by $v^{-1}[k]$ and $P_j$ respectively.  Moreover, after rotating $r(D)^*_{j}$ clockwise 90 degrees and reflecting it across a vertical axis, we obtain the desired region $R(P_j)$, that yields the composition factor diagram for $U_j$, as a subset of $r(D)$.  For an example of this transformation see Figure~\ref{fig:ex2}.  Note that in this way a northeast lattice path $\pathne{}$ in $r(D)^*_{j}$ becomes a southwest 
path $\pathsw{}$ in $r(D)$.  Also, horizontal steps of  $\pathne{}$ become vertical steps of $\pathsw{}$.  This completes the proof of the theorem.
\end{proof}

\begin{remark}\label{rmk:irreducible}
Note the module $U_j$ is actually indecomposable because by construction of the diagram $D_j$ its bottom and top contour do not intersect, except on the boundary of $r(D)^*_j$,  
	and $R(P_j)\subset R(Q_{k})$.  In particular, the cluster-tilting module $U_{v,{\bf w}}$ is basic and the Pl\"{u}cker coordinate $\Delta_{P_j}$ is irreducible.
\end{remark}



\subsection{An explicit description of the endomorphism quiver
	$\Gamma_{U_{v,{\bf w}}}$}
\label{sec:morphisms}
In order to understand the endomorphism quiver, we need
to analyze morphisms between indecomposable summands of $U_{v,{\bf w}}$.   

Recall that for a box $b_i \in R(x)$, $U_i$ denotes 
the associated summand of $U_{v,{\bf w}}$.  Also, recall that $\Rect(b_i)$ is the maximal rectangle in $D$ whose southeast corner is $b_i$. 

\begin{theorem}
	\label{thm:morphisms}
Consider $(v,w)$ where 
	$v\in W^K_{max}$ and 
	 $w = xv$ is length-additive.
Let ${\bf w}=\mathbf{xv = xw_k v'}$ 
be a standard reduced expression for $w$.  
For any pair of modules $U_i, U_j \in \textup{ind}\,U_{v,{\bf w}}$ there exists an irreducible morphism $U_i\to U_j$ in $\textup{add}\,U_{v,{\bf w}}$ if and only if one of the following conditions holds: 
\begin{itemize}
\item[(i)] $\Rect(b_j)$ is obtained from $\Rect(b_i)$ by removing a row
\item[(ii)] $\Rect(b_j)$ is obtained from $\Rect(b_i)$ by removing a column
\item[(iii)] $\Rect(b_j)$ is obtained from $\Rect(b_i)$ by adding a hook shape.
\end{itemize}
Moreover, there exists at most one irreducible morphism between $U_i$ and  $U_j$.  
\end{theorem}

Before proving \cref{thm:morphisms}, we make the following key observation. 

\begin{remark}\label{rem-image}
Let $f: U_i\to U_j$ be a homomorphism, and 
	suppose that 
	$N$ is an indecomposable direct summand of $\text{im}\,f$.
Then because 
	$\text{im}f$ is a submodule of $U_j$ and is (isomorphic to) a quotient of $U_i$, 
	the composition diagram for $N$ embeds into those for $U_i, U_j$. 
	Moreover, 
 $N$ is \emph{closed under predecessors} in $U_i$:
	for all vertices $x$ and $y\in I$ in the composition diagrams 
	for $N$ and $U_i$, respectively, 
	such that $y$ lies immediately above $x$ in $U_i$ (that is $\begin{smallmatrix}y\\&x\end{smallmatrix}$ or $\begin{smallmatrix}&y\\x\end{smallmatrix}$) 
		we have that $y$ is also in the composition diagram for $N$. 
	Similarly, $N$ is 
	  \emph{closed under successors} in $U_j$:
		for all vertices $x,y\in I$ in the diagrams for $N,U_j$ such that $y$ lies immediately below $x$ in $U_j$ (that is  $\begin{smallmatrix}x\\&y\end{smallmatrix}$ or $\begin{smallmatrix}&x\\y\end{smallmatrix}$),
			we have that $y$ is also in the diagram for $N$.

Conversely, for any $N$ that is closed under predecessors in $U_i$ and closed under successors in $U_j$ we get a morphism $f: U_i\to U_j$ with image $N$.
\end{remark}

We will prove \cref{thm:morphisms} in two steps.  First we treat the case $v'=e$, i.e. $v=w_K$.

\begin{proposition}\label{lem:morphismsbasecase}
	\cref{thm:morphisms} is true when $v=w_K$, i.e. 
$v' = e$.
\end{proposition}

\begin{proof}
By the proof of Theorem~\ref{thm-module-structure} all indecomposable summands of $U=U_{v,{\bf w}}$ are of the form given in Figure~\ref{fig:moduleU}.  Moreover, we must have $S_k=\textup{Soc}\,U_i=\text{Soc}\,U_j$ and either $c_i+r_i =k$ or $a_i+r_i=n-k$ for any $U_i\in \text{ind}\,U$.   Thus, we can rephrase the statement of the theorem in terms of these new parameters $a_i, c_i, r_i$ that define a given summand of $U$.  Here both (ia) and (ib) correspond to case (i) of the theorem, depending if $b_i$ is above or below the main diagonal. Similar correspondences hold for the remaining cases.  

\begin{itemize}

\item[(ia)] $r_i=r_j+1$, $a_i=a_j$, and $c_i+r_i=c_j+r_j=k$ 
\item[(ib)] $r_i=r_j$, $c_i=c_j-1$, and $a_i+r_i=a_j+r_j=n-k$

\vspace{.3cm}
\item[(iia)] $r_i=r_j$, $a_i=a_j-1$, and $c_i+r_i=c_j+r_j=k$
\item[(iib)] $r_i=r_j+1$, $c_i=c_j$, and  $a_i+r_i=a_j+r_j=n-k$

\vspace{.3cm}
\item[(iiia)] $r_i=r_j-1$, $a_i=a_j+1$,  and $c_i+r_i=c_j+r_j=k$
\item[(iiib)] $r_i=r_j-1$, $c_i=c_j+1$,  and $a_i+r_i=a_j+r_j=n-k$.
\end{itemize}



By the 
	construction of the region 
	$D_j\subset \mathcal{D}_{v,w}$ (see Figure~\ref{fig:top-contour}), given $U_i \in \text{add}\,U$ defined by $a_i, r_i, c_i$ a module $U_z$ defined by $a_z,r_z,c_z$ is also in $\text{add}\,U$ if $r_z\leq a_z$ and either $a_z=a_i, c_z\geq b_i$ or $c_z=c_i, a_z\geq a_i$.  Indeed, every module in $\text{add}\,U$ corresponds to a unique box in $R(x)$.  Given a box $b_i \in R(x)$ associated to the module $U_i$, all the boxes $b_z\in D$ above and to the left of $b_i$ are also in $R(x)$.  The module $U_z$ with the above properties is precisely the one coming from such a box $b_z\in R(x)$.  Thus, $U_z\in \text{add}\,U$ as claimed. 

	Below we consider an arbitrary morphism $f:U_i\to U_j$, and using the particular structure of the modules we show that it factors through another summand $U'$ of $U$.  Moreover, we obtain two maps $U_i\to U'$ and $U'\to U_j$ whose composition is $f$ together with additional conditions on the structure of $U'$.  Since we are interested in the case when $f$ is irreducible, we can reduce $f$ to the case $U_i\to U' = U_j$ or $U'=U_i \to U_j$.  We then continue in the same way replacing $f$ by one of the two morphisms.  At every step we obtain more information about the particular structure of $U_i$ and $U_j$ until we recover the case listed in the theorem. 

\begin{figure}
\hspace*{0cm}\scalebox{1}{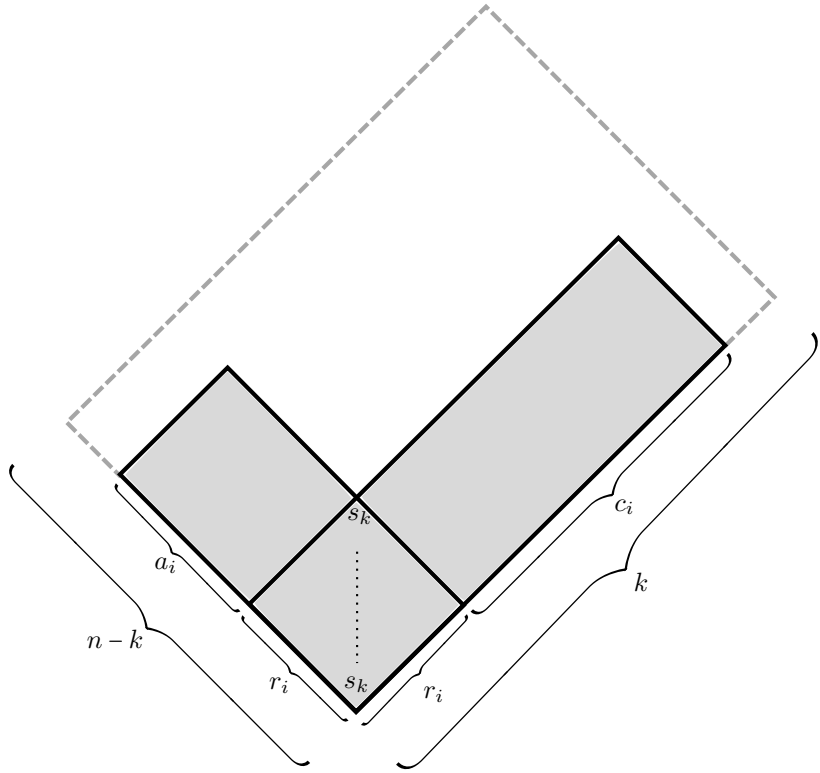}
 \caption{Module $U_i$}
   \label{fig:moduleU}
\end{figure}

Let $f: U_i\to U_j$ be a nonzero nonidentity morphism in $\text{mod}\,\Lambda$.  Since $U_j$ has a one-dimensional socle it follows that $\text{im}\,f$, which is a submodule of $U_j$, is indecomposable.  Let $N=\text{im}\,f$. By Remark~\ref{rem-image} it is closed under predecessors in $U_i$ and closed under successors in $U_j$.  Moreover, the socle of $N$ is also $S_k$, and we obtain the configuration depicted in Figure~\ref{fig:mapf}.  Here, $r_z\leq r_i, r_j$, $r_z+c'_z\leq r_j+c_j$, and $r_z+a_z\leq r_j+a_j$.  Conversely, for every such $N$ as in the figure we obtain a nonzero morphism $U_i\to U_j$.  

\begin{figure}
\hspace*{0cm}\scalebox{.8}{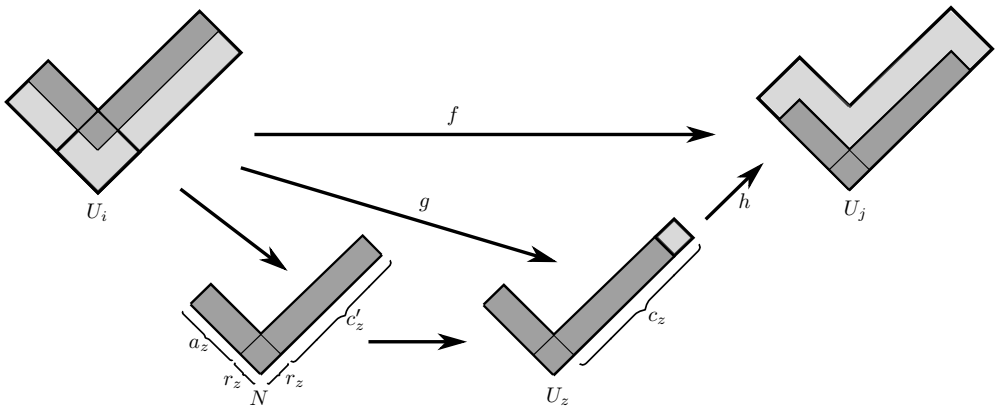}
 \caption{Morphism $f:U_i\to U_j$ with image $N$}
   \label{fig:mapf}
\end{figure}

First, we consider the case $r_i+c_i=k$ and $r_j+c_j=k$.  Note that $N$ is not necessarily in $\text{add}\,U$.  Thus, we construct a module $U_z\in\text{ind}\,\Lambda$ of the same structure as $U_i, U_j$ defined by $a_z=a_i, r_z$, and $c_z$ such that $c_z+r_z=k$.  Since $r_z\leq r_i$ and $c_z\geq c_i$ it follows that $U_z\in\text{ind}\,U$.  We also obtain maps $g:U_i\to U_z$ and $h:U_z\to U_j$  such that $f=hg$.  This implies that $f$ is reducible in $\text{add}\,U$ unless $g=1$ or $h=1$.


As we are interested in irreducible morphisms $f$, suppose first that $h=1$.  Thus, $U_z=U_j$ and $f=g$.   If $r_z=r_i$, then $U_i=U_z$ and $g=f=1$ contrary to our original assumption that $f$ is not the identity morphism.  Now, if $r_z<r_i$ consider a module $U_t$ defined by $a_t=a_i, r_t=r_z+1$ and $c_t$ such that $c_t+r_t=k$.  In particular, $c_t=c_z-1$.  Since $r_t\leq r_i$ and $c_t>c_i$ it follows that $U_t\in \text{add}\,U$.   In this case, we note that $f$ factors through $U_t$.  That is, there exist maps $\rho, \pi$ as below
$$\xymatrix@R=5pt{U_i\ar[rr]^{f=g} \ar[dr]_{\rho}&& U_z\\&U_t\ar[ur]_{\pi}}$$
such that $f=\pi\rho$. Note that by definition $\pi\not=1$ as $c_t\not=c_z$.  Since we are interested in irreducible morphisms $f$, we consider the case $\rho=1$ and $f=\pi$.   If $f=\pi$ then we have $U_i=U_t$ and $U_j=U_z$.  
By construction, $a_i=a_j$, $r_i=r_j+1$ and $c_i+r_i=c_j+r_j=k$, which agrees with case (ia). Conversely, by the structure of $U_i$ and $U_j$ it is easy to see that such $f$ is indeed irreducible in $\text{add}\,U$.

Now, consider the case $g=1$.  Thus, $h=f$ and $U_z=U_i=X$.  Let $U_q$ be the module defined by $a_q=a_z+1$, $r_q=r_z$, $c_q=c_z$, provided that $a_z+r_z<n-k$. Observe that $U_q\in \text{add}\,U$ because $r_q=r_z$ and $c_q>a_z$.  If $a_q+r_q\leq a_j+r_j$, we see that $f$ factors through $U_q$.  In particular, there exist morphisms $\sigma,\delta$ as below
$$\xymatrix@R=5pt{U_z\ar[rr]^{f=h} \ar[dr]_{\sigma}&& U_j\\&U_q\ar[ur]_{\delta}}$$
where $f=\delta\sigma$.  Since we are looking for irreducible maps we take $\delta=1$.  Note that $\sigma\not=1$ as $a_z<a_q$.  In the case $\delta=1$ we have $f=\sigma$ is injective, and $U_z=U_i$, $U_q=U_j$.  Therefore, $r_i=r_j$, $a_i=a_j-1$, and $r_i+c_i=r_j+c_j=k$, which is precisely the conditions of case (iia).   Also, because $f$ is injective, we can see by the particular structure of $U_i$ and $U_q$ that it is actually irreducible in $\text{add}\,U$.

It remains to consider the the case $f=h$ and $a_z+r_z=a_j+r_j$.  First, we observe that $r_z\not=r_j$, as otherwise $U_i=U_z=U_j$ and $f$ is the identity map.  Thus, let $U_p$ be the module defined by $r_p=r_z+1, a_p=a_z-1, c_p=c_z-1$, provided $a_z,c_z$ are both nonzero.  In this case, $U_p\in\text{add}\,U$ because $r_p\leq r_j$ and $a_p\geq a_j$.  Thus, we see that $f=h$ factors through $U_p$ 
$$\xymatrix@R=5pt{U_z\ar[rr]^{f=h} \ar[dr]_{\epsilon}&& U_j\\&U_p\ar[ur]_{\theta}}$$
where $f=\theta\epsilon$.   Note that $\epsilon\not=1$ by construction, therefore we consider the case $\theta=1$.  Thus, $f=\epsilon$ and $U_z=U_i, U_j=U_p$, where $r_i=r_j-1$, $a_i=a_i+1$, and $c_i+r_i=c_j+r_j=k$.  In particular, this agrees with case (iiia) of the lemma.  Again, since $f$ is injective it is easy to see that it is irreducible in $\text{add}\,U$.  

Finally, suppose that $f=h$ and $a_z+r_z=a_j+r_j$ as above, but $a_z=0$ or $c_z=0$.   If $a_z=0$ then $r_z=a_j+r_j$.  We also know that $U_z$ maps invectively into $U_j$ via $f$.  Therefore, $r_z\leq r_j$ which implies that $a_j=0$.  We obtain $U_z=U_i=U_j$ and $f$ is the identity morphism.   This is a contradiction.   On the other hand, if $c_z=0$ then $r_z=k$.  Since $r_z\leq r_j\leq k$, we obtain $r_j=k$.  Also, $a_j+r_j\leq k$ implies that $a_j=0$ and we deduce a contradiction as above.  

This completes the proof when $c_i+r_i=c_j+r_j=k$.  A similar argument applies in the case $a_i+r_i=a_j+r_j=n-k$.  Therefore, it remains to consider the situation when $r_i+c_i=k$ and $a_j+r_j=n-k$ while $r_i+a_i<n-k$ and $c_j+r_j<k$ and vice versa.  In particular, we want to show that every morphism in this case is reducible.  Suppose that $f: U_i\to U_j$ where $r_i+c_i=k$ and $r_j+c_j=n-k$ while $r_i+a_i<n-k$ and $c_j+r_j<k$.  The other case follows similarly.   Now, obtain a module $U_u$ defined by $r_u=r_i, a_u+r_u=n-k, c_u+r_u=k$.  Note that $U_u$ is different from both $U_i$ and $U_j$.  Moreover, $U_u\in \text{add}\,U$ because $r_u=r_i$ and $a_u>a_i$.   We obtain that $f$ factors through $U_u$.  In particular, $f$ is reducible in $\text{add}\,U$ and the resulting maps $U_i\to U_u$ and $U_u\to U_j$ are between types of modules that we considered earlier.  This shows that such $f$ does not yield any new irreducible morphisms, as desired.
\end{proof}

In the second step in the proof of Theorem~\ref{thm:morphisms}, we relate morphisms between summands of $U_{v,{\bf w}}$, and morphisms between summands of $U_{w_K,{\bf xw_K}}$ where $w=xv$ and $v\in W^K_{max}$.  

\begin{lemma}\label{morphism:ind}
Let $w=xv$, where 
$v\in W^K_{max}$
and $\ell(w)=\ell(x)+\ell(v)$.  Denote the cluster-tilting modules coming from a standard reduced expressions for the pairs $(w_K, xw_K)$ and $(v,w)$ by $U, U'$ respectively.   Let $U_i, U_j \in \text{ind}\, U$ and let $U'_i, U'_j \in \text{ind}\, U'$ be the corresponding summands of $U'$.  Then, there exists a bijection between irreducible morphisms $U_i\to U_j$ in $\textup{add}\,U$ and irreducible morphisms $U'_i\to U'_j$ in $\textup{add}\,U'$.
\end{lemma}

\begin{proof}
By \cite[Proposition 5.16]{BKT} there are equivalences of categories $\mathcal{C}_{x} \xrightarrow{\sim} \mathcal{C}_{v,w}$ and $\mathcal{C}_{x} \xrightarrow{\sim} \mathcal{C}_{w_K,xw_K}$. In particular, the categories $\mathcal{C}_{v,w}$ and $\mathcal{C}_{w_K,xw_K}$ are also equivalent.  By \cite[Remark 5.2]{Leclerc} this equivalence identifies the two cluster-tilting modules $U$ and $U'$.   In particular, this implies that there is a bijection between irreducible morphisms $U_i\to U_j$ in $\textup{add}\,U$ and irreducible morphisms $U'_i\to U'_j$ in $\textup{add}\,U'$.
\end{proof}

Together Proposition \ref{lem:morphismsbasecase} and Corollary \ref{morphism:ind} prove Theorem~\ref{thm:morphisms}.
Next, we present the main theorem of this section. 

\begin{theorem}\label{thm:seedscoincide}
Let $w=xv$ be a length additive factorization and $v\in W^K_{max}$. For a standard reduced expression ${\bf w}$ of $w$, the labeled quiver 
	$\Gamma_{U_{v,{\bf w}}}$ 
	coincides with $Q_{v, w}$. 
\end{theorem}

\begin{proof}
	By 
	\cref{def:labeledquiver} and 
	\cref{thm:morphisms}, 
	the quivers coincide.
And by the construction of $\Delta_{P_j}$ and \cref{lem:pluecker},
	the labels of the vertices coincide as well.
\end{proof}

As a corollary, we obtain \cref{prop:rectanglesseed}.


\section{The proofs of \cref{thm:main} and \cref{thm:main2}}\label{sec:proofs}

In this section we first prove \cref{thm:main2}, and then deduce \cref{thm:main} from it.

\subsection{The proof of \cref{thm:main2}}

Let $v \leq w$ be permutations where $v \in W^K_{\max}$ and $w=xv$ is a length-additive factorization. Let \textbf{w'} be a standard reduced expression for $w':=xw_K$ and let $G_{v, w}$ be the graph obtained from the bridge graph $B_{w_K, \textbf{w'}}$ by applying $v^{-1}$ to the boundary vertices. We label the faces of $G_{v, w}$ using the target labeling and let $Q_{v, w}$ be the labeled dual quiver of $G_{v, w}$ with the vertex labeled $v^{-1}([k])$ removed. So far, we have shown that $Q_{v, w}$ is the rectangles seed (\cref{cor:dualquiver}), and that $Q_{v, w}$ agrees with $\Gamma_{U_{v, \textbf{w}}}$ (\cref{thm:seedscoincide}). 

Now, let $G$ be a plabic graph obtained from $G_{v, w}$ by a sequence of moves (M1)-(M3)
. The boundary faces of $G$ have the same labels as the boundary faces of $G_{v, w}$. Let $Q$ be the dual quiver of $G$, with the vertex labeled $v^{-1}([k])$ removed. Recall that a square move at a face of a plabic graph changes the dual quiver via mutation at the corresponding vertex. So we can obtain $Q$ from $Q_{v, w}$ by a sequence of mutations. On the other hand, this same sequence of mutations can be performed on the corresponding cluster-tilting module $U_{v, {\bf w}}$ and its labeled quiver $\Gamma_{U_{v, {\bf w}}}$ resulting in a new module $U$ and its labeled quiver $\Gamma_U$. Now, labeling $Q$ with target labels, we claim that $Q=\Gamma_U$. The two quivers are clearly equal if we ignore the labels, so we only need to show that the labelings coincide. In order to do so, we first establish that the face labels of $G$ have the following property.

\begin{definition} Let $I, J \in \binom{[n]}{k}$. We say $I$ and $J$ are \emph{weakly separated} if for all $a, b \in I\setminus J$ and $c, d \in J \setminus I$ with $a<b$ and $c<d$, we never have that $a < c< b<d$ or $c < a < d<  b$. 
\end{definition}

\begin{proposition} \label{prop:rectweaklysep} Let $v \leq w$ be permutations where $v \in W^K_{\max}$ and $w=xv$ is a length-additive factorization. Let $G$ be a reduced plabic graph that can be obtained from $G_{v, w}$ by a sequence of moves (M1)-(M3). If $I, J \in \mathcal{F}_{\target}(G)$, then $I$ and $J$ are weakly separated.
\end{proposition}

\begin{proof}

Recall from \cref{lem:sameseed} that $H^{mir}_{v, w}$ is the graph obtained from $B_{w_K, \textbf{w'}}$ by reflecting in the mirror and applying $w^{-1}$ to the boundary vertices. There is a clear one-to-one correspondence between faces of $G_{v, w}$ and faces of $H^{mir}_{v, w}$, and the target labels of corresponding faces in each graph agree. Further, performing a sequence of moves to corresponding faces of $G_{v, w}$ and $H^{mir}_{v, w}$ will result in two graphs with the same target face labels. So instead of considering the plabic graph $G$, we will consider the plabic graph $H$ we obtain by performing an analogous sequence of moves to $H^{mir}_{v, w}$. 

First, we deal with the case when $w=w_0$.  From the definition of $H^{mir}_{v, w}$, $H^{mir}_{v, w_0}$ is a normal plabic graph with boundary vertices labeled $1, \dots, n$ going clockwise. It follows immediately from \cite[Theorem 1.5]{OPS} that $\mathcal{F}_{\target}(H)$ consists of pairwise weakly separated sets.

Now, suppose $w<w_0$. Note that by construction, $H^{mir}_{v, w_0}$ can be obtained from $H^{mir}_{v, w}$ by adding additional bridges. In other words, $H^{mir}_{v, w}$ is a subgraph of $H^{mir}_{v, w_0}$, whose boundary labels are inherited from the trips of $H^{mir}_{v, w_0}$. Thus, one can perform a sequence of moves to this subgraph to obtain $H$ as a subgraph of a reduced plabic graph. The weak separation of target labels of $H$ follow again from \cite[Theorem 1.5]{OPS}.

\end{proof}

This property is important because of the following lemma, which will ensure that square moves on $G_{v, w}$ correspond to valid 3-term Pl\"ucker relations. 

\begin{lemma} \label{lem:cyclicsquares} Let $G$ be a generalized plabic graph such that the elements of $\mathcal{F}_{\target}(G)$ are pairwise weakly separated, and let $f$ be a square face of $G$ whose vertices are all of degree 3. Suppose the trips coming into the vertices of $f$ are $T_{i \to a}$, $T_{j \to b}$, $T_{k \to c}$, and $T_{l\to d}$ reading clockwise around $f$ (see \cref{fig:mut-lemma}). Then $a, b, c, d$ are cyclically ordered.
\end{lemma}

\begin{proof}
Consulting \cref{fig:mut-lemma}, the target labels of faces around $f$ are $Rab, Rbc, Rcd, Rad$, where $R$ is some $(k-2)$-element subset of $[n]$ and $Rab:=R \cup \{a, b\}$. The fact that $Rad$ and $Rbc$ are weakly separated implies that either $a, b, c, d$ or $a, c, b, d$ is cyclically ordered. The fact that $Rab$ and $Rcd$ are weakly separated implies that the former is true.
\end{proof}

We can now show that if $G$ is a generalized plabic graph move-equivalent to $G_{v, w}$, square moves on $G$ agree with the categorical mutation of modules in $\mathcal{C}_{v, w}$. This, together with \cref{thm:seedscoincide}, completes the proof of \cref{thm:main2}.

\begin{lemma}\label{mutation}
Let $G$ be a reduced plabic graph that is move-equivalent to $G_{v, w}$.  Suppose that the (target) labeled quiver $Q(G)=\Gamma_{U}$, for some cluster-tilting module $U\in\mathcal{C}_{v, w}$.   If $G'$ is obtained from $G$ by performing a square move at some face $F$ of $G$,  then 
$$Q(G')=\Gamma_{U'}$$
as labeled quivers, where $U'$ denotes the mutation of $U$ at the corresponding indecomposable summand $U_F$ of $U$.
\end{lemma}

\begin{proof}

The label of the square face $F$ and its surrounding faces are given in Figure~\ref{fig:mut-lemma}.   Here, $R$ is a $(k-2)$-element subset of $[n]$ and $Rac$ stands for $R\cup\{a,c\}$.  
Thus, $F$ has label $Rac$ in $G$ and after the mutation it has label $Rbd$. By \cref{prop:rectweaklysep}, the target face labels of $G$ are pairwise weakly separated, so by \cref{lem:cyclicsquares}, $a,b,c,d$ are cyclically ordered.  
Now, consider the local configuration in $\Gamma_U$ around the vertex $\Delta_{Rac}$ corresponding to the summand $U_F$ of $U$.   By definition of mutation, $U'=U/U_F\oplus U_F'$, where $U_F'$ is defined by the two short exact sequences as follows. 
$$\xymatrix@C=15pt{0\ar[r]&U_F' \ar[r] & U_{Rbc}\oplus U_{Rad} \ar[r] & U_F \ar[r] & 0 && 0\ar[r]&U_F \ar[r] & U_{Rab}\oplus U_{Rcd} \ar[r] & U_F' \ar[r] & 0}$$
where we identify summand of $U$ with the labels of the corresponding faces in $G$. 
By the properties of the cluster-character map $\varphi$ this yields the relation
$$\varphi_{U_F}\varphi_{U_F'} = \varphi_{U_{Rbc}}\varphi_{U_{Rad}}+\varphi_{U_{Rab}}\varphi_{U_{Rcd}}.$$
Note that if one of the faces adjacent to $F$ has label $v^{-1}([k])$ then the associated module $U_{v^{-1}([k])}$ is the zero module and $\varphi_{U_{v^{-1}([k])}} = \Delta_{v^{-1}([k])}=1$ by Remark~\ref{rem:lex_min}.   In this case, the relation above still holds.
Since the two labeled quivers $Q(G)$ and $\Gamma_U$ coincide, each function $\varphi_{U_E}\in \mathbb{C}[\mathcal{R}_{w_k, w}]$, where $E$ is a face in $G$, is simply a Pl\"ucker coordinate coming from the label of the face. 
In particular, we have the following.
$$\varphi_{U_F}=\varphi_{Rac}=\Delta_{Rac}  \hspace{.5cm}
\varphi_{U_{Rab}}=\Delta_{Rab}\hspace{.5cm}
\varphi_{U_{Rbc}}=\Delta_{Rbc}\hspace{.5cm}
\varphi_{U_{Rcd}}=\Delta_{Rcd}\hspace{.5cm}
\varphi_{U_{Rad}}=\Delta_{Rad}$$
Therefore, the relation above becomes 
$$\Delta_{Rac}\varphi_{U_F'} = \Delta_{{Rbc}}\Delta_{{Rad}}+\Delta_{{Rab}}\Delta_{{Rcd}}$$
which is precisely a three-term Pl\"ucker relation in the corresponding skew Schubert variety.  Thus, we conclude that $\varphi_{U_F'}=\Delta_{Rbd}$.  
This shows that the two labeled quivers $Q(G')$ and $\Gamma_{U'}$ agree.   
\end{proof}

\begin{remark} \label{rmk:manymutate} Since all graphs in \cref{lem:sameseed} give rise to the same labeled seed (up to reversing all arrows in the quiver, which does not affect mutation), and a sequence of moves on any one can be translated to a sequence of moves on any other that effects the dual quiver in the same way , \cref{mutation} shows that any reduced plabic graph move-equivalent to a graph in \cref{lem:sameseed} gives rise to a seed for ${\pi_k(\mathcal{R}_{v, w}})$.
\end{remark}

\begin{figure}
\centering
\includegraphics[width=\textwidth]{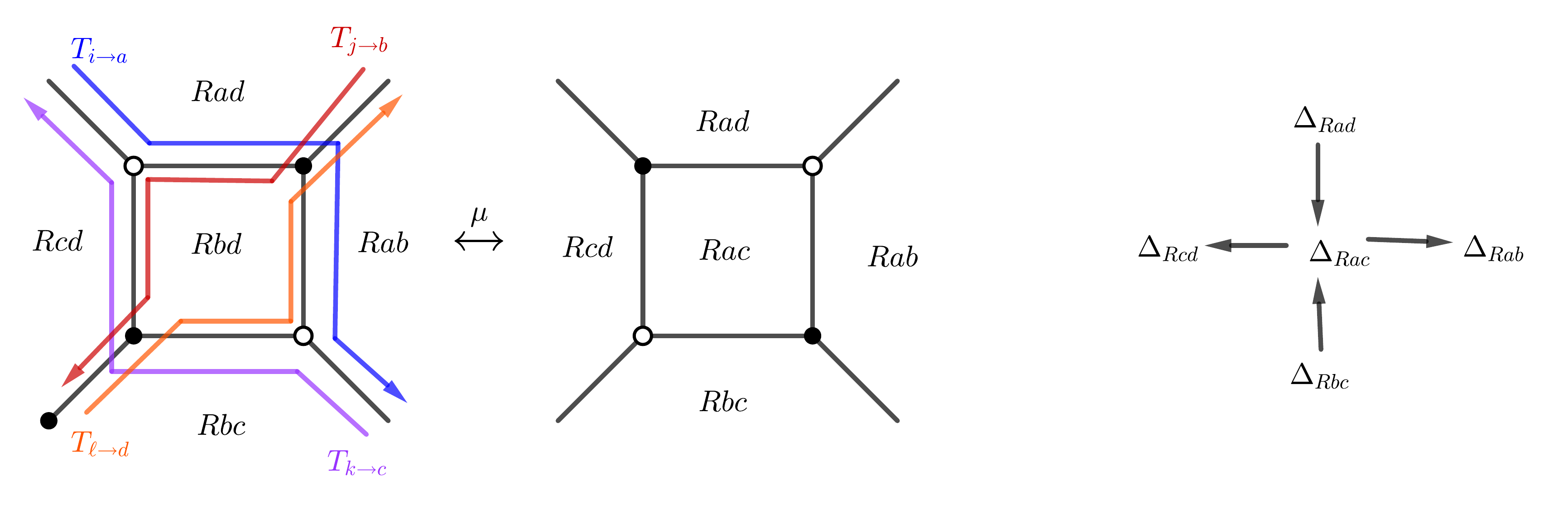}
 \caption{Plabic graphs $G'$ and $G$ respectively, and the labeled quiver $Q(G)$}
   \label{fig:mut-lemma}
\end{figure}

\subsection{The proof of \cref{thm:main}}\label{section:1.5}

We now explain how to deduce \cref{thm:main} from \cref{thm:main2}.

Recall that for $v \in W^K_{\max}$, ${\pi_k(\mathcal{R}_{v, w_0})} = X^\circ_\lambda$, where $\vertsw{\lambda}=v^{-1}([k])$. The decorated permutation corresponding to $\pi_k(\mathcal{R}_{v, w_0})$ is $v^{-1}w_0$. 

Recall also that we can obtain $v^{-1}$ in list notation from $\lambda$ by labeling the southeast border of $\lambda$ with $1, \dots, n$ going southwest and first reading the labels of vertical steps going northeast and then reading the labels of the horizontal steps going northeast. To obtain $v^{-1}w_0$, we reverse the order in which we read the border of $\lambda$, first reading the labels of horizontal steps going southwest and then reading the labels of the vertical steps going southwest. So $v^{-1}w_0$ is equal to the permutation $\ppermsw{\lambda}$ appearing in \cref{thm:main}.

Let $x:=w_0 v^{-1}$. The factorization $w_0=xv$ is length-additive. Let $\textbf{w'}$ be a standard reduced expression for $w':=xw_K$. If we take $B_{w_K, \textbf{w'}}$, apply $w_0^{-1}$ to the boundary vertices, and ``reflect in the mirror", we obtain a graph $H^{mir}_{v, w_0}$ which has trip permutation $\ppermsw{\lambda}$ and whose boundary vertices are labeled with $1, \dots, n$ clockwise. According to \cref{thm:seedscoincide} and \cref{lem:sameseed}, if we label the dual quiver of $H^{mir}_{v, w_0}$ using target labels, we obtain a seed for the coordinate ring of (the affine cone over) $X^\circ_\lambda$. And by  \cref{rmk:manymutate}, if $G$ is any reduced plabic graph move-equivalent to $H^{mir}_{v, w_0}$ (that is, with boundary vertices labeled $1, \dots, n$ clockwise and trip permutation $\ppermsw{\lambda}$), then the (target) labeled dual quiver $Q(G)$ gives a seed.

%

\section{Applications}\label{sec:conclusion}

In this section we give applications of \cref{thm:main} and \cref{thm:main2}.

\subsection{The coordinate rings of Schubert and skew-Schubert varieties}\label{sec:skew}

Combining \cref{thm:main} and \cref{thm:main2} with
\cite[Theorem 3.3]{MullerSpeyer0} and \cite{Muller}, we obtain the following corollary.

\begin{cor}
Let $v \leq w$, where $v \in W^K_{\max}$ and $w=xv$ is length-additive. Then the cluster algebra $\C[\widehat{\pi_k(\mathcal{R}_{v, w})}]$ is locally acyclic, and thus is finitely generated, normal, locally a complete
intersection, and equal to its own upper cluster algebra.
\end{cor}


Combining our result with \cite[Theorem 1.2]{FordSer}, 
we find that the quivers giving
rise to the cluster structures for Schubert and skew Schubert varieties admit
\emph{green-to-red sequences}, which by \cite{GHKK} implies that the cluster algebras have
\emph{Enough Global Monomials}. Hence, we have the following corollary. 

\begin{cor} Let $v \leq w$, where $v \in W^K_{\max}$ and $w=xv$ is length-additive. Then the cluster algebra $\C[\widehat{\pi_k(\mathcal{R}_{v, w})}]$ has a canonical basis of
\emph{theta functions}, parameterized by the lattice of $g$-vectors.
\end{cor}

\subsection{Skew Schubert varieties whose cluster structure has finite type}

%
%

In \cite{Scott}, Scott classified the Grassmannians whose coordinate rings
have a cluster algebra of finite type.  He showed that in general the
cluster algebras have infinite type, except in the following cases:
the coordinate ring of $Gr_{2,n}$ has a cluster algebra of type $A_{n-3}$,
while the coordinate rings of $Gr_{3,6}$, $Gr_{3,7},$ and $Gr_{3,8}$ have
cluster algebras of types $D_4$, $E_6$, and $E_8$, respectively.

It is straightforward to classify for which skew Schubert varieties ${\pi_k(\mathcal{R}_{v, w})}$ the cluster structure described here is finite type. It depends only on $wv^{-1}$. We will need the following two facts. 

\begin{proposition}[\cite{CalKel}] \label{prop:sametree} Let $Q$ and $Q'$ be orientations of trees $T$ and $T'$, respectively. If $Q$ can be obtained from $Q'$ by a sequence of mutations, then $T$ and $T'$ are isomorphic.
\end{proposition}

\begin{lemma}[\protect{\cite[Remark 5.10.9]{introCA}}]\label{lem:infinitesubquiver} Let $Q$ be a quiver and let $Q'$ be a subquiver of $Q$ consisting of some vertices of $Q$, which inherit being frozen or mutable from $Q$, and all arrows between them. Then if $Q$ is mutation equivalent to a (disjoint union of) type ADE Dynkin diagram, so is $Q'$.
\end{lemma}

\begin{proposition} Let $v \leq w$, where $v \in W^K_{\max}$ and $w=xv$ is length-additive. Let $\lambda=\partne{x([k])}$ and let $\lambda'$ be the diagram obtained from $\lambda$ by removing all boxes that touch the southeast boundary of $\lambda$. Then the cluster algebra 
	$\mathcal{A} = \C[\widehat{\pi_k(\mathcal{R}_{v, w})}]$ given in \cref{thm:main2} is

\begin{enumerate}
\item type $A$ if and only if $\lambda'$ does not contain a $2 \times 2$ rectangle;
\item type $D$ if and only if $\lambda'=(i, 2)$ or its transpose for $i \geq 2$;
\item type $E_6$, $E_7$, or $E_8$ if and only if $\lambda'$ or its transpose 
	is one of $(3, 3)$, $(3, 2, 1)$, $(4, 3)$, $(4, 2, 1)$, $(3, 3, 1)$, $(5, 3)$, $(5, 2, 1)$, $(4, 4)$, $(4, 2, 2)$.
\end{enumerate}
 
In particular, the cluster algebra associated to 
	the Schubert variety $X_{\lambda}$ is of finite type if and only if $\lambda'$ is in the above list.
\end{proposition}

\ytableausetup{boxsize=1em}

\begin{figure}[h]
\centering
\begin{tabular}{l l | l l | l l}
Type $E_6$ & & Type $E_7$ & & Type $E_8$ & \\
 \noalign{\smallskip}
 \hline
 \noalign{\smallskip}
 $\ydiagram[*(gray)]{3, 3}*[*(white)]{3+1, 3+1, 4}$
  & $\raisebox{1em}{\xymatrix@R=.25cm@C=0.25cm{ \ar[dr] \bullet & \ar[l] \ar[dr] \bullet& \ar[l] \bullet\\
    \ar[u] \bullet & \ar[u] \ar[l]  \bullet & \ar[u] \ar[l] \bullet}}$
  & $\ydiagram[*(gray)]{4,3}*[*(white)]{4+1, 3+2, 4}$
  &  $\raisebox{1em}{\xymatrix@R=.25cm@C=0.25cm{ \ar[dr] \bullet & \ar[l] \ar[dr] \bullet& \ar[l] \bullet & \ar[l] \bullet\\
    \ar[u] \bullet & \ar[u] \ar[l]  \bullet & \ar[u] \ar[l] \bullet}}$
  &  $\ydiagram[*(gray)]{5,3}*[*(white)]{5+1, 3+3,4}$
  & $\raisebox{1em}{\xymatrix@R=.25cm@C=0.25cm{ \ar[dr] \bullet & \ar[l] \ar[dr] \bullet& \ar[l] \bullet & \ar[l] \bullet & \ar[l] \bullet\\
    \ar[u] \bullet & \ar[u] \ar[l]  \bullet & \ar[u] \ar[l] \bullet}}$ \\
   \noalign{\smallskip}
 \hline
 \noalign{\smallskip}
 $\ydiagram[*(gray)]{3, 2, 1}*[*(white)]{3+1, 2+2, 1+2, 2}$
  & $\raisebox{1em}{\xymatrix@R=.25cm@C=0.25cm{ \ar[dr] \bullet & \ar[l] \bullet& \ar[l] \bullet \\
    \ar[u] \bullet & \ar[u] \ar[l]  \bullet \\
    \ar[u] \bullet}}$
  & $\ydiagram[*(gray)]{4,2, 1}*[*(white)]{4+1,2+3, 1+2, 2}$
  & $\raisebox{1em}{\xymatrix@R=.25cm@C=0.25cm{ \ar[dr] \bullet & \ar[l] \bullet& \ar[l] \bullet & \ar[l] \bullet\\
    \ar[u] \bullet & \ar[u] \ar[l]  \bullet \\
    \ar[u] \bullet}}$
  &  $\ydiagram[*(gray)]{5,2, 1}*[*(white)]{5+1,2+4, 1+2, 2}$
  & $\raisebox{1em}{\xymatrix@R=.25cm@C=0.25cm{ \ar[dr] \bullet & \ar[l] \bullet& \ar[l] \bullet & \ar[l] \bullet & \ar[l] \bullet\\
    \ar[u] \bullet & \ar[u] \ar[l]  \bullet \\
    \ar[u] \bullet}}$\\
   \noalign{\smallskip}
 \hline
 \noalign{\smallskip}
&
& $\ydiagram[*(gray)]{3, 3, 1}*[*(white)]{3+1, 3+1, 1+3, 2}$
& $\raisebox{1em}{\xymatrix@R=.25cm@C=0.25cm{ \ar[dr] \bullet & \ar[l] \ar[dr] \bullet& \ar[l] \bullet\\
    \ar[u] \bullet & \ar[u] \ar[l]  \bullet & \ar[u] \ar[l] \bullet\\
    \ar[u] \bullet}}$
& $\ydiagram[*(gray)]{4, 4}*[*(white)]{4+1, 4+1, 5}$
& $\raisebox{1em}{\xymatrix@R=.25cm@C=0.25cm{ \ar[dr] \bullet & \ar[l] \ar[dr] \bullet& \ar[l] \ar[dr]\bullet & \ar[l] \bullet\\
    \ar[u] \bullet & \ar[u] \ar[l]  \bullet & \ar[u] \ar[l] \bullet & \ar[u] \ar[l] \bullet}}$\\
 \noalign{\smallskip}
 \hline
 \noalign{\smallskip}
&
&
&
&$\ydiagram[*(gray)]{3, 3, 1, 1}*[*(white)]{3+1,3+1, 1+3, 1+1, 2}$
& $\raisebox{1em}{\xymatrix@R=.25cm@C=0.25cm{ \ar[dr] \bullet & \ar[l] \ar[dr] \bullet& \ar[l] \bullet\\
    \ar[u] \bullet & \ar[u] \ar[l]  \bullet & \ar[u] \ar[l] \bullet\\
    \ar[u] \bullet\\
    \ar[u] \bullet}}$ \\
\end{tabular}
\caption{\label{fig:exceptionaltypes}Up to transposition, the smallest partitions giving rise to quivers of types $E_6, E_7, E_8$, whose mutable parts are shown on the right. The boxes corresponding to mutable vertices are shaded. Adding any number of boxes to the first row or column of these partitions only adds isolated frozen vertices to the quiver, and so also gives rise to a quiver of type $E_6, E_7, E_8$.}
\end{figure}
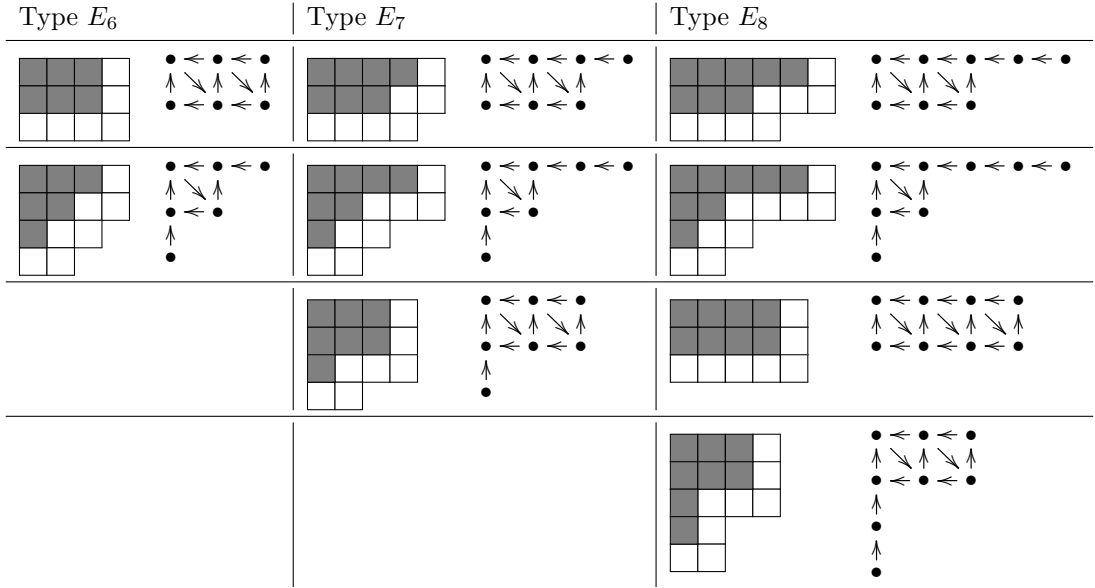

\begin{proof}
\begin{enumerate}
\item The backwards direction follows from the fact that if $\lambda'$ does not contain a $2 \times 2$ rectangle, then $Q_{v, w}$ is an orientation of a path. For the other direction, recall that the mutable part of all type A quivers can be obtained from a triangulation of a polygon \cite[Lemma 5.3.1]{introCA}. It is not hard to see that there is no arrangement of 4 arcs in a triangulation that gives the quiver we draw from a $2 \times 2$ rectangle according to \cref{def:labeledquiver}, so if $\lambda'$ contains a $2 \times 2$ rectangle as a subdiagram, $\A$ is not type $A$.
\item The backwards direction follows from inspection of the associated quivers; if one mutates at the vertex in the northwest box, one obtains an orientation of a type $D$ Dynkin diagram. If $|\lambda'|\leq 8$, then necessity follows from direct computation and \cref{prop:sametree}. Four partitions of $8$ are not finite type (see \cref{fig:mininfinite}), so by \cref{lem:infinitesubquiver} any partition containing one of these four will not be finite type. The partitions of $9$ that are not of type $A$, are not $(7, 2)$ or its transpose, and do not contain a partition of 8 that is infinite type are shown in \cref{fig:mininfinite}; they are all infinite type. Thus, the only partitions of $9$ that are finite type and not type $A$ are $(7, 2)$ and its transpose. From this, we can conclude that $\A$ is type $D$ only if $\lambda'=(i, 2)$ or its tranpose. Indeed, $\A$ is infinte type if $\lambda'$ is not type $A$ and contains any partition of $9$ that is not $(7, 2)$ or its transpose, or, equivalently, if $\lambda' \neq(i, 2)$ or its transpose. 
\item By direct computation, using \cref{prop:sametree}.
\end{enumerate}
\end{proof}

\begin{figure}[h]
\centering
\begin{tabular}{l  l | l  l}
 $\ydiagram[*(gray)]{4,3,1}*[*(white)]{4+1, 3+2, 1+3, 2}$ & 
 $\raisebox{1em}{\xymatrix@R=.25cm@C=0.25cm{ \ar[dr] \bullet & \ar[l] \ar[dr] \bullet& \ar[l] \bullet & \ar[l] \bullet\\
    \ar[u] \bullet & \ar[u] \ar[l]  \bullet & \ar[u] \ar[l] \bullet \\
    \ar[u] \bullet}}$
  & $\ydiagram[*(gray)]{5, 4}*[*(white)]{5+1, 4+2, 5}$ &
   $\raisebox{1em}{\xymatrix@R=.25cm@C=0.25cm{ \ar[dr] \bullet & \ar[l] \ar[dr] \bullet& \ar[l] \ar[dr] \bullet& \ar[l] \bullet & \ar[l] \bullet\\
    \ar[u] \bullet & \ar[u] \ar[l]  \bullet & \ar[u] \ar[l] \bullet &\ar[u] \ar[l] \bullet \\}}$
  \\ 
 \noalign{\smallskip}
 \hline
 \noalign{\smallskip}
 $\ydiagram[*(gray)]{4,2,1, 1}*[*(white)]{4+1, 2+3, 1+2, 1+1, 2}$ & 
  $\raisebox{1em}{\xymatrix@R=.25cm@C=0.25cm{ \ar[dr] \bullet & \ar[l] \bullet& \ar[l] \bullet & \ar[l] \bullet\\
    \ar[u] \bullet & \ar[u] \ar[l]  \bullet  \\
    \ar[u] \bullet&\\
    \ar[u] \bullet}}$
  &$ \ydiagram[*(gray)]{6, 3}*[*(white)]{6+1, 3+4, 4}$ 
  &  $\raisebox{1em}{\xymatrix@R=.25cm@C=0.25cm{ \ar[dr] \bullet & \ar[l] \ar[dr] \bullet& \ar[l] \bullet & \ar[l] \bullet & \ar[l] \bullet & \ar[l] \bullet\\
    \ar[u] \bullet & \ar[u] \ar[l]  \bullet & \ar[u] \ar[l] \bullet }}$ \\
  \noalign{\smallskip}
 \hline
 \noalign{\smallskip}
$\ydiagram[*(gray)]{3, 3, 2}*[*(white)]{3+1, 3+1, 2+2, 3}$
& 	$\raisebox{1em}{\xymatrix@R=.25cm@C=0.25cm{ \ar[dr] \bullet & \ar[l] \ar[dr] \bullet& \ar[l] \bullet \\
    		\ar[u] \ar[dr] \bullet & \ar[u] \ar[l]  \bullet & \ar[u] \ar[l] \bullet \\
    		\ar[u] \bullet & \ar[u] \ar[l]  \bullet}}$
 & 	$\ydiagram[*(gray)]{6, 2, 1}*[*(white)]{6+1, 2+5, 1+2, 2}$
 & $\raisebox{1em}{\xymatrix@R=.25cm@C=0.25cm{ \ar[dr] \bullet & \ar[l] \bullet& \ar[l] \bullet & \ar[l] \bullet & \ar[l] \bullet\\
    \ar[u] \bullet & \ar[u] \ar[l]  \bullet  \\
    \ar[u] \bullet}}$
   \\
 \noalign{\smallskip}
 \hline
 \noalign{\smallskip}
&& $\ydiagram[*(gray)]{3, 3, 1, 1, 1}*[*(white)]{3+1, 3+1, 1+3, 1+1, 1+1, 2}$ & 
$\raisebox{1em}{\xymatrix@R=.25cm@C=0.25cm{ \ar[dr] \bullet & \ar[l] \ar[dr] \bullet& \ar[l] \bullet &\\
    \ar[u] \bullet & \ar[u] \ar[l]  \bullet & \ar[u] \ar[l] \bullet \\
    \ar[u] \bullet \\
    \ar[u] \bullet \\
    \ar[u] \bullet }}$

\end{tabular}
\caption{\label{fig:mininfinite} These partitions (and their transposes) are the smallest partitions giving rise to quivers of infinite type, whose mutable parts are shown on the right. The boxes corresponding to mutable vertices are shown in green. Adding any number of boxes to the first row or column of these partitions only adds isolated frozen vertices to the quiver, and so also gives rise to a quiver of infinite type.}
\end{figure}
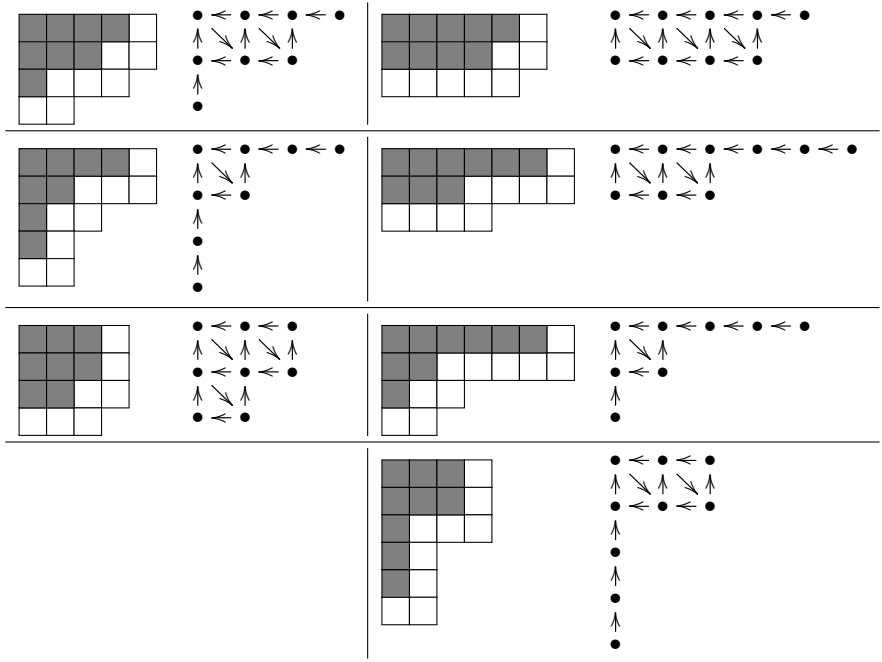

%


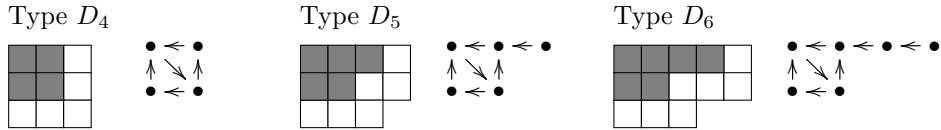
\begin{figure}[h]
\centering
\begin{tabular}{l l l l l l l l}
Type $D_4$ & && Type $D_5$ & &&Type $D_6$\\
$\ydiagram[*(gray)]{2, 2}*[*(white)]{2+1, 2+1, 3}$
  & $\raisebox{1em}{\xymatrix@R=.25cm@C=0.25cm{ \ar[dr] \bullet & \ar[l] \bullet& \\
    \ar[u] \bullet & \ar[u] \ar[l]  \bullet}}$
&&	$\ydiagram[*(gray)]{3, 2}*[*(white)]{3+1, 2+2, 3}$
  & $\raisebox{1em}{\xymatrix@R=.25cm@C=0.25cm{ \ar[dr] \bullet & \ar[l] \bullet&  \ar[l] \bullet\\
    \ar[u] \bullet & \ar[u] \ar[l]  \bullet}}$
&&$\ydiagram[*(gray)]{4, 2}*[*(white)]{4+1, 2+3, 3}$
  & $\raisebox{1em}{\xymatrix@R=.25cm@C=0.25cm{ \ar[dr] \bullet & \ar[l] \bullet& \ar[l] \bullet & \ar[l] \bullet\\
    \ar[u] \bullet & \ar[u] \ar[l]  \bullet}}$
\end{tabular}
\caption{A series of Schubert varieties which yield the type 
	$D_n$ cluster algebras.}
\label{FinType}
\end{figure}

\subsection{Applications to the preprojective algebra}

As an application of Theorem~\ref{thm-module-structure}, we obtain an explicit way to compute the summands of a cluster-tilting module $U_{v,w}$, whereas 
Leclerc's definition is constructive.  This provides a novel connection between Pl{\"u}cker coordinates and the structure of the summands of $U_{v,{\bf w}}$.  It is an interesting problem to determine whether this correspondence extends beyond the case of Schubert and skew-Schubert varieties.   Such a combinatorial interpretation of the modules would be useful in computing morphisms between the summands of $U_{v,{\bf w}}$ for arbitrary $(v,{\bf w})$.  Moreover, given two modules $U, U' \in \mathcal{C}_{v,w}$ that correspond to Pl{\"ucker} coordinates $\Delta_{P}, \Delta_{P'}$ on the positroid variety ${\pi_k(\mathcal{R}_{v,w})}$, it is natural to ask whether we can detect an extension between $U$ and $U'$ in terms of the corresponding lattice paths $\pathsw{P}, \pathsw{P'}$.  In particular, this would tell us whether two cluster variables $\Delta_{P}, \Delta_{P'}$ are compatible in the cluster algebra 
$\mathbb{C}[\widehat{\pi_k(\mathcal{R}_{v,w})}]$.  
This could
provide new insights into the representation theory of 
preprojective algebras.

Moreover, when $w=xv$ is  length additive and $v\in W^K_{max}$, we can 
explicitly write down 
many of the seeds for the pair $(v,w)$ using the combinatorics of 
plabic graphs.  
Thus we find that these cluster algebras have all the nice properties
 mentioned in \cref{sec:skew} (they are locally acyclic, equal to their upper cluster algebra,
 admit green-to-red sequences, have a canonical
 basis of theta functions, etc).

\appendix
\section{Skew Schubert varieties}\label{sec:whichpositroids}

Besides decorated permutations, $\Le$-diagrams are another combinatorial
object indexing positroid varieties.  In this section we will 
 give a recipe for the \Le-diagrams of the 
\emph{skew Schubert varieties}, i.e. the 
positroids of the form $\pi_k(\mathcal{R}_{v, w})$, where
$v\in W^K_{max}$ and 
$w$ has a length-additive factorization $xv$. 
Recall that in this case, $x \in {^KW}$.
The trip permutation of such a positroid is $v^{-1}xv$. 
While we do not know a combinatorial characterization of these trip permutations, we can describe the corresponding \Le-diagrams. 

\subsection{The \protect\Le-diagrams associated to skew Schubert varieties}
We first need some preliminary notions, following \cite{LW}. 

For a Young diagram $\lambda$ that fits inside of a $k \times (n-k)$ rectangle, let $\permne{\lambda} \in {^K W}$ be the Grassmannian permutation of type $(k,n)$ with $\permne{\lambda}([k])= \vertne{\lambda}$.

\begin{defn} An \emph{$\oplus$-diagram} (``o-plus diagram") $O$ of shape $\lambda$ is a Young diagram $\lambda$ that has been filled with $0$'s and $+$'s. We say $O$ is of \emph{type (k, n)} if $\lambda$ fits into a $k \times (n-k)$ rectangle. An $\oplus$-diagram is a \emph{\Le-diagram} (``Le-diagram") if the ``\Le-property" holds: there is no $0$ such that there is a $+$ above it in the same column and a $+$ to its left in the same row (see \cref{fig:leprop}).
\end{defn}

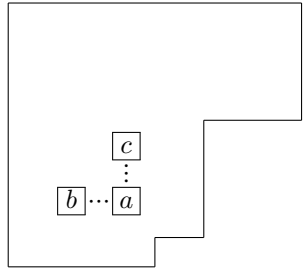
\begin{figure}
\centering
\setlength{\unitlength}{1.3mm}
\begin{picture}(35, 35)
\put(5, 32){\line(1, 0){30}}
\put(25, 20){\line(1, 0){10}}
\put(20, 8){\line(1,0){5}}
\put(5, 5){\line(1, 0){15}}
\put(5, 5){\line(0, 1){27}}
\put(20, 5){\line(0, 1){3}}
\put(25, 8){\line(0, 1){12}}
\put(35, 20){\line(0, 1){12}}
\put(10, 16){$\begin{ytableau}
 \none & \none & c \\
\none & \none & \none[\vdots]\\
 b & \none[\cdots] & a 
 \end{ytableau}$}
\end{picture}
\caption{\label{fig:leprop} The \protect\Le- property: if $b, c=+$ then $a=+$.}
\end{figure}

Suppose $\lambda$ fits inside of a $k \times (n-k)$ rectangle. By \cref{lem:redexpression}, given a reading order, we can obtain a reduced expression $\textbf{u}$ for $\permne{\lambda}$ from $\lambda$. Fixing a reading order, each $\oplus$-diagram $O$ of shape $\lambda$ gives a subexpression $\textbf{r}$ of $\textbf{u}$, obtained by replacing each simple transposition in a box filled with a $+$ by a 1. The permutation $r$ given by this subexpression does not depend on the reading order \cite[Proposition 4.6]{LW}, so we will denote it by $r(O)$ (for the ``reading word" of $O$). 

Note that by \cite[Lemma 19.3]{Postnikov}, $O$ is a \Le-diagram if and only if $\mathbf{r}$ is a \emph{positive distinguished subexpression} of $\mathbf{u}$ (see \cref{def:posdistsubexpression}).

\begin{prop} \label{prop:lebiject} Let $M$ be a \Le-diagram of shape $\lambda$ with reading word $r$, and let $u=\permne{\lambda}$. Then $M$ corresponds to the positroid variety $\pi_k(\mathcal{R}_{u^{-1}w_0, r^{-1}w_0})$
\end{prop}

\begin{proof} By \cite[Theorem 19.1]{Postnikov}, $M$ corresponds to $\pi_k(\mathcal{R}_{r^{-1}, u^{-1}})$. Note that in the compete flag variety, the map $Bx \to Bxw_0$ gives an isomorphism between $\mathcal{R}_{v, w}$ and $\mathcal{R}_{ww_0, vw_0}$. The proposition follows immediately.
\end{proof}

\begin{remark} Let $v \in W^K_{\max}$ and $u \in W^K_{\min}$. The \Le-diagram of ${\pi_k(\mathcal{R}_{v, w_0})}\cong X^\circ_{v^{-1}([k])}$ has shape $\partsw{v^{-1}([k])}$ and every box contains a $+$. For $u\in W^K_{min}$, the \Le-diagram of 
	${\pi_k(\mathcal{R}_{w_K, w_K u})} \cong (X^{u^{-1}([k])})^\circ$ is the $k \times (n-k)$ rectangle where all boxes above $\pathsw{u^{-1}([k])}$ contain $0$'s and all boxes below contain $+$'s. 
\end{remark}

We can use \emph{\Le-moves} to change $\oplus$-diagrams into $\Le$-diagrams. 

\begin{defn} \cite[Section 5]{LW} Suppose $O$ is an $\oplus$-diagram containing a rectangular subdiagram where all non-corner boxes are filled with zeros and the northeast and southwest corners are filled with pluses (shown below). If $b$ is 0, a \emph{\Le-move} changes $b$ to $+$ and changes $a$ either from $0$ to $+$ or from $+$ to $0$.
\[\begin{ytableau}
a & 0 & 0 & 0 & +\\
0 & 0 & 0 & 0 & 0 \\
+ & 0 & 0 & 0 & b\\
\end{ytableau}
\]

\end{defn}

Note that these are actually the \emph{rectangular} \Le-moves of \cite[Definition 4.11]{LW}. 

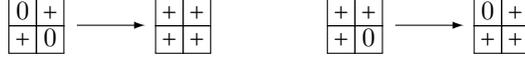
\begin{figure}
\setlength{\unitlength}{1.3mm}
\centering
\begin{picture}(30, 10)
\put(5, 8){$\begin{ytableau}
0 & +\\
+ & 0
\end{ytableau}$}
\put(12, 8){\vector(1, 0){7}}
\put(20, 8){$\begin{ytableau}
+ & +\\
+ & +
\end{ytableau}$}
\end{picture}
\hspace{.1cm}
\begin{picture}(30, 10)
\put(5, 8){$\begin{ytableau}
+ & +\\
+ & 0
\end{ytableau}$}
\put(12, 8){\vector(1, 0){7}}
\put(20, 8){$\begin{ytableau}
0 & +\\
+ & +
\end{ytableau}$}
\end{picture}
\caption{\label{fig:lemoves}Two \protect\Le-moves.}
\end{figure}

See \cref{fig:lemoves} for examples of \Le-moves.

The key properties of $\Le$-moves are as follows.
\begin{lemma}\cite[Lemma 4.13, Proposition 4.14]{LW} Let $O$ be an $\oplus$-diagram.
\begin{enumerate}
\item $O$ can be made into a $\Le$-diagram $M$ (``{\Le}-ified") by a finite sequence of \Le-moves. 
\item If $O'$ is related to $O$ by a sequence of $\Le$-moves, then $r(O)=r(O')$.
\item $M=:M(O)$ does not depend on the sequence of $\Le$-moves.
\end{enumerate}
\end{lemma}

Now, consider $x \in {^K W}$ and $v \in W_{max}^K$ such that $\ell(xv)=\ell(x)+\ell(v)$. Recall from \cref{prop:lengthsadd} that $\pathne{x([k])}$ lies above $\pathsw{v^{-1}([k])}$. Let $\lambda_x$ and $\lambda_v$ be the partitions above $\pathne{x([k])}$ and $\pathsw{v^{-1}([k])}$, respectively. 

\begin{definition}
Let $O_{x, v}$ be the $\oplus$-diagram of shape $\lambda_v$ with the boxes in $\lambda_x$ filled with $+$'s and all other boxes filled with $0$'s (see \cref{fig:leificationex}).
\end{definition}

\begin{proposition} \label{prop:lerichardson} $M(O_{x, v})$, the {\Le}-ification of $O_{x,v}$, is the \Le-diagram of $\pi_k(\mathcal{R}_{v, xv})$.
\end{proposition}

%

\begin{remark}
	\cref{prop:lerichardson}, together with the definition of $O_{x,v}$ (which is determined by 
	two noncrossing lattice paths in a rectangle, i.e. a skew Young diagram),
	is the reason that we refer to these positroid varieties as 
	\emph{skew Schubert varieties}.
\end{remark}

\ytableausetup{boxsize=normal}

\begin{figure}
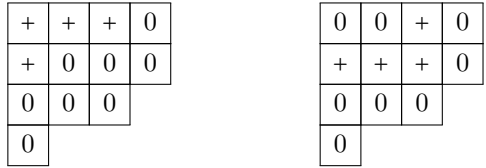

\centering
\begin{tabular}{l c r}
$\begin{ytableau}
+ & + & + & 0\\
+ & 0 & 0& 0\\
0 & 0 & 0\\
0
\end{ytableau}$ &\hspace{.5in} &
$\begin{ytableau}
0 & 0 & + & 0\\
+ & + & +& 0\\
0 & 0 & 0\\
0
\end{ytableau}$
\end{tabular}

\caption{\label{fig:leificationex} On the left, $O_{x,v}$ and on the right, $M(O_{x, v})$ for $x=(1, 2, 4, 7, 3, 5, 6, 8)$ and $v= (4, 3, 8, 2, 7, 6, 1, 5)$.}
\end{figure}

\begin{proof}
Let $M:=M(O_{x, v})$. Note that $\permne{\lambda_v}$ (that is, the Grassmannian permutation of type $(k, n)$ which maps $[k]$ to $\vertne{\lambda_v}$) is equal to $w_0 v^{-1}$. The reading word of $O_{x, v}$, and thus of $M$, is $w_0 v^{-1}x^{-1}$; this follows from the fact that there is a reading order for $\lambda_v$ which reads the boxes of $\lambda_v \setminus \lambda_x$ before the boxes of $\lambda_x$. So by \cref{prop:lebiject}, $M$ corresponds to $\pi_k(\mathcal{R}_{v, xv})$.
\end{proof}

\section{A cluster structure not realizable by generalized plabic graphs}
\label{sec:nonrealizable}

If $\pi_k(\mathcal{R}_{v, w})$ is not a skew Schubert variety, 
it is in general impossible to realize the seeds
from Leclerc's construction (\cref{Lec-seed}) as 
labeled quivers coming from (generalized) plabic graph. Indeed, this can fail even in $Gr_{2, 5}$. Before giving an example, we briefly review Leclerc's construction for the pair $(v, \mathbf{w})$, where $v \in W^{K}_{\max}$, $v<w$ and $\textbf{w}$ is a reduced expression for $w$.

\begin{defn} \label{def:posdistsubexpression} Let $v \leq w$ be permutations and $\textbf{w}=s_{i_t} \cdots s_{i_1}$ a reduced expression for $w$. The \emph{positive distinguished subexpression} for $v$ in $\textbf{w}$ is a reduced expression $\textbf{v}=v_t \dots v_1$ where $v_j \in \{s_{i_j}, e\}$. We give $\textbf{v}$ in terms of the products $v_{(j)}:=v_j \dots v_2 v_1$. We set $v_{(0)}=e$ and

\[ v_{(j)}=\left\{
\begin{array}{ll}
	s_{i_j}v_{(j-1)} & \text{ if } vv_{(j)}^{-1}s_{i_j}<vv_{(j)}^{-1} \\
v_{(j-1)} & \text{ otherwise}.
\end{array}
\right.
\]

In other words, the positive distinguished subexpression for $v$ is the rightmost subexpression for $v$ in $\textbf{w}$, working from right to left.

\end{defn}

Let $\mathbf{v} $ be the positive distinguished subexpression for $v$ in $\textbf{w}=s_{i_t}\cdots s_{i_2}s_{i_1}$. Let $w_{(j)}=s_{i_j}\cdots s_{i_2}s_{i_1}$ for $1\leq j \leq t$ and let $v_{(j)}=v_j \cdots v_1$ be as in the above definition. Let $J \subset \{1, \dots, t\}$ be the collection of indices $j$ such that $v_j=e$. According to \cref{Lec-seed}, the cluster variables in the seed corresponding to $(v, \mathbf{w})$ are the distinct irreducible factors of 
$\prod_{j\in J} \Delta_{v^{-1}_{(j)}\{[i_j]\},w^{-1}_{(j)}\{[i_j]\}}$. 

\begin{example}
Consider $v=(2, 5, 1, 4, 3)$, $w=(5, 3, 4, 2, 1)$ and the following reduced expression $\mathbf{w}$ for $w$, where the positive distinguished subexpression for $v$ is in bold:

$$\mathbf{w}=s_1 s_2 \mathbf{s_1} \mathbf{s_3} s_2 \mathbf{s_4} \mathbf{s_3} \mathbf{s_2} s_1.
$$

Note that $w$ does not have a length-additive factorization ending in $v$.

If one computes the generalized minors
$\Delta_{v^{-1}_{(j)}([i_j]),w^{-1}_{(j)}([i_j])}$
coming from \cref{Lec-seed}, they are not all irreducible.
However, if we associate Pl\"ucker coordinates to the irreducible 
factors of these generalized minors (as in \cref{Lec-variables}),  we obtain
 $\Delta_{13}$, $\Delta_{23}, \Delta_{14}, \Delta_{45}, \Delta_{15}$. 

However, 
$\{13, 23, 14, 45, 15\}$ 
cannot be the set of face labels of a generalized plabic graph for $Gr_{2, 5}$. 
This comes from the fact that the number $2$ appears only once
among 
the set $\{13, 23, 14, 45, 15\}$. 
In more detail, 
 suppose $G$ were such a generalized plabic graph. 
 $G$ has no internal faces and no lollipops. Without loss of generality, $G$ is source-labeled. The face $f$ labeled $23$, 
 is adjacent to one other face, labeled $13$. So consider the trip $T$ beginning at $2$ and ending at $j$. We know that $f$ is the only face to the left of this trip, and that $T$ must pass through vertices of degree 2 only. Then the trip beginning at $j$ is again $T$, traveled in the opposite direction. Thus, $j$ must be in the label of every face besides $f$, a contradiction.

On the other hand, $\{13, 23, 14, 45, 15\}$ is a \emph{subset} of the face labels of a plabic graph for the top cell in $Gr_{2, 5}$, since it is a weakly separated collection. Further, variables in the rectangles seed for the skew-Schubert varieties is always a subset of the face labels of a plabic graph $G$ for the top cell (and the quiver for the rectangles seed is obtained from $Q(G)$ by deleting some vertices and freezing others). One might ask if the seeds given in Leclerc's construction can always be obtained from a plabic graph for the top cell in this way. The following example will show that this is not the case.
\end{example}

\begin{example}
Consider $v=(3, 2, 7, 6, 1, 5, 4)$, $w=(7, 6, 4, 2, 5, 3, 1)$, and the following reduced expression $\mathbf{w}$ for $w$, where the positive distinguished subexpression for $v$ is in bold:
 $$ \mathbf{w}=\mathbf{s_1}s_2s_3\mathbf{s_2}s_1\mathbf{s_4 s_5 s_4 s_3} s_2 \mathbf{s_6 s_5 s_4 s_3} s_2 \mathbf{s_1}s_5 s_2
 $$

The irreducible factors of the generalized minors $\Delta_{v^{-1}_{(j)}([i_j]),w^{-1}_{(j)}([i_j])}$ are $\Delta_{135}$,  $\Delta_{126}$, $\Delta_{235}$, $\Delta_{345}$, $\Delta_{145}$, $\Delta_{467},$ $\Delta_{127}$, and $\Delta_{125}$ (the first variable is mutable and the others are frozen). Note that $467$ and $235$ are not weakly separated, so this set of Pl\"{u}cker coordinates cannot be a subset of the face labels of a plabic graph for the top cell.
\end{example}


\bibliographystyle{alpha}
\bibliography{bibliography}

\end{document}